\documentclass[11pt]{article}

\usepackage[T1]{fontenc}
\usepackage[margin=1in]{geometry}
\usepackage{amsmath,amssymb,amsfonts,mathtools}
\usepackage{xcolor}
\usepackage{hyperref}
\usepackage{amsthm}
\usepackage{aliascnt}
\usepackage{enumitem}
\usepackage[capitalize,noabbrev]{cleveref}
\usepackage{dsfont}
\hypersetup{
  hidelinks,
  pdftitle={Repeated Averaging on Random $d$-Regular Graphs},
  pdfauthor={Dong Yao and Lingfu Zhang}
}
\usepackage{etoolbox}
\apptocmd{\thebibliography}{%
  \setlength{\itemsep}{2pt}%
  \setlength{\parskip}{0pt}%
}{}{}

\newtheorem{theorem}{Theorem}[section]
\newaliascnt{prop}{theorem}
\newtheorem{prop}[prop]{Proposition}
\aliascntresetthe{prop}
\crefname{prop}{proposition}{propositions}
\Crefname{prop}{Proposition}{Propositions}
\newaliascnt{lemma}{theorem}
\newtheorem{lemma}[lemma]{Lemma}
\aliascntresetthe{lemma}
\newaliascnt{corollary}{theorem}

\aliascntresetthe{corollary}
\newaliascnt{assumption}{theorem}

\aliascntresetthe{assumption}
\newaliascnt{definition}{theorem}
\newtheorem{definition}[definition]{Definition}
\aliascntresetthe{definition}
\newaliascnt{remark}{theorem}

\aliascntresetthe{remark}

\newcommand{\E}{\mathbb E}
\newcommand{\Pp}{\mathbb P}
\newcommand{\R}{\mathbb R}
\newcommand{\N}{\mathbb N}
\renewcommand{\P}{\mathbb{P}}
\newcommand{\Z}{\mathbb Z}
\newcommand{\T}{\mathbb T}

\newcommand{\cE}{\mathcal E}
\newcommand{\cF}{\mathcal F}

\renewcommand{\hat}{\widehat}

\newcommand{\cP}{\mathcal P}
\newcommand{\one}{\mathbf 1}
\newcommand{\Poi}{\operatorname{Poi}}
\newcommand{\Var}{\operatorname{Var}}
\newcommand{\dist}{\operatorname{dist}}
\newcommand{\tx}{\mathsf{tx}}
\newcommand{\dd}{\mathrm d}
\newcommand{\abs}[1]{\left|#1\right|}
\newcommand{\norm}[1]{\left\|#1\right\|}
\newcommand{\cvp}{\xrightarrow{\Pp}}
\newcommand{\law}{\xrightarrow{\mathrm d}}

\title{Repeated averaging on random $d$-regular graphs}
\author{
  Dong Yao
  \thanks{School of Mathematics and Statistics/RIMS, Jiangsu Normal University, Xuzhou, Jiangsu, China.
  E-mail: dongyao@jsnu.edu.cn}
  \and
  Lingfu Zhang
  \thanks{\protect\raggedright Division of Physics, Mathematics and Astronomy, Caltech,
  Pasadena, CA, USA. E-mail: lingfuz@caltech.edu}
}
\date{November 18, 2026}

\begin{document}
\maketitle

\begin{abstract}
Repeated averaging on graphs is a natural prototype of local
stochastic exchange, modeling consensus formation, information
spreading, and wealth redistribution.
Its convergence has attracted longstanding interest, from
Bourgain's complete-graph $L^1$ cutoff question in the 1980s (answered in recent years
by Chatterjee, Diaconis, Sly, and Zhang)
to the systematic study of Aldous and Lanoue.
Towards the still outstanding and widely asked $L^1$ cutoff question on general graphs,
we establish a criterion for $L^1$ cutoff based on size-biased sampling.
We then prove $L^1$ cutoff with a Gaussian profile on random
$d$-regular graphs, for any $d\ge 3$ and uniformly over starting vertices.
Remarkably, although the expected mass profile agrees exactly
with the corresponding random walk distribution, cutoff occurs
strictly later than both the random walk cutoff time and the
universal lower bound attained on complete graphs,
revealing an unexpected entropy gap.
Our methods are robust, involving ideas of polymer representations, random geometry, universal-cover lifting, and adaptive revealing. We expect them to be applicable to repeated averaging and similar interacting systems with random exchange rules on a wide range of graphs.
\end{abstract}
\setcounter{tocdepth}{1}
\tableofcontents

\section{Introduction}

Consider the repeated averaging process on a finite connected graph
$G=(V,E)$: a real number $X(v)$ is assigned to each vertex $v$, and at each step an edge $e=\{u,v\}$ is chosen uniformly
from $E$; the two endpoint values are both replaced by
$(X(u)+X(v))/2$.  
This is a canonical asynchronous consensus or gossip model: vertices
represent agents and each update is a local exchange of information or
opinion.  The model is also connected to quantum computing through an
analogy with random quantum circuits \cite{MGW};
see \cite{DeGroot,ChatterjeeSeneta,BoydGhoshPrabhakarShah,Shah,
OlshevskyTsitsiklis} for more background.

The mathematical study goes back to a question of
Bourgain from the 1980s, as recounted in \cite{CDSZ}. A systematic study was undertaken by Aldous and Lanoue \cite{AL}.  
So far, its mixing
behavior, scaling limits, and variants have attracted substantial
attention
\cite{MGW,QS,CQS,SauConcentration,Eide,CQS24,CQSMeanField,Spiro}.
A more detailed comparison with the literature is below in
\Cref{sec:context-related-work}.

Our primary interest is in the $L^1$ convergence of this process.
Note that every update is doubly stochastic, so total mass $\sum_{v\in V} X(v)$ is preserved.
Naturally, we assume each $X(v)\ge 0$ and the total mass is $1$. 
Then the limiting profile is $(\frac{1}{n},\ldots, \frac{1}{n})=:\frac{\one}{n}$, where $n=|V|$. 
The $L^1$ distance, i.e., $\sum_{v\in V} |X(v)-1/n|$, lies in $[0,2]$, and is non-increasing at each update.
It gives a canonical linear measure of the failure of consensus.
For example, when repeated averaging is taken as a model of wealth redistribution, the $L^1$ distance is twice the Hoover index \cite{Hoover1936}.  

Mathematically, the $L^1$ decay is of particular interest, as it reveals a sharp phase transition. 
Typically, one starts with $X(v)=1$ for one vertex $v$, and all other numbers equal to zero; by linearity of this model, such initial profiles are the extremal ones that can generate all other initial profiles, and are the slowest in terms of $L^1$ decay (see e.g., \cite[Section~2.3, proof of Theorem~1.3]{CDSZ} and \cite[Theorem~2]{MGW}).
On a range of graphs, the $L^1$ distance when starting from such a `one-point' initial profile first stays near $2$, and then falls to near $0$ in a window negligible relative to the waiting time.
Adapting a convention from random walks, this sharp transition is referred to as \emph{cutoff}.
Determining whether this transition occurs, and identifying its
location, window size, and profile of decay within the window, is a central challenge. So far, it has been achieved for a few particular well-structured graphs \cite{CDSZ,CQS}. 

Another reason why the $L^1$ decay and cutoff are of interest is the exact connection with ordinary random walk.  
Consider the continuous-time setting where each edge updates according to its own independent Poisson clock.
Let $X_t(u)$ denote the number at a vertex $u$ at time $t$. If initially $X_0(u)=\mathds{1}[u=v]$ for some vertex $v$, then 
\begin{equation}   \label{eq:Xtuv}
    \E X_t(u) = p_t(v, u),
\end{equation}
where $p_t$ is the transition probability of the underlying continuous-time random walk (under a corresponding rate);
see e.g., \cite[Lemma~1]{AL}.
Therefore, the repeated averaging process can be viewed as random walk plus extra randomness (and we will repeatedly exploit this fact in this paper).
Random walk $L^1$ decay, i.e., total-variation mixing, also has a central cutoff phenomenon, which has been studied extensively since
the foundational work on card shuffling, through landmark results on
sparse graphs, to recent structural criteria
\cite{DiaconisShahshahani,AldousDiaconis,DiaconisCutoff,
LS,LubetzkyPeres,
SalezNonnegativeCurvature,PedrottiSalezLocalProduct}.

We study $L^1$ mixing on sparse expander graphs, by which we mean
bounded-degree graphs with a uniform random walk spectral gap
(equivalently, in the regular setting, uniform edge expansion).
Such graphs are of particular interest: for random walk, expansion forces mixing on the logarithmic scale, and is closely related to cutoff.
More precisely, although expansion does not by itself
force cutoff (see e.g., examples in \cite{LubetzkySlyExplicitExpanders}),  random walk cutoff is shown for many canonical sparse expanders, such as random
$d$-regular graphs \cite{LS} and all Ramanujan graphs
\cite{LubetzkyPeres}. 
There is an intriguing conjecture that every vertex-transitive
expander exhibits cutoff, which remains open \cite[Conjecture~1.7]{SalezVarentropyExpanders}.  
For a bounded-degree expander sequence $(G_n)$, Salez proved that the
associated continuous-time simple random walks exhibit total-variation 
cutoff if and only if a varentropy bound is satisfied \cite{SalezVarentropyExpanders}.

For repeated averaging, the $L^1$-cutoff problem on sparse expanders, in particular random $d$-regular graphs, remained largely open. It has been broadly discussed and asked; see e.g., \cite[Proposition~5 and Section~1.3]{MGW} and \cite[Section~1.2.3]{CQSRandomRegular}.
More concretely, here are the questions:
\begin{enumerate}  
\item Does repeated averaging exhibit $L^1$ cutoff? 
\item If cutoff occurs, at what time does it happen?
\item If cutoff occurs, what are its window size and profile?   
\end{enumerate}     
For many sparse expanders, the associated random walk has cutoff at
its entropic time, with square-root (of mixing time) window size and a Gaussian profile.  
It is natural to guess repeated averaging has cutoff at the same time with similar window size and profile.  
In fact, \eqref{eq:Xtuv} implies a one-sided bound: the repeated averaging cutoff time should be bounded below by the corresponding random walk cutoff time.

The same cutoff time prediction is true on the hypercube \cite{CQS}, but fails on
the complete graph \cite{CDSZ}.  More generally,
\cite[Theorem~1]{MGW} proves a universal lower bound for cutoff time, for every connected graph (with a given number of vertices), so a random walk prediction below
this fragmentation scale cannot be correct.  For random $d$-regular
graphs, \cite[Section~1.2.3]{CQSRandomRegular} compares the random walk
and entropy lower bounds and observes that the latter is stronger for
$d>7$, suggesting a possible finite-degree change of mechanism. 

We develop a toolkit for analyzing the repeated averaging process on sparse expanders.
In this paper, we settle the three outstanding questions above on uniformly random
$d$-regular graphs, the canonical random sparse expanders.  We prove
$L^1$ cutoff for every $d\ge3$.  
The cutoff time is shown to be at the entropic time of a typical
quenched mass on the infinite $d$-regular tree. 
The quenched mass has a decay rate strictly smaller than the random walk transition probability decay rate; therefore,
the cutoff time is strictly larger than the random walk cutoff time by some multiplicative constant. 
We also prove that the cutoff time is strictly larger than the universal lower bound from \cite{MGW} for each fixed $d$, and  approaches that bound as $d\to \infty$. 
Moreover, we show that the cutoff window has square-root window size with a Gaussian profile.
The techniques developed here can potentially be used to study repeated averaging and other related stochastic dynamics on more general graphs (see \Cref{ssec:subseq} below).

We next state the precise results of this paper.
We will always work in the continuous-time setting: there is an independent Poisson clock at each edge; and whenever it rings, the two numbers incident to the edge get averaged.
The discrete-time version then follows from a simple de-Poissonization.

\subsection{Random $d$-regular graphs}
Unless otherwise stated, on a $d$-regular graph, each edge
rings at rate $1/d$, so that the total ringing rate is $n/2$.

\begin{theorem}
\label{thm:main-rdreg}
Fix $d\ge3$, and let $n\to\infty$ through integers for which $nd$ is even. 
Let $G_n=(V_n, E_n)$ be a uniformly random simple connected $d$-regular graph on $n$ vertices. 
For each vertex $v$, let $X_t^{(n,v)}$ denote the vector of the numbers at time $t$, for the repeated averaging starting from $X_0^{(n,v)}(u)=\mathds{1}[u=v]$.
Then there are constants $\kappa_*, \sigma>0$, depending only on $d$, such that for any $a\in\R$ and $\epsilon>0$, 
\[
  \max_{v\in V_n}
  \P\Big[
    \abs{
      \norm{X_{\tau_n(a)}^{(n,v)}-\one/n}_1
      -2\Phi(-a)
    }>\epsilon \Big| G_n
  \Big]
  \to 0,
\]
in probability.
Here $\Phi$ is the CDF of the standard normal distribution, 
and
\begin{equation}\label{eq:profile-time}
  \tau_n(a)
  :=
  \frac{\log(n)}{\kappa_*}
  +
  \frac{a\sigma}{\kappa_*^{3/2}}\sqrt{\log(n)}.
\end{equation}
\end{theorem}
In words, this is a quenched result, in the sense of conditioning on the sequence of graphs. It is plausible to refine our analysis to upgrade the convergence in probability to almost sure convergence. We do not pursue this technical strengthening here.

The constants $\kappa_*$ and $\sigma$ can be determined in terms of repeated averaging on the infinite $d$-regular tree $\T_d$ (see \Cref{ssec:tree}). 
We also have the following bounds on the rate $\kappa_*$.
\begin{theorem}   \label{thm:kappasbounds}
We have $0<\kappa_*<\min\{\kappa_{\rm RW}, \log(2)\}$, where $\kappa_{\rm RW}:=\frac{d-2}{2d}\log(d-1)$ is the random walk entropy rate on $\T_d$. 
Moreover, $\kappa_*\to \log(2)$ as $d\to\infty$.
\end{theorem}

The universal lower bound on cutoff time in \cite{MGW} would translate to $\kappa_*\le \log(2)$. 
The corresponding random walk, on the other hand, has total-variation cutoff at $\frac{\log(n)}{\kappa_{\rm RW}}$, proved by Lubetzky and Sly \cite{LS}.
Therefore, as alluded to above, Theorem \ref{thm:kappasbounds} in particular implies that the cutoff time $\frac{\log(n)}{\kappa_*}$ is larger than either the universal lower bound or the random walk cutoff time. This is beyond some natural speculations, as in \cite[Section~1.2.3]{CQSRandomRegular}.

The strict inequality $\kappa_*<\kappa_{\rm RW}$ can be understood as a result of the extra layer of randomness on top of random walks and the non-trivial interaction between these two layers of randomness. This will be more evident from the size-biased sampling perspective, to be discussed  in \Cref{ssec:tree}.

The $\kappa_*<\log(2)$ part actually follows from a more general bound below, applicable to any regular graph.
\begin{theorem}
\label{thm:intro-general-lower}
For each $n\in\N$, let $G_n=(V_n, E_n)$ be a connected $d_n$-regular graph with $|V_n|=n$. For repeated averaging with  rate-$1/d_n$ edge clocks, and $0<\epsilon<1$, define
\begin{equation}\label{eq:tn_epsilon}
      t_{n,\epsilon}
  :=
  \inf\left\{
    t\ge0:
    \max_{v\in V_n}
    \E\norm{X_t^{(n,v)}-\one/n}_1\le2\epsilon
  \right\},
\end{equation}
where $X_t^{(n,v)}$ is defined in the same way as in \Cref{thm:main-rdreg}.
Then we have
\begin{equation}\label{eq:general_mixing_lower}
      \liminf_{n\to\infty}
  \frac{t_{n,\epsilon}}{\log_2(n)}\ge1.
\end{equation}
If $d_n\equiv d$ is fixed, there is $\underline c_d>1$, depending only
on $d$, such that
\begin{equation}\label{eq:bounded_mixing_lower}
    \liminf_{n\to\infty}
  \frac{t_{n,\epsilon}}{\log_2(n)}\ge\underline c_d.  
\end{equation}
\end{theorem}

For regular graphs, the first inequality sharpens the general
entropy lower bound of \cite{MGW} by removing its fixed-accuracy loss $\epsilon$ in
the leading coefficient.  The strict fixed-degree inequality is a further
step.
More precisely, the $\log(2)$ bound stems from the fact that each edge ring can reduce the number at either endpoint of the ringing edge by at most a half; the bounded-degree condition then ensures that, at least for a portion of the steps, the decrease is less than a half.
The $\kappa_*\to \log(2)$ part of \Cref{thm:kappasbounds} then precisely says that, for random $d$-regular graphs, such slowdown disappears in the $d\to\infty$ limit.

\subsection{A general criterion for cutoff}   \label{ssec:crit}
The mechanism behind \Cref{thm:main-rdreg} is a general entropic
criterion for repeated-averaging cutoff.  Let
$G_n=(V_n,E_n)$ be a sequence of connected graphs with $|V_n|=n$.
Let 
$\lambda_{2,n}$ be the second-smallest eigenvalue of the unnormalized
combinatorial Laplacian $L_n$ (where
$L_nf(v)=\sum_{u\sim v}(f(v)-f(u))$).
Since $G_n$ is not necessarily a regular graph, we let each edge ring at rate $1$ in repeated averaging.  
Let $X_0^{(n)}:V_n\to \R_{\ge 0}$
be a possibly random probability vector on $V_n$, independent of the Poisson clocks, satisfying $\sum_{v\in V_n}X_0^{(n)}(v)=1$.
Let $X_t^{(n)}$ be the averaging profile at any time $t>0$.

There is a natural random walk coupled to the repeated averaging.
To begin with, conditional on $X_0^{(n)}$, choose $W_0^{(n)}$ with law $X_0^{(n)}$.
Whenever an edge incident to $W_t^{(n)}$ rings, independently let the
walk stay at its current vertex or cross that edge, each with
probability $1/2$.  Then it follows that
\begin{equation}   \label{eq:WXsample}
  \P\big[ W_t^{(n)}=v \big| X_t^{(n)}\big]
  =X_t^{(n)}(v),
  \qquad \forall v\in V_n.
\end{equation}
We write
\[
  X_t^{*,(n)}:=X_t^{(n)}(W_t^{(n)}).
\]
In words, $X_t^{*,(n)}$ is the number seen by this walk. We can also think of $X_t^{*,(n)}$ as a size-biased sampling from the coordinates of $X_t^{(n)}$, or as
the number containing a typical unit of the mass (rather
than at a uniformly sampled vertex).  
\begin{theorem}
\label{thm:intro-entropic-criterion}
If  there is a positive sequence $\{t_n,n\geq 1\}$ such that
\begin{equation}\label{eq:hypo_general_cri}
      \liminf_{n\to\infty}
  \frac{\lambda_{2,n}t_n}{\log(n)}>0,\qquad
  nX_{t_n}^{*,(n)}\cvp\infty,
  \qquad
  n^{1-\eta}X_{t_n}^{*,(n)}\cvp0
  \quad\text{for every }\eta>0,
\end{equation}
then there is a sequence $s_n=o(t_n)$ such that
\begin{equation}   \label{eq:Xt20}
  \norm{X_{t_n}^{(n)}-\one/n}_1\cvp2,\qquad
  \norm{X_{t_n+s_n}^{(n)}-\one/n}_1\cvp0.
\end{equation}
\end{theorem}

The two assumptions on $X_{t_n}^{*,(n)}$ say that a typical unit of
mass lies in a coordinate of size slightly above order $1/n$.  Thus almost
all mass is carried by $o(n)$ coordinates at time $t_n$, so the
$L^1$ distance at time $t_n$ is still nearly $2$.  The spectral-gap assumption will then smooth the profile and bring the $L^1$ distance to nearly $0$, in $o(t_n)$ additional time.

We state the criterion here to keep the idea of the mechanism transparent. 
It can be upgraded, so that the next-order fluctuations of $X_t^{*,(n)}$ determine cutoff profiles; see \Cref{prop:threshold-L1} below. 

The criterion can also be viewed as entropy concentration.  In fact, the assumptions on $X_{t_n}^{*,(n)}$ can be understood as a concentration of the random variable $-\log X_{t_n}^{*,(n)}$. By \eqref{eq:WXsample}, we have
\[
  \E\big[-\log X_t^{*,(n)}\big| X_t^{(n)}\big]
  =-\sum_{v\in V_n}X_t^{(n)}(v)\log X_t^{(n)}(v).
\]
The right-hand side is the entropy of the profile at time $t$.
Then at $t=t_n$ it also concentrates around $\log(n)$.
This is in the spirit of Salez's varentropy criterion for random walk cutoff on expanders
\cite{SalezVarentropyExpanders}.  The essential difference is the
extra layer of randomness here: once a graph is fixed, a random walk
heat kernel is deterministic, whereas $X_t^{(n)}$ remains random
through the graphical environment and the size-biased walk samples
from that same environment.  
To control this common-environment
randomness, we need to consider two walks at the profile level, as in \Cref{prop:threshold-L1} below.

\subsection{Infinite $d$-regular tree}    \label{ssec:tree}

The criterion above reduces the cutoff problem to understanding the scaling limit of the size-biased mass: a law of large numbers locates cutoff, while a central limit theorem gives the
next-order window and Gaussian profile. 
To achieve this for random $d$-regular graphs, we exploit the fact that such graphs are (with high probability) locally tree-like, as established in \cite{LS}.
We first
establish the required size-biased mass limiting theorem on the infinite $d$-regular tree, and then transfer it to the finite graph using the locally tree-like results.

Fix $d\ge3$, and let $\T=\T_d$ be the infinite $d$-regular tree rooted at
$o$. 
Consider repeated averaging on $\T$ with edge rate $1/d$, and let $X_t^{\T}$ be the time $t$ averaging profile started from $X_0^{\T}(v)=\mathds{1}[v=o]$.
Let $(W_t^{\T})_{t\ge 0}$ be the associated size-biased walk.
More precisely, take $W_0^{\T}=o$; whenever an edge incident to the walk rings, the walk jumps across that edge with probability $1/2$.
Denote $X_t^{*,\T}:=X_t^{\T}(W_t^{\T})$.

\begin{theorem}
\label{thm:intro-tree-clt}
As $t\to\infty$, $\frac{-\log X_t^{*,\T}-\kappa_*t}{\sigma\sqrt{t}}$ converges in distribution to the standard normal distribution.
Here $\kappa_*$ and $\sigma$ are the same as those constants in \Cref{thm:main-rdreg}.
\end{theorem}
In other words, we can characterize the constants $\kappa_*$ and $\sigma$ from \Cref{thm:main-rdreg} as the mean and standard-deviation constants in the central limit theorem of $-\log X_t^{*,\T}$.

To get the finite random $d$-regular cutoff result \Cref{thm:main-rdreg}, we will also prove a two-walk version of \Cref{thm:intro-tree-clt} (see \Cref{thm:two-walker-clt} below).
The one-walk limit determines the
asymptotic mass above the threshold $1/n$ in an expected sense, while the two-walk limit
further upgrades it to convergence in probability.  The tree-to-finite comparison Proposition \ref{prop:finite-cover} then transfers both limits to a random $d$-regular graph;
\Cref{prop:threshold-L1} then gives the Gaussian profile.

The strict inequality $\kappa_*<\kappa_{\rm RW}$ illustrates a logarithmic correlation between the profile $X_t^{\T}$ and the walk $W_t^\T$,
which further implies a fractal behavior of the profile.  Indeed, together with a law of large numbers for random walks on trees, we get
\[
  \frac1t\log
  \frac{X_t^{*,\T}}{p_t^{\T}(o,W_t^{\T})}
  \cvp
  \kappa_{\rm RW}-\kappa_*>0,
\]
where $p_t^{\T}$ is the transition probability of the continuous time random walk on $\T$ (with jump rate across any edge being $\frac{1}{2d}$).
Thus, at the same endpoint, the mass sampled from the realized profile $X_t^{*,\T}$
is exponentially larger than the annealed random walk transition
probability.

This has an interesting geometric implication for the profile. At time
$t$, with high probability the walk has radius $\frac{(d-2)t}{2d}+o(t)$, and the corresponding
sphere has $\exp(\kappa_{\rm RW}t+o(t))$ vertices.  A profile
spread uniformly over these outward directions would give mass
$\exp(-\kappa_{\rm RW}t+o(t))$ to a typical endpoint.  However, a
size-biased sampling gives $\exp(-\kappa_*t+o(t))$.  This indicates that $X_t^{\T}$ is fragmented among an exponentially smaller
random collection of branches or rays, rather than distributed
uniformly among all rays of $\T$.

\subsection{Further proof novelties}   \label{ssec:fpn}
Other than using the coupled walk perspective to capture cutoff, and exploiting local tree-like structures, there are some other key aspects of our proofs at a more technical level that we explain now.

\paragraph{Strict slowdown via adaptive revealing.}
The strict inequality $\kappa_*<\kappa_{\rm RW}$ is not only an interesting result by itself, but also plays crucial roles in proving the central limit theorem on the infinite $d$-regular tree $\T=\T_d$, and the tree-to-finite comparison. We next outline the proof of this inequality.

A priori, $\kappa_*$ is defined as the $t\to\infty$ limit of
$\frac{-\E[\log X_t^{*,\T}]}{t}$.  The existence of this limit
follows from a straightforward subadditivity argument; see
\Cref{prop:kappa-star} below.
Then \eqref{eq:Xtuv} and entropy concavity would lead to $\kappa_*\le\kappa_{\rm RW}$.
To prove strictness, the new idea is to reveal the Poisson clocks gradually.
Namely, as before we let $p_t^\T$ be the transition probability of the continuous-time random walk on $\T$, with edge jump rate $\frac{1}{2d}$.
Fix $t>0$, let $\cF_s$ be the sigma-algebra generated by the Poisson clock rings
revealed through time $s<t$.
We denote (with $V^\T$ being the vertex set of $\T$)
\[
  M_s(v):=\E[X_t^\T(v)\mid\cF_s]
  =\sum_{u\in V^\T}X_s^\T(u)p_{t-s}^\T(u,v).
\]
The process $M_s(v)$ interpolates between the annealed transition probability $p_t^\T(o,v)$, and
the quenched mass $X_t^\T(v)$.  With
$H(q):=-\sum_{v\in V^\T} q(v)\log q(v)$ for any probability vector $q(\cdot)$, differentiation along the reveal time $s$
gives
\[
  H(p_t^\T(o,\cdot))-\E H(X_t^\T)
  =\int_0^t - \frac{\dd}{\dd s}
    \E[H(M_s)]\,\dd s.
\]
We analyze the derivative in the integrand, by factoring the effect of a ringing edge through a mass gradient and a random walk transition probability gradient. 
These end up showing that at least order $t^{-1/2}$ entropy is produced per unit time around time $s$ for a positive proportion of $s\in [0,t]$, and hence
\[
  H(p_t^\T(o,\cdot))-\E H(X_t^\T)\ge c\sqrt t.
\]
Finally, applying the infimum formula for the subadditive rate at one sufficiently
large time $t$ turns this finite-time deficit into
$\kappa_*<\kappa_{\rm RW}$. 

\paragraph{Central limit theorem and comparison: via polymer and random geometry.}
Take any graph $G=(V, E)$ with independent Poisson clocks on the edges.
Consider the set of all space-time paths $\gamma:[s, t]\to V$, which can only jump across an edge when its Poisson clock rings (and it can also choose to stay when the clock at an incident edge rings). 
Any such  path receives weight
$2^{-N(\gamma)}$, where $N(\gamma)$ counts rings incident to the path.
With a fixed time interval $[s, t]$ and starting point $\gamma(s)$, this gives a path measure, or \textit{polymer measure}. 
This precisely gives the coupled random walk, conditional on all the Poisson clocks. (See \Cref{sec:setup} below for more details.)

To get the central limit theorem on $\T$, essentially one needs some mixing of $-\log X_t^{*,\T}$ in time $t$.
More precisely, we need to get a fast enough decay of memory for $\log\frac{X_t^{*,\T}}{X_{t-1}^{*,\T}}$ given $\log\frac{X_s^{*,\T}}{X_{s-1}^{*,\T}}$, as $t-s \to \infty$.
This can be encoded as comparing four random paths under some polymer measures, with different starting points and time intervals, but the same Poisson clocks.
Using a path switching argument, the mixing can be encoded as the probability that two random paths stay disjoint. 
Conditional on one path, the other one avoiding its trajectory would have a probability of roughly $e^{-\kappa_{\rm RW}(t-s)}$, while the renormalization factor would be around $e^{-\kappa_*(t-s)}$.
Here by exploiting the inequality that $\kappa_{\rm RW} > \kappa_*$, we have that the mixing is exponential in $t-s$.

A caveat here is that mixing is not sufficient for a non-degenerate central limit theorem: one also needs that the scaled variance $\sigma^2>0$.
We use a Mermin-Wagner type argument, via slightly changing the Poisson clock rates. 
This gives that the variance of $-\log X_t^{*,\T}$ is unbounded. 
Using the above exponential mixing, such unboundedness can be upgraded to at least linear growth; thereby  we conclude $\sigma^2>0$.

For the tree-to-finite comparison, essentially one lifts polymers on a finite $d$-regular graph, to polymers on the universal cover graph (which is $\T$).
Error terms in such a lifting can come from two sources: synchronization of Poisson clocks in the projection from the cover graph, and the lifting of paths with different winding numbers. 
Both error terms are controlled using geometric arguments on polymers, with tree-like estimates for random $d$-regular graphs from \cite{LS}.
The strict inequality $\kappa_{\rm RW} > \kappa_*$ provides us with an order $\log(n)$ time interval between the random walk cutoff time and the repeated averaging cutoff time, in which many of the geometric arguments are carried out.

\subsection{Robustness of the method and potential extensions}   \label{ssec:subseq}

As is evident so far, in proving the repeated averaging cutoff results on random $d$-regular graphs, we effectively have developed a toolkit. Many of the components are robust and can be applied to a wider range of graphs and processes.

Roughly speaking, our proofs indicate that, at least for repeated averaging cutoff, it suffices that the sequence of finite graphs satisfy three conditions: (1) a uniform spectral gap for the final smoothing, (2) local tree-likeness (at least for most vertices in neighborhoods with growing radius) for the decay of correlation, and (3) mesoscopic and global geometric information preventing polymer concentration due to graph structure at larger scales. 

In the absence of condition (1), cutoff may fail. For instance, Quattropani and Sau \cite{QS} proved absence of cutoff for repeated averaging on graphs with finite-dimensional geometries (in the sense of a Nash inequality), including finite-dimensional tori as a particular example. 
The condition (2) may be replaced by the Benjamini-Schramm convergence to an infinite `homogeneous graph', such as a unimodular Galton-Watson tree or an infinite vertex-transitive graph. In the later case, the tree structure requirement may not be necessary, provided that the random walk and repeated averaging entropy have different  rates of growth.
The condition (3) can often be supplied by estimates on the spread of random walk transition  probabilities, or random walk mixing time information. 

Another natural (and orthogonal to the above) direction is to extend our results from repeated averaging to more general random exchange rules.
For example, Caputo, Quattropani, and Sau study a family of mean-field exchange models in which each pairwise update is in terms of a random fraction sampled from a prescribed distribution \cite{CQSMeanField}.
A key example here is Kac's walk, which is more commonly stated as a random two-coordinate rotation at each step \cite{KacKineticTheory}.
In \cite{CQSMeanField}, these models are considered on the complete graph, with cutoff established.
For many of these models, the limiting stationary state is not the deterministic vector $\one/n$ but a nontrivial stationary law.
Accordingly, their cutoff results are in terms of Wasserstein-$1$ distance.
Very recently, Jain and Mizgerd established total-variation cutoff for Kac's walk started from a coordinate vector \cite{JainMizgerdKacCutoff}.
We hope that combining our techniques with those of \cite{CQSMeanField} could yield Wasserstein-$1$ cutoff for these general exchange models on various sparse expander graphs.
It would also be intriguing to determine whether our toolkit can help establish total-variation cutoff for Kac's walk with coordinate-pair updates restricted to the edges of a sparse expander.

\subsection{Further discussions on related work}
\label{sec:context-related-work}

We conclude the introduction with a more detailed comparison with the
literature.  Some of these results were mentioned above; here we place
them in a broader context or provide more detailed discussions.

The systematic treatment of the repeated averaging process
by Aldous and Lanoue \cite{AL} established basic results
including spectral contraction, entropy bounds, and particle
duality.  The first $L^1$-cutoff theorem was proved by Chatterjee,
Diaconis, Sly, and the second-named author \cite{CDSZ}.  Under the
normalization of one averaging update per unit time, cutoff occurs at
$\frac{n}{2}\log_2(n)$, with a window of order
$n\sqrt{\log(n)}$ and a Gaussian profile.  Subsequently, Movassagh,
Szegedy, and Wang \cite{MGW} proved a universal $L^1$ mixing-time lower
bound and sharp-order $L^1$ estimates for bounded-degree expanders,
cycles, stars, and dumbbells, together with bounds in other $L^p$
norms.  Quattropani and Sau \cite{QS} treated finite-dimensional
geometries and proved cutoff for the binomial-splitting dual, while
Caputo, Quattropani, and Sau \cite{CQS} established cutoff on
hypercubes and complete bipartite graphs.  Spiro extended spectral
convergence estimates to hypergraph averaging \cite{Spiro}.
On the complete graph, Caputo, Quattropani, and Sau \cite{CQS24} studied
repeated block averaging, establishing sharp criteria for $L^1$ cutoff
and identifying the cutoff window and Gaussian profile when cutoff
occurs.  As noted
above, the more recent work \cite{CQSRandomRegular} studies
repeated averaging on random regular graphs and proves, from a typical
starting vertex, $L^2$ cutoff as well as $L^p$ pre-cutoff for
$1\le p<2$.

The repeated averaging process has also been studied in other contexts and aspects.  On discrete tori and $\mathbb Z^d$, results
include concentration, local smoothing, hydrodynamic limits, and
central limit and nonequilibrium fluctuation theorems
\cite{SauConcentration,Eide,NagahataYoshida,SauFluctuations}.  For
i.i.d.~initial profiles, recent works study convergence on infinite
bounded-degree graphs, poly-logarithmic consensus bounds on finite
graphs, and the dependence on moments of the initial law
\cite{GantertVilkas,ElboimPeresPeretz,ZuoMoment}.  On the complete graph,
empirical distributions and kinetic limits have been investigated in
\cite{CaoGini,CamposFrancoHeydenreichSchrocke}.  These works differ from
ours in their geometry, initial data, or asymptotic question: they
concern Euclidean or infinite-volume settings, typical random initial
conditions, or macroscopic evolution.

Finally, we compare our result with the random walk cutoff theorem of
Lubetzky and Sly \cite{LS}, which is both an important input to and a
source of inspiration for our work.  After Poissonizing their
discrete-time theorem to match our time-normalization,
the random walk total-variation cutoff occurs at $\log(n)/\kappa_{\rm RW}$, with a
window of order $\sqrt{\log(n)}$ and a Gaussian profile.  This
closely parallels \Cref{thm:main-rdreg}: at a high level, in both
settings observables sampled along a walk have fast temporal decorrelation, leading to Gaussian fluctuations.  The
key difference is precisely the extra layer of randomness from the Poisson clocks, which leads to a size-biased walk in  repeated averaging.

A useful heuristic analogy is the $1$D random walk versus the $(1+1)$D directed polymer model, where paths are reweighted by random
space-time disorder; see e.g., \cite{ZygourasDirectedPolymers}. 
At strong disorder the polymer model is expected to belong to the Kardar-Parisi-Zhang (KPZ) universality class and
exhibit 1:2:3 scaling, but the scaling convergence remains out
of reach, except for some special disorder environments
\cite{SeppalainenPolymerScaling,OConnellYorBurke}.
See the recent works \cite{DrillickParekhSauCriticalKPZ,BarraquandCasiniKMPKPZ} for more direct connections from the repeated averaging and related models to the KPZ universality class.

Likewise, on random regular graphs, repeated averaging is substantially
less explicit than random walk.  Although $\kappa_*$ and $\sigma$ can
in principle be approximated numerically using e.g., the infinite-tree
central limit theorem in \Cref{thm:intro-tree-clt}, we do not give
closed-form expressions for them, in contrast to the corresponding
random walk constants in \cite{LS}.  The proofs are
correspondingly more involved.  In \cite{LS}, local tree-likeness is exploited to reduce much of the analysis to
simple random walk on the infinite $d$-regular tree, where spherical
symmetry enables explicit transition probability calculations.  For repeated
averaging, the Poisson-clock environment makes the analogous
observables substantially more complicated.  Our random geometry
arguments instead reduce the required estimates to expectations of
polymer weights, which, after suitable averaging, can be expressed in
terms of transition probabilities for one or several coupled random
walks.

\paragraph{Organization of the remaining text}
In \Cref{sec:setup}, we develop the graphical construction and polymer
representation, introduce the size-biased walks, and establish the geometric
localization estimates used later.  \Cref{sec:general-criterion} proves the
general entropic criterion that converts threshold estimates for
size-biased weights into $L^1$ cutoff, in particular proving \Cref{thm:intro-entropic-criterion}.  The next three sections concern
the infinite tree: \Cref{sec:strictslow} constructs the entropy rate
$\kappa_*$ and proves its strict slowdown relative to random walk;
\Cref{sec:cltontrees} establishes exponential forgetting and the
one-walk central limit theorem, which is essentially \Cref{thm:intro-tree-clt}; and \Cref{sec:two-walkers} proves the
joint two-walk limit and positivity of the limiting variance.
\Cref{sec:cover} transfers these tree estimates to finite
$d$-regular graphs (satisfying certain conditions) through the universal cover, using technical geometric constructions and estimates on finite and infinite $d$-regular graphs.  \Cref{sec:profile}  verifies that the
required conditions hold for random $d$-regular graphs with high probability, and derives the
Gaussian profile uniformly over any starting vertex, thereby completing the proof of \Cref{thm:main-rdreg}.  Finally, \Cref{sec:general-lower} proves the universal lower
bound for regular graphs, i.e., \Cref{thm:intro-general-lower}, and the convergence of $\kappa_*$ to $\log(2)$ as the degree $d\to \infty$, completing the proof of \Cref{thm:kappasbounds}.

\paragraph{Acknowledgement}
DY is partially supported by the National Key 
R\&D Program of China (No. 2023YFA1010101) and an NSFC grant (No. 12571161).
LZ is partially supported by the NSF grant DMS-2505625, and a Sloan fellowship.
Part of the work was completed when LZ was visiting DY at Jiangsu Normal University in the summer of 2025, and LZ would like to thank DY and Jiangsu Normal University for their hospitality. 

\paragraph{Statement on AI}
All mathematical ideas and proofs were developed and written by the authors themselves. AI tools were used only for literature review and for correcting typos and grammatical errors.

\paragraph{Notations}
In the rest of this paper, we fix $d\ge 3$ unless otherwise stated, and typically take the limit of $n\to\infty$ (for finite graphs) or $t\to\infty$ (for infinite trees). 
We will use $C,c$ to denote large and small positive constants, which are independent of $n$ or $t$, but may depend on $d$ and other parameters explicitly held fixed, and their precise values may change from line to line.
For quantities $f$ and $g$ depending on $n$ or $t$,
we use
$f=O(g)$ and $f\lesssim g$ to denote
$|f|< C|g|$, and write
$f\asymp g$ when $c|g|< |f|< C|g|$. We also write $f=o(g)$ for
$f/g\to0$ as $n\to\infty$ or $t\to\infty$. 
For any $x,y \in \R$, we write $x\vee y=\max\{x, y\}$, and $x\wedge y = \min\{x, y\}$.
For two vertices $u, v$ in a graph, $u\sim v$ denotes that there is an edge $\{u, v\}$.

\section{Preliminaries on polymers in a Poisson field}
\label{sec:setup}

In this section, we give the setup of the space-time Poisson field on graphs and the directed polymers on it. We will also prove some polymer geometric estimates on trees, which will be used repeatedly throughout.

Take a connected graph $G=(V,E)$ with finite degree at each vertex, and $\rho>0$. Let $\Pi\subset E\times \R$ be a rate $\rho$ Poisson point process, i.e., each edge in $E$ has an independent Poisson clock  ringing at rate $\rho$. 
We start by defining the random transition kernel from $\Pi$.

\begin{definition} \label{def:compath}
We consider the family of all
paths \emph{compatible} with $\Pi$.
Namely, such a path is a c\`adl\`ag function $\gamma:[s,t]\to V$ for some $s<t$, and may jump
only across an edge at one of its ringing times: if
$\gamma(r)\ne\gamma(r-)$ for some $s<r\le t$, then
$(\{\gamma(r-),\gamma(r)\},r)\in\Pi$.
For any given $s\le t$ and $u,v\in V$, write
\[
  \mathcal C_{s,t}[u,v;\Pi]
  :=\{\gamma:[s,t]\to V\text{ is compatible with }\Pi,
       \gamma(s)=u,\ \gamma(t)=v\}.
\]
We call $\Pi$ \emph{locally finite} if
$\bigcup_{v\in V}\mathcal C_{s,t}[u,v;\Pi]$ is finite for every
$s\le t$ and $u\in V$.
We note that local finiteness holds almost surely for graphs with bounded degree.
\end{definition}

\begin{definition}
Given that $\Pi$ is locally finite, for a compatible path
$\gamma:[s,t]\to V$, its weight is defined to be
\[
  w(\gamma;\Pi):=2^{-\abs{\{(e,r)\in\Pi:s<r\le t,\ \gamma(r-)\in e\}}}.
\]
For $u, v\in V$, define the polymer weight from $(u, s)$ to $(v, t)$ by
\begin{equation}\label{eq:polymer-kernel}
  w_{s,t}(u, v;\Pi)
  :=
  \sum_{\gamma \in \mathcal C_{s,t}[u, v;\Pi]}
  w(\gamma;\Pi).
\end{equation}
\end{definition}
Whenever there is no confusion with the Poisson environment, we typically drop $\Pi$ from the notation above .

The polymer weight is precisely the random transition kernel for repeated averaging. 
Namely, for repeated averaging starting from $X_0: V\to [0,1]$, driven by the Poisson point process $\Pi$, the time-$t$ profile is
\begin{equation} \label{eq:XWrel}
X_t(v) = \sum_{u\in V} X_0(u) w_{0,t}(u, v).
\end{equation}
Moreover, the polymer weights are doubly stochastic and satisfy a composition law: for any $r<s<t$ and $u, v\in V$,
\begin{equation}\label{eq:chapman-kolmogorov}
\sum_{u'\in V} w_{r,t}(u',v)=\sum_{v'\in V} w_{r,t}(u,v')=1,\qquad
  w_{r,t}(u,v)
  =
  \sum_{u'\in V} w_{r,s}(u, u')w_{s,t}(u', v).
\end{equation}
There is also consistency for the path weights: for any $r<s<t$, and $\gamma: [r,s]\to V$ compatible with $\Pi$, we have
\[
w(\gamma)=\sum_{\gamma' \in \bigcup_{v\in V}\mathcal{C}_{r,t}[\gamma(r), v]:\gamma'|_{[r,s]}=\gamma} w(\gamma').
\]
From these properties, we can then define random walks.
\begin{definition}\label{def:random walk gene by polymer}
Given $\Pi$ that is locally finite, for any $u\in V$ and $s<t$, the \emph{random walk generated by $\Pi$ (in the time interval $[s,t]$ and starting from $u$)} is a random function $W:[s,t]\to V$, with $\P[W=\gamma | \Pi] = w(\gamma)$ for any $\gamma\in\bigcup_{v\in V}\mathcal C_{s,t}[u,v]$. We also write $W_r=W(r)$ for $r\in [s, t]$.

Moreover, by consistency of the path weights, we can send $t\to\infty$ and obtain a \emph{random walk generated by $\Pi$ (in the time interval $[s,\infty)$ and starting from $u$)}.
\end{definition}
As mentioned in the introduction, the law of such $W$ is effectively as follows. Conditional on $\Pi$, whenever a Poisson clock rings at an edge incident to the walk $W$, it jumps across the edge with probability $1/2$, independently of $\Pi$.
For any $r\in [s, t]$, we have 
\begin{equation}\label{eq:size-biased}
  \P[W_r=v\mid\Pi]=X_r(v),
\end{equation}
where $X_r$ is the repeated averaging profile at time $r$, started from $X_s(v)=\mathds{1}[v=u]$.
The annealed law of $W$ is that of a continuous-time random walk: since $\Pi$ has rate $\rho$, the walk $W$ starts from $u$ at time $s$, and jumps across each edge with rate $\rho/2$, independently.

We next provide an estimate on the localization of polymers. 
This will be the cornerstone of our later geometric arguments.
\begin{definition}\label{defn:splspt}
Suppose that $G$ is a tree, and $u=u_0,u_1,\ldots,u_k=v$ is the
shortest path from $u$ to $v$, and $s<t$.  The space-time line from
$(u,s)$ to $(v,t)$ is
\[
  \Gamma_{s,t}^{u,v}(r)
  :=u_{\lfloor (r-s)k/(t-s)\rfloor},
  \qquad s\le r\le t.
\]
For $x,y\ge0$, the space-time $x\times y$ tube around $\Gamma_{s,t}^{u,v}$ is defined by
\[
  \mathcal V_{s,t;x}(u,v;y)
  :=
  \Bigg\{
    (w,r)\in V\times[s,t]:
    \inf_{\substack{r'\in[s,t]\\ |r'-r|\le x}}
    \dist\bigl(w,\Gamma_{s,t}^{u,v}(r')\bigr)
    \le y
  \Bigg\}.
\]
We also define
\[
  \mathring{\mathcal C}_{s,t;x}[u,v;y]
  :=
  \left\{
    \gamma\in\mathcal C_{s,t}[u,v]:
    (\gamma(r),r)\in\mathcal V_{s,t;x}(u,v;y),
    \ \forall r\in[s,t]
  \right\},
\]
which is the collection of compatible paths contained in the tube.
\end{definition}

\begin{lemma}\label{lem:tube}
Let $G=\T_d$ and set the rate of $\Pi$ to $\rho=1/d$. 
Take $t>0$, and let $W$ be the walk
generated by $\Pi$ in the time interval $[0, t]$, starting from the root $o$.  Then for $1\le x\le t$ and $y\ge1$, we have
\begin{equation}\label{eq:wotin}
  \P[W|_{[0,t]} \not\in \mathring{\mathcal C}_{0,t;x}[o,W_t;2y] ]
< C e^{-cx^2/t}+Ct(d-1)^{-y/2},
\end{equation}
and
\begin{equation}\label{eq:tubeint}
      \E\int_0^t
  \mathds{1}[(W_r,r)\notin
    \mathcal V_{0,t;x}(o,W_t;2y)]
  \,\dd r <
  Cte^{-cx^2/t}+Ct(d-1)^{-y/2}.
\end{equation}
\end{lemma}

\begin{proof}
We start by proving \eqref{eq:wotin}.  Let $V_{r,t}$ be the last
common ancestor of $W_r$ and $W_t$ in the rooted tree $\T_d$, and put
\[
  R_r:=\dist(o,W_r),
  \qquad
  J_r:=\dist(o,V_{r,t}).
\]
Then
\begin{equation}\label{eq:wtaudiff}
    \begin{split}
          \inf_{\substack{r'\in[0,t]\\|r-r'|\le x}}
  \dist(W_r,\Gamma_{0,t}^{o,W_t}(r'))
  =& \dist(W_r, V_{r,t})+ 
    \inf_{\substack{r'\in[0,t]\\|r-r'|\le x}}
    \dist(V_{r,t},\Gamma_{0,t}^{o,W_t}(r'))\\
      =& R_r-J_r+ 
      \inf_{\substack{r'\in[0,t]\\|r-r'|\le x}}
      \left|J_r-\left\lfloor\frac{r'}{t}R_t\right\rfloor\right|\\
      \le & 2(R_r-J_r)+
           \inf_{\substack{r'\in[0,t]\\|r-r'|\le x}}
           \left|R_r-\frac{r'}{t}R_t\right|+1.
    \end{split}
\end{equation}
We control the two terms separately.  If
$R_r-J_r>y/2$, the future radial walk must move at least $y/2$
levels toward $o$ before reaching $W_t$.  Gambler's ruin for the
outward-biased radial chain, applied at the walk jump times in each
unit interval, gives
$\P[\mathcal{E}_{WV}^c]
  \le Ct(d-1)^{-y/2}$, where
$\mathcal{E}_{WV}:=\{\sup_{r\le t}(R_r-J_r)\le y/2\}$.
For the second term in the last line of \eqref{eq:wtaudiff}, we define the good radial event
\[
  \mathcal{E}_{\rm rad}
  :=\Big\{\sup_{0\le r\le t}|R_r-\nu_dr|\le \nu_dx/3\Big\},
\]
where $\nu_d=\frac{d-2}{2d}$.
Standard maximal estimates for the radial walk give $\P[\mathcal{E}_{\rm rad}^c]  \le Ce^{-cx^2/t}$. On the other hand, on $\mathcal{E}_{\rm rad}$, we have that for any $r$,
\begin{equation*}
    \begin{split}
       R_r-\frac{(r-x)\vee 0 }{t}R_t &\geq 
  (\nu_d x/3)\wedge 0=0,\\
R_r-\frac{(r+x)\wedge  t }{t}R_t & \leq 
(-\nu_d x/3) \vee \bigg(\sup_{0 \le s\le t} R_s-R_t\bigg)
\leq 0 \vee \bigg(\sup_{0 \le s\le t} R_s-R_t\bigg).
    \end{split}
\end{equation*}
Also note that under $\mathcal{E}_{WV}$, we have $\sup_{r\le t}R_r-R_t\le y/2$.
Therefore, with \eqref{eq:wtaudiff}, under $\mathcal{E}_{WV}\cap\mathcal{E}_{\rm rad}$ we now have
\[
  \inf_{\substack{r'\in[0,t]\\|r-r'|\le x}}
  \dist(W_r,\Gamma_{0,t}^{o,W_t}(r')) \le \frac32y+1.
\]
For $y\ge2$, the right-hand side is at most $2y$. Thus we get
\eqref{eq:wotin} by the upper bounds on $\P[\cE_{WV}^c]$ and $\P( \cE_{\rm rad}^c)$.  The range $1\le y<2$ is absorbed by
taking a large enough $C$.

For the integrated estimate \eqref{eq:tubeint}, we use the corresponding fixed-time
version of the preceding argument. 
Namely, for fixed $r$, by considering the events $\{R_r-J_r\le y/2\}$ and $\cE_{\rm rad}$, we conclude that
\[
  \P[(W_r,r)\notin
    \mathcal V_{0,t;x}(o,W_t;2y)]< Ce^{-cx^2/t}+C(d-1)^{-y/2}.
\]
Integrating this estimate over $r\in[0,t]$ and applying Fubini's
theorem proves \eqref{eq:tubeint}.
\end{proof}

The following lemma will be frequently used 
in combination with the preceding one to compare the weight at a random walk endpoint and its tube-restricted version.
\begin{lemma}\label{lem:sb-markov}
Take a countable set $\Lambda$, and $p, q:\Lambda\to \R_{\ge 0}$, with
$q\le p$ and $\sum_{x\in\Lambda}p(x)=1$. Let $Z$ be a $\Lambda$-valued random variable with law $p$. Then, for
$0<\alpha<1$,
\[
  \P[q(Z)<\alpha p(Z)]
  \le
  \frac{1-\sum_{x\in \Lambda}q(x)}{1-\alpha}.
\]
\end{lemma}

\begin{proof}
We note that $p(Z)>0$ almost surely.
The event in the left-hand side can be written as
$(p(Z)-q(Z))/p(Z)>1-\alpha$, while the right-hand side equals $$\frac{1}{1-\alpha}\E\left[\frac{p(Z)-q(Z)}{p(Z)}\right].$$ Thus the conclusion follows by Markov's inequality.
\end{proof}

\section{General $L^1$ decay criterion}
\label{sec:general-criterion}

In this section, we build the connection between $L^1$ decay in repeated averaging and a biased sampling of the profile.
The main result is \Cref{prop:threshold-L1} below, which gives two-sided bounds on the $L^1$ distance using threshold probabilities. 
It will be used to prove the Gaussian cutoff profile in \Cref{thm:main-rdreg}.
From \Cref{prop:threshold-L1} we also deduce the cutoff criterion \Cref{thm:intro-entropic-criterion} in this section.

As in \Cref{ssec:crit}, we let $G=(V,E)$ be a connected graph with $n$ vertices, and let
$\lambda_2>0$ be the second-smallest eigenvalue of its combinatorial Laplacian.
Throughout this section, the Poisson point process $\Pi\subset E\times \R$ has rate $\rho=1$. 
We consider repeated averaging on $G$ driven by $\Pi$, with initial 
profile $X_0$, which may be random but independent of $\Pi$, and almost surely
\[
  X_0(v)\ge0, \quad \forall\; v\in V,
  \qquad
  \sum_{v\in V}X_0(v)=1.
\]
We first recall the $L^2$ contraction for repeated averaging
\cite[Proposition 2]{AL}.
\begin{lemma}\label{lem:L2-contraction}
For every $t\ge0$,
\[
  \E[
    \norm{X_t-\one/n}_2^2\mid X_0
  ]
  \le
  e^{-\lambda_2t/2}
  \norm{X_0-\one/n}_2^2.
\]
\end{lemma}

We next state the key conversion.
\begin{prop}
\label{prop:threshold-L1}
Conditional on $X_0$, let $U_1,U_2$ be independent with common law
$X_0$, and put $X_0^{*,i}:=X_0(U_i)$ for $i=1, 2$.  For $\theta>1$, define the unconditional probabilities
\[
  \pi_1:=\P[X_0^{*,1}>\theta/n],
  \qquad
  \pi_2:=\P[X_0^{*,1}>\theta/n,\ X_0^{*,2}>\theta/n].
\]
Then, for $0<\epsilon<1$, we have
\begin{equation}\label{eq:X011bd}
\P\Big[
  \norm{X_0-\one/n}_1
  \ge
  2\Big(
    \pi_1-\sqrt{\epsilon^{-1}(\pi_2-\pi_1^2)}-\theta^{-1}
  \Big)
\Big]\ge1-\epsilon.
\end{equation}
Moreover, with 
\begin{equation}\label{eq:cleanup_time_def}
          s:=
  \frac2{\lambda_2}
  (\log((1-\pi_1)\theta))\vee 0
  +
  \frac8{\lambda_2}\log(1/\epsilon)
\end{equation}
we have 
\begin{equation}\label{eq:Xp1ubd}
\P\Big[
  \norm{X_s-\one/n}_1
  \le
  2\Big(
    \pi_1+\sqrt{\epsilon^{-1}(\pi_2-\pi_1^2)}
  \Big)+\epsilon
\Big]\ge1-2\epsilon.
\end{equation}
\end{prop}

\begin{proof}
Let
\[
  M_\theta
  :=
  \sum_{v\in V}X_0(v)\mathds{1}[X_0(v)>\theta/n].
\]
Then we have $\E M_\theta=\pi_1$ and $\E M_\theta^2=\pi_2$.  Hence, with probability
at least $1-\epsilon$,
\begin{equation}\label{eq:gamma-concentration}
  \pi_1-\sqrt{\epsilon^{-1}(\pi_2-\pi_1^2)} \le M_\theta\le \pi_1+\sqrt{\epsilon^{-1}(\pi_2-\pi_1^2)}.
\end{equation}
Write $\mathcal H_\theta:=\{v\in V:X_0(v)>\theta/n\}$.  Then necessarily
$|\mathcal H_\theta|/n\le1/\theta$, and
\begin{equation}\label{eq:mtheta_comp}
      M_\theta
  =\sum_{v\in\mathcal H_\theta}
    \left(X_0(v)-\frac1n\right)
    +\frac{|\mathcal H_\theta|}{n}
  \le\frac12\norm{X_0-\frac{\one}{n}}_1+\theta^{-1}.
\end{equation}
This proves the lower bound \eqref{eq:X011bd}.

For the upper bound \eqref{eq:Xp1ubd}, let $X_t^{\rm lo}$ be the repeated averaging
process driven by the same Poisson point process $\Pi$, with initial profile
$X_0^{\rm lo}(v)=X_0(v)\mathds{1}[X_0(v)\le\theta/n]$.
Its total mass is $1-M_\theta$, and
\[
  \norm{X_0^{\rm lo}-(1-M_\theta)\one/n}_2^2 \le \norm{X_0^{\rm lo}}_2^2
  \le\frac{\theta(1-M_\theta)}n.
\]
By \Cref{lem:L2-contraction} and the Cauchy-Schwarz inequality, we have
\begin{multline*}
  \E\norm{X_s^{\rm lo}-(1-M_\theta)\one/n}_1
  \le
    \sqrt{n}\E\norm{X_s^{\rm lo}-(1-M_\theta)\one/n}_2
    \le e^{-\lambda_2s/4} \E \sqrt{\theta(1-M_\theta)} \\
    \le 
  e^{-\lambda_2s/4}\sqrt{\theta(1-\pi_1)}\le \epsilon^2,
\end{multline*}
where the last inequality is by the choice of $s$.
Finally, $X_s-X_s^{\rm lo}$ is component-wise nonnegative and has
total mass $M_\theta$, so
\[
\begin{aligned}
  \norm{X_s-\one/n}_1
  &\le
  \norm{X_s^{\rm lo}-(1-M_\theta)\one/n}_1
  +\norm{(X_s-X_s^{\rm lo})-M_\theta\one/n}_1\\
  &\le
  \norm{X_s^{\rm lo}-(1-M_\theta)\one/n}_1+2M_\theta.
\end{aligned}
\]
The upper bound \eqref{eq:Xp1ubd} now follows from Markov's inequality and \eqref{eq:gamma-concentration}.
\end{proof}

The cutoff criterion in \Cref{thm:intro-entropic-criterion} follows from applying the two bounds in \Cref{prop:threshold-L1} at an entropic time sequence.

\begin{proof}[Proof of \Cref{thm:intro-entropic-criterion}]
We apply \Cref{prop:threshold-L1} to $X_{t_n}$, with $t_n$ viewed as
the initial time. 
The hypothesis that $nX_{t_n}^{*,(n)}\cvp\infty$ precisely says that there is a sequence $\theta_n \to \infty$, such that
\[
  \P[X_{t_n}^{*,(n)}>\theta_n/n]\to 1.
\]
Then for $\pi_{1,n}:=\P[X_{t_n}^{*,(n)}>\theta_n/n]$ and $\pi_{2,n}:=\P[X_{t_n}^{*,(n)}>\theta_n/n, \hat{X}_{t_n}^{*,(n)}>\theta_n/n]$, where $\hat{X}_{t_n}^{*,(n)}$ is an independent copy of $X_{t_n}^{*,(n)}$ conditional on $X_{t_n}$, we have that as $n\to\infty$, both of them tend to 1, and thus $\pi_{2,n}-\pi_{1,n}^2\to 0$.
Take a sequence $\epsilon_n\downarrow0$ slowly enough that
$(\pi_{2,n}-\pi_{1,n}^2)/\epsilon_n\to 0$.  Applying \eqref{eq:X011bd} now proves the first convergence in \eqref{eq:Xt20}.

For the other bound, the other hypothesis on $X_{t_n}^{*,(n)}$ implies
the existence of a sequence
$\eta_n\downarrow 0$, such that 
by instead taking $\theta_n=n^{\eta_n}$, we have that both $\pi_{1,n}$ and $\pi_{2,n}$ tend to zero as $n\to\infty$. 
Then since our first hypothesis in \eqref{eq:hypo_general_cri} implies that $\lambda_{2,n}t_n\to \infty$, we can choose $\epsilon_n\downarrow0$ slowly
enough such that
\[
  \pi_{1,n}/\epsilon_n\to0,
  \qquad
  \frac{\log(1/\epsilon_n)}{\lambda_{2,n}t_n}\to 0.
\]
Now, by setting
\[
  s_n:=
  \frac{
    2\eta_n}{\lambda_{2,n}}\log(n)+\frac{8}{\lambda_{2,n}}\log(1/\epsilon_n),
\]
we have $s_n=o(t_n)$ by the first hypothesis of \eqref{eq:hypo_general_cri} and the choice of $\eta_n$ and $\epsilon_n$.
Then using \eqref{eq:Xp1ubd}, we get the second convergence in \eqref{eq:Xt20}.
\end{proof}

\section{Entropy rate and strict slowdown on the tree}
\label{sec:strictslow}
In this and the next two sections, we work on the infinite $d$-regular tree $\T=\T_d=(V,E)$, with a Poisson point process $\Pi\subset E\times \R$ with rate $1/d$.
We consider repeated averaging $X_t$ starting from unit mass at the root $o$, i.e., $X_0(v)=\mathds{1}[v=o]$.
Let $W$ be the random walk generated by $\Pi$ in the time interval $[0, \infty)$, starting from $o$.

Let $p_t$ be the time-$t$ transition probability of the random walk on $\T$, with edge jump rate $1/(2d)$. 
From \eqref{eq:size-biased}, we have $p_t(o,v)=\E X_t(v)$.
For any $q:V\to\R_{\ge 0}$ with $\sum_{v\in V} q(v)=1$, we write $H(q):=-\sum_{v\in V}q(v)\log q(v)$ for the entropy.
Then define
\[
  H_{\rm ann}(t):= H(p_t(o,\cdot)), \qquad 
  H_{\rm que}(t):= H(X_t).
\]

In this section, we construct the entropy rate $\kappa_*$ of the repeated averaging process on $\T$, and prove the strict inequality $0<\kappa_*<\kappa_{\rm RW}$, where we recall $\kappa_{\rm RW}=((d-2)\log(d-1))/(2d)$.
We start by collecting some asymptotics for the random walk transition probabilities and entropy $H_{\rm ann}(t)$.  

\begin{itemize}
    \item 
First, as $t\to\infty$, we have
\begin{equation}   \label{eq:rw-entropy}
    H_{\rm ann}(t)=\kappa_{\rm RW}t+O(\log(t)).
\end{equation}
This estimate follows by splitting the entropy into a spherical and a radial part, which account for the two terms respectively (see e.g., \cite[Section~6.4]{FlorescuPeresRacz}).
\item
Second, there is a local central limit theorem for the radial projection of the walk $W$: as $t\to\infty$, 
\[
  \sup_{m\in \Z_{\ge 0}}
  \left|
    \sqrt{\pi t} \P[\dist(o,W_t)=m]
    -
    \exp\left(-\frac{(m-(d-2)t/(2d))^2}{t}\right)
  \right|
  \to 0.
\]   
This follows from expressing the transition probabilities of the radial part $\dist(o,W_t)$ in terms of those of a random walk on $\Z$ (as in, e.g., \cite[Proposition~2.5(i)]{CowlingMedaSetti}), and applying a local central limit theorem for random walks on $\Z$ (see, e.g.,
\cite[Theorem 3.5.3]{Durrett} and
\cite[Section 2.5.2]{LL}).
Consequently, for any fixed $M>0$, uniformly over all $u\in V$ with
$\bigl|\dist(o,u)-(d-2)t/(2d)\bigr|\le M\sqrt t$, 
\begin{equation}   \label{eq:tree-local-clt}
  p_t(o,u)\asymp t^{-1/2}(d-1)^{-\dist(o,u)},
\qquad  \frac{p_t(o,u^-)}{p_t(o,u)}
  \to d-1,
\end{equation}
where $u^-$ is the parent of $u$.
\item 
Let $\hat{W}$ be an independent (conditional on $\Pi$) copy of $W$. Then we have 
\begin{equation}\label{eq:distance-lyapunov2}
  \P\big[\dist(W_t, \hat{W}_t) =0\big]
  \le\P\big[\dist(W_t, \hat{W}_t) \le (d-2)t/(2d)\big]
  \le e^{-ct},
  \qquad \forall\;t\ge0.
\end{equation}
This can be proved by considering $\dist(W_t, \hat{W}_t)$ as a continuous-time Markov chain on $\Z_{\ge 0}$, with upward drift $(d-2)/d$ away from $0$ and $1$, and using Chernoff bounds.
\end{itemize}

The main result of this section is the following.
\begin{prop}\label{prop:kappa-star}
The random process $t\mapsto L_t:=-\log X_t(W_t)$ is stochastically subadditive: for any $s, t\ge 0$, we have that $L_{s+t}$ is stochastically dominated by $L_s + \hat{L}_t$, where $\hat{L}_t$ is an independent copy of $L_t$.
Then we can define 
\begin{equation}\label{eq:kappa-star}
\kappa_*:=\lim_{t\to\infty}\frac{\E[ -\log X_t(W_t)]}{t}=\inf_{t>0}\frac{\E [-\log X_t(W_t)]}{t},
\end{equation}
and $0<\kappa_*<\kappa_{\rm RW}$.
\end{prop}
The claimed subadditivity follows straightforwardly from the polymer model construction and transitivity of $\T$.
The strict slowdown inequality follows by continuously interpolating between random walks and repeated averaging, and establishing a lower bound on the rate of decrease of the entropy along this interpolation.
More precisely, since repeated averaging can be viewed as random walks with an extra layer of randomness, the interpolation is realized through an adaptive revealing of the extra randomness, as outlined in \Cref{ssec:fpn}.

\begin{proof}[Proof of \Cref{prop:kappa-star}]
\textbf{Subadditivity.}
From \eqref{eq:XWrel} and the composition law in \eqref{eq:chapman-kolmogorov}, we have
\[
X_{s+t}(W_{s+t}) = w_{0,s+t}(o, W_{s+t}) \ge w_{0,s}(o,W_s) w_{s,s+t}(W_s, W_{s+t}) = X_s(W_s) w_{s,s+t}(W_s, W_{s+t}).
\]
Using transitivity of $\T$ and stationarity of the point process $\Pi$, we have that $w_{s,s+t}(W_s, W_{s+t})$ has the same law as $w_{0,t}(o, W_t)=X_t(W_t)$. Thus the claimed stochastic subadditivity follows.
In particular, we have $\E L_{s+t}\le \E L_s + \E L_t$.
The
continuous-parameter subadditive theorem (see, e.g., \cite[Theorem 16.2.9]{KM}) then proves
\eqref{eq:kappa-star}.

\smallskip
\noindent
\textbf{Strict positivity.}
Since $x\mapsto -\log x$ is convex, we have
\begin{equation}\label{eq:lb_kappa}
    \E L_t \ge 
  -\log\E X_t(W_t)
  =
  -\log\E\Big[\sum_{v\in V}X_t(v)^2\Big] = -\log \P[W_t=\hat{W}_t].
\end{equation}
Here $\hat{W}$ is an independent (conditional on $\Pi$) copy of $W$, and the two equalities follow from \eqref{eq:size-biased}.
Then \eqref{eq:distance-lyapunov2} gives $\kappa_*>0$.

\smallskip
\noindent
\textbf{Strict slowdown.}
Recall that $p_t(o,v)=\E X_t(v)$. Set $F(x):=-x\log x$. Then by concavity of $F$, we have
$\E H_{\rm que}(t)\le H_{\rm ann}(t)$
and hence $\kappa_*\le\kappa_{\rm RW}$.  We now prove that, for all
sufficiently large $t$,
\begin{equation}\label{eq:entropy-gap-sqrt}
  H_{\rm ann}(t)-\E H_{\rm que}(t)\ge c\sqrt t.
\end{equation}
Indeed, once \eqref{eq:entropy-gap-sqrt} is proved, 
\eqref{eq:rw-entropy} gives, for all sufficiently large $t$,
\[
  \E H_{\rm que}(t)
  \le \kappa_{\rm RW}t+C\log(t)-c\sqrt t
  <\kappa_{\rm RW}t.
\]
By \eqref{eq:size-biased}, the infimum characterization of $\kappa_*$ in
\eqref{eq:kappa-star} can be written as $\kappa_*=\inf_{t>0}\E H_{\rm que}(t)/t$.
These imply that $\kappa_*<\kappa_{\rm RW}$.

It remains to prove \eqref{eq:entropy-gap-sqrt}.
For $0\le s\le t$, let
\[
  \cF_s:=\sigma\!\left(
    \Pi\cap(E\times[0,s])
  \right),\qquad
  M_s(v)
  :=
  \E[X_t(v)\mid\cF_s]
  =
  \sum_{u\in V}X_s(u)p_{t-s}(u,v).
\]
Then for each $v\in V$, $s\mapsto M_s(v)$ is a nonnegative martingale.  
Again using the concavity of $F$, we have that $\E F(M_s(v))$ is non-increasing in $s$.
Moreover, we can lower bound the rate of decrease of $\E F(M_s(v))$ in $s$: in the case where $(\{u, u'\}, s)\in \Pi$ for some edge $\{u, u'\}\in E$, $M_s(v)-M_{s-}(v)$ equals
\[
  \frac12 (X_{s-}(u')-X_{s-}(u))
  (p_{t-s}(u,v)-p_{t-s}(u',v)).
\]
Using the fact that
\[
  F(y)-F(x)\le F'(x)(y-x)
  - \frac{(x-y)^2}{2\max\{x,y\}},
\]
which follows from integrating $F''(x)=-1/x$, and the fact that $M_s(v)$ is a martingale, we conclude that
\[
F(M_0(v))-\E F(M_t(v))\gtrsim \int_0^t \sum_{\{u,u'\}\in E} (p_{t-s}(u,v)-p_{t-s}(u',v))^2 \E \big[ (X_{s-}(u')-X_{s-}(u))^2 / \hat{M}_s(v)
    \big]  \dd s,
\]
where $\hat{M}_s(v):=M_s(v) + \sum_{u\in V}X_s(u)\sum_{u'\sim u}p_{t-s}(u',v)$.
Using Cauchy-Schwarz, we have
\[
\E \big[ (X_{s-}(u')-X_{s-}(u))^2 / \hat{M}_s(v)
    \big]  \ge   ( p_s(o,u) - p_s(o,u') )^2 / \E \hat{M}_s(v).
\]
We also have that 
\[
\E \hat{M}_s(v) = p_t(o, v) + \sum_{u\in V}p_s(o,u)\sum_{u'\sim u}p_{t-s}(u',v) \lesssim p_{t+1}(o, v),
\]
since the probability of staying put or making one prescribed jump
during one time unit is bounded below uniformly.
With the three estimates above, we conclude that
\[
F(M_0(v))-\E F(M_t(v))\gtrsim \int_0^t \sum_{\{u,u'\}\in E} (p_{t-s}(u,v)-p_{t-s}(u',v))^2  ( p_s(o,u) - p_s(o,u') )^2 / p_{t+1}(o,v)  \dd s.
\]
Take $v\in V$ with $|\dist(o,v)-(d-2)t/(2d)|\le \sqrt{t}$. Restrict the above integral to $s\in [t/3, 2t/3]$, and the sum to $u, u'$ being on the geodesic from $o$ to $v$, with $|\dist(o,u)-(d-2)s/(2d)|\le 2\sqrt{t}$ and $u'$ being the parent of $u$.
By \eqref{eq:tree-local-clt}, for such $v,s,u,u'$,
\[
\begin{split}
  |p_{t-s}(u,v)-p_{t-s}(u',v)|
  &\gtrsim
  t^{-1/2}(d-1)^{-\dist(o,v)+\dist(o,u)},\\
  |p_s(o,u)-p_s(o,u')|
  &\gtrsim
  t^{-1/2}(d-1)^{-\dist(o,u)},\\
  p_{t+1}(o,v)
  &\lesssim
  t^{-1/2}(d-1)^{-\dist(o,v)}.
\end{split}
\]
Thus we have
\[
F(M_0(v))-\E F(M_t(v))\gtrsim (d-1)^{-\dist(o,v)}.
\]
For any other $v$, we still have $F(M_0(v))-\E F(M_t(v))\ge 0$, by the fact stated above that $\E F(M_s(v))$ is non-increasing in $s$.

Then, since $H_{\rm ann}(t)=\sum_{v\in V} F(M_0(v))$ and $H_{\rm que}(t)=\sum_{v\in V} F(M_t(v))$, summing over all $v\in V$ gives \eqref{eq:entropy-gap-sqrt}, from which the conclusion follows.
\end{proof}

\section{Exponential forgetting and tree central limit theorems}
\label{sec:cltontrees}
Continuing to work on the infinite $d$-regular tree $\T=(V, E)$ with the setup of the previous section, 
in this section we shall prove the central limit theorem for $-\log X_t(W_t)$ while leaving the non-degeneracy of the limiting variance to the next section.

As already discussed in \Cref{ssec:fpn}, the main strategy is to consider the increments $-\log \frac{X_t(W_t)}{X_{t-1}(W_{t-1})}$, and show they have exponential memory decay. 
More precisely, 
for $s\ge0$ and $t\ge1$, we consider the restarted process using only clocks after time $s$, and define
\[
  P_t^{(s)}
  :=-\log\frac{w_{s,s+t}(W_s,W_{s+t})}
  {w_{s,s+t-1}(W_s,W_{s+t-1})}
  +\E\log\frac{w_{s,s+t}(W_s,W_{s+t})}
  {w_{s,s+t-1}(W_s,W_{s+t-1})}.
\] 
We let $P_t:=P_t^{(0)}$.
We will prove an exponential forgetting estimate in \Cref{lem:expdecay}, which bounds the difference $|P_{s+t}-P_t^{(s)}|$.
The key technical ingredient here is a four-path switching estimate (\Cref{lem:four-path}), proved by analyzing a pair of polymers.

In preparation for these, we first establish basic stochastic domination relations and moment estimates for polymers along the random walk path. 

\begin{itemize}
    \item For any $0\le s<t$, we have
\begin{equation}\label{eq:forward-ratio-one}
  \E\frac{X_s(W_s)}{X_t(W_t)}
  \le e^{t-s},\qquad \E\frac{X_t(W_t)}{X_s(W_s)}\le 1.
\end{equation}
    The first estimate follows from the fact that
    \begin{equation}   \label{eq:stocha}
    -\log_2 w_{s,t}(W_s,W_t) \text{ is stochastically dominated by } \Poi(t-s),
    \end{equation}
    where, for $\lambda>0$, $\Poi(\lambda)$ denotes a Poisson random variable with mean $\lambda
    $.
    The second estimate in \eqref{eq:forward-ratio-one} follows from 
    \[
 \E\left[
    \frac{X_t(W_t)}{X_s(W_s)}
    \,\middle|\,\Pi
  \right] = \sum_{u,v} w_{s,t}(u,v) X_s(u) \mathds{1}[X_s(u)>0] \frac{X_t(v)}{X_s(u)} \le 1,
    \]
    where the last inequality uses \eqref{eq:chapman-kolmogorov}.
    \item Take any $\eta>0$. Then there is $T$ so that $\E[-\log X_T(W_T)]\le(\kappa_*+\eta/2)T$. Thus by \Cref{prop:kappa-star} and \eqref{eq:stocha}, 
$-\log X_t(W_t)$ is stochastically dominated by a sum of
$\lfloor t/T\rfloor$ independent copies of
$-\log X_T(W_T)$ plus a remainder dominated by
$\log(2)\Poi(T)$. Thus we have
\begin{equation}    \label{eq:typical-weight}
  \P\big[
    X_t(W_t)\le e^{-(\kappa_*+\eta)t}
  \big]
  < Ce^{-ct},
  \qquad t\ge0,
\end{equation}
for $C, c>0$ depending on $\eta$.
    \item Using the representation of $\dist(o,W_t)$ in terms of a random walk on $\Z$, as in the proof of \eqref{eq:tree-local-clt}, and the bound $p_t(o,u)\lesssim (d-1)^{-\dist(o,u)}$ which comes from symmetry, we have that
    for every $\eta>0$,
\begin{equation}  \label{eq:tree-entropic-tail}
  \P\big[
    p_t(o,W_t)>e^{-(\kappa_{\rm RW}-\eta)t}
  \big]
  < Ce^{-ct}, \quad \forall\, t\geq 0,
\end{equation}
for $c,C>0$ depending on $\eta$.
\end{itemize}

We first isolate the weight of the pairs of paths for which the two tails cannot be
switched. Call a pair of paths $\gamma_1:[s_1, t_1]\to V$ and $\gamma_2:[s_2, t_2]\to V$ switchable if their histories intersect or share a
spacetime Poisson clock jump, i.e., $\gamma_1(r)=\gamma_2(r)$, or $(\{\gamma_1(r), \gamma_2(r)\}, r)\in \Pi$ for some $r\in [s_1, t_1]\cap [s_2, t_2]$.
Otherwise they are called non-switchable.
For $s,h\ge0$ and $A,A',B,B'\in V$, let
\[
 w_{\rm ns}^{s,s+h}(A,B\mid A',B')=
 \sum_{\substack{\gamma_1 \in \mathcal C_{s,s+h}[A,B],\gamma_2 \in \mathcal C_{s,s+h}[A',B']\\ \gamma_1 \textrm{ and }\gamma_2 \textrm{ are non-switchable}}} w(\gamma_1) w(\gamma_2).
\]

\begin{lemma}
\label{lem:graphical-nonconnection}
 We have $\E w_{\rm ns}^{s,s+h}(A,B\mid A',B')
  \le p_h(A,B)p_h(A',B')$.
\end{lemma}

\begin{proof}
The non-switchability restriction has an exact Markovian description.
Let $U,U'$ be independent random walks on $\T$, starting from $A$ and $A'$ at time $s$ respectively, and each jumping
across an edge at rate $1/(2d)$. Set
$\tau:=\inf\{r\ge0:U_{s+r}=U_{s+r}'\}$.
We claim that
\[
  \E w_{\rm ns}^{s,s+h}(A,B\mid A',B')
  =
  \P\left[
    U_{s+h}=B,\ U_{s+h}'=B',\ \tau>h
  \right].
\]
Indeed, we can couple $\Pi$ with $U, U'$ until $\tau$, as follows. 
First, we split $\Pi=\Pi^U\cup \Pi^{U'}$, such that for every point in $\Pi$, it belongs to either $\Pi^U$ or $\Pi^{U'}$, with equal probability. Then each of $\Pi^U$ and $\Pi^{U'}$ is a Poisson point process with rate $1/(2d)$. 
Let $U$ (resp.~$U'$) be determined by $\Pi^U$ (resp.~$\Pi^{U'}$), such that whenever $U$ (resp.~$U'$) is at an endpoint of the edge corresponding to a point in $\Pi^U$ (resp.~$\Pi^{U'}$), it must jump across the corresponding edge. 
This gives the correct law of $U, U'$ until $\tau$, and the claim above follows.
Dropping the condition $\tau>h$ then gives
the conclusion.\end{proof}

We now combine the non-switching bound with the previously given tail bounds to prove the switching estimate.
Throughout the rest of this section, we denote
$\mathcal A=W_s$, $\mathcal B=W_{s+t-1}$, $\mathcal D=W_{s+t}$,
and define the cross-ratio
\begin{equation}\label{eq:cross-ratio}
  \xi_{s,t}
  :=
  \log
  \frac{
    w_{0,s+t-1}(o,\mathcal B)w_{s,s+t}(\mathcal A,\mathcal D)
  }{
    w_{0,s+t}(o,\mathcal D)w_{s,s+t-1}(\mathcal A,\mathcal B)
  }.
\end{equation}

\begin{lemma}\label{lem:four-path}
For any $s\ge0$ and
$t\ge1$,
\begin{equation}\label{eq:cross-ratio-bound}
  \E\big| e^{\xi_{s,t}}-1\big|\le Ce^{-ct}.
\end{equation}
\end{lemma}
 
\begin{proof} 
In the proof below we write $w_{\rm ns}$ as a shorthand for  $w_{\rm ns}^{s,s+t-1}$.
For any $A,B,A',B'\in V$,
\begin{equation}\label{eq:weightswitch}
  w_{s,s+t-1}(A,B)w_{s,s+t-1}(A',B')
   -w_{s,s+t-1}(A,B')w_{s,s+t-1}(A',B)\le w_{\rm ns}(A,B\mid A',B').
\end{equation}
Indeed, this follows from switching the two tails at the first switchable point, for $\gamma_1 \in \mathcal C_{s,s+t-1}[A,B]$ and $\gamma_2 \in \mathcal C_{s,s+t-1}[A',B']$.
The composition law in \eqref{eq:chapman-kolmogorov} gives, for any $A, B, D\in V$,
\begin{align*}
  w_{0,s+t-1}(o,B)w_{s,s+t}(A,D)
  &=
  \sum_{A',B'\in V}w_{0,s}(o,A')
  w_{s,s+t-1}(A',B)
  w_{s,s+t-1}(A,B')
  w_{s+t-1,s+t}(B', D),\\
  w_{0,s+t}(o,D)w_{s,s+t-1}(A,B)
  &=
  \sum_{A',B'\in V}w_{0,s}(o,A')
  w_{s,s+t-1}(A',B')
  w_{s,s+t-1}(A,B)
  w_{s+t-1,s+t}(B',D).
\end{align*}
Using \eqref{eq:weightswitch} in these decompositions yields
\[
e^{\xi_{s,t}}-1 \le   \frac{
    \sum_{A', B\in V} w_{0,s}(o,A')
  w_{\rm ns}(\mathcal A,B\mid A',\mathcal B)
  w_{s+t-1,s+t}(B,\mathcal D)
  }{
    w_{0,s+t}(o,\mathcal D)w_{s,s+t-1}(\mathcal A,\mathcal B)
  },
\]
\[
1-e^{\xi_{s,t}} \le   \frac{
    \sum_{A', B'\in V} w_{0,s}(o,A')
  w_{\rm ns}(\mathcal A, \mathcal B\mid A', B')
  w_{s+t-1,s+t}(B',\mathcal D)
  }{
    w_{0,s+t}(o,\mathcal D)w_{s,s+t-1}(\mathcal A,\mathcal B)
  }.
\]
Note that $\P[\mathcal A=A, \mathcal B=B, \mathcal D=D \mid \Pi]=w_{0,s}(o,A)w_{s,s+t-1}(A,B)w_{s+t-1,s+t}(B,D)$. We therefore conclude that
\begin{equation}\label{eq:cross-ratio-normalized}
  \E|e^{\xi_{s,t}}-1|
  \le
  2\E\bigg[
  \sum_{D, A,B,A',B'\in V} 
  F_D(A,B;A',B')
  w_{0,s}(o,A)w_{s+t-1,s+t}(B,D)\bigg],
\end{equation}
where  
\[
  F_D(A,B;A',B')
  :=
  \begin{cases}
  \displaystyle
  \frac{w_{0,s}(o,A')
    w_{\rm ns}(A,B\mid A',B')
    w_{s+t-1,s+t}(B',D)}{w_{0,s+t}(o,D)},
    & w_{0,s+t}(o,D)>0,\\
  0,&w_{0,s+t}(o,D)=0.
  \end{cases}
\]
We next bound \eqref{eq:cross-ratio-normalized}. We will split the sum in terms of the vertices $A$ and $B$.

Fix a small $\eta>0$. All the constants $C,c>0$ in the rest of this proof are allowed to depend on $\eta$.
For any $D\in V$, let $\mathcal G_\eta(D,t)$ be the set of pairs $(A,B)$
satisfying the following conditions:
\[
 \dist(B,D)\le\eta t,\qquad
  \dist(A,B)\le t,\qquad
  \dist(A,\hat{A})\le\eta t,\qquad
  p_{t-1}(A,B)
  \le e^{-(\kappa_{\rm RW}-\eta)(t-1)},
\]
where $\hat{A}$ is the last common ancestor of $A$ and $D$ on the rooted tree $\T$.
Then the first three inequalities imply that $\dist(\hat{A},D) \le (1+2\eta)t$, and 
\begin{equation}\label{eq:good-pair-cardinality}
  |\mathcal G_\eta(D,t)|
  \le C_1 e^{C_0\eta t},
\end{equation}
where the positive constant $C_0$ in \eqref{eq:good-pair-cardinality} depends only on $d$.
As $W$ is a random walk with jump rate $1/(2d)$ across each
edge, uniformly in $s\ge0$ we have
\begin{equation}\label{eq:pwst}
  \P\bigl[(W_s,W_{s+t-1})
    \notin\mathcal G_\eta(W_{s+t},t)\bigr]
  < Ce^{-ct}.
\end{equation}
Indeed, $\P[\dist(W_{s+t-1}, W_{s+t})>\eta t], \P[\dist(W_s, W_{s+t-1})>t] < Ce^{-ct}$ by large-deviation estimates for the number of jumps of $W$ in a given interval.  The probability $\P[\dist(W_s,\tilde{W}_s)>\eta t]$, where $\tilde{W}_s$ is the last common ancestor of $W_s$ and $W_{s+t}$ on the rooted tree $\T$, can also be bounded by $Ce^{-ct}$, using the side-depth estimate in the proof of \Cref{lem:tube}.  Finally,
$\P[p_{t-1}(W_s,W_{s+t-1})
  > e^{-(\kappa_{\rm RW}-\eta)(t-1)}]\le Ce^{-ct}$ follows from 
\eqref{eq:tree-entropic-tail}.

We now split the sum in \eqref{eq:cross-ratio-normalized} into three parts.

\noindent\textbf{Part (i): $(A,B)\notin \mathcal G_\eta(D,t)$.}
Let
$E_{\rm i}$ be the sum in \eqref{eq:cross-ratio-normalized}, restricted to such $(A, B)$.
By the definition of the non-switchable weight,
$$w_{\rm ns}(A,B\mid A',B')
\le w_{s,s+t-1}(A,B)w_{s,s+t-1}(A',B').$$ Hence
\begin{equation}\label{eq:fab}
   \sum_{A',B'\in V}F_D(A,B;A',B')\le w_{s,s+t-1}(A,B),
\end{equation}
which then implies
\begin{equation}\label{eq:parti}
    \begin{split}
    \E [E_{\rm i}] 
    \le &\E\bigg[\sum_{D,A,B\in V; (A,B)\notin \mathcal G_{\eta}(D,t)}  w_{0,s}(o,A)
    w_{s,s+t-1}(A,B) w_{s+t-1,s+t}(B,D)\bigg]\\
    = &    \P[(W_s,W_{s+t-1})\notin\mathcal G_\eta(W_{s+t},t)]< C e^{-ct},
    \end{split}
\end{equation}
where the last step is due to \eqref{eq:pwst}.

\noindent\textbf{Part (ii):
$w_{s,s+t-1}(A,B)<e^{-(\kappa_*+\eta)t}$.}
Let
$E_{\rm ii}$ be the sum in \eqref{eq:cross-ratio-normalized}, restricted to such $(A, B)$.
Using \eqref{eq:fab} again gives
\begin{equation}\label{eq:partii}
    \begin{split}
      \E [E_{\rm ii} ]\le &\E\bigg[\sum_{D,A,B\in V;  w_{s,s+t-1}(A,B)<e^{-(\kappa_*+\eta)t} }  w_{0,s}(o,A)
    w_{s,s+t-1}(A,B) w_{s+t-1,s+t}(B,D)\bigg]\\
    = & \P[w_{s,s+t-1}(W_s, W_{s+t-1})<
    e^{-(\kappa_*+\eta)t}
    ] < C e^{-ct},
    \end{split}
\end{equation}
where the last inequality is by \eqref{eq:typical-weight}.

\noindent\textbf{Part (iii): $(A,B)\in\mathcal G_\eta(D,t)$ and
$w_{s,s+t-1}(A,B)\ge e^{-(\kappa_*+\eta)t}$.}
Let
$E_{\rm iii}$ be the sum in \eqref{eq:cross-ratio-normalized}, restricted to such $(A, B)$.
For such $(A, B)$, we have
\[
\begin{split}
&F_D(A,B;A',B')
 w_{0,s}(o,A)w_{s+t-1,s+t}(B,D)\\
=&
\frac{w_{0,s}(o,A)w_{s,s+t-1}(A,B)
 w_{s+t-1,s+t}(B,D)}{w_{0,s+t}(o,D)}\cdot
\frac{w_{0,s}(o,A')w_{\rm ns}(A,B\mid A',B')
 w_{s+t-1,s+t}(B',D)}{w_{s,s+t-1}(A,B)}\\
\le &
e^{(\kappa_*+\eta)t}
w_{0,s}(o,A')w_{\rm ns}(A,B\mid A',B')
w_{s+t-1,s+t}(B',D),
\end{split}
\]
because the first fraction is at most $1$ by \eqref{eq:chapman-kolmogorov}.  Then using the definition of $\mathcal G_{\eta}(D,t)$, 
\Cref{lem:graphical-nonconnection} and independence of weights across the
disjoint time intervals,
\begin{align}
\E[E_{\rm iii}]
&\le
e^{(\kappa_*+\eta)t}
\sum_{D,A,B,A',B'\in V;\,(A,B)\in\mathcal G_\eta(D,t)}
p_s(o,A')p_{t-1}(A,B)p_{t-1}(A',B')p_1(B',D)
\nonumber\\
&\le
e^{(\kappa_*+\eta)t-(\kappa_{\rm RW}-\eta)(t-1)}
\sum_{D\in V}|\mathcal G_\eta(D,t)|
 p_{s+t}(o,D)
\nonumber\\
&\le
C e^{(\kappa_*-\kappa_{\rm RW}+(2+C_0)\eta)t},
\label{eq:partiii}
\end{align}
where we used \eqref{eq:good-pair-cardinality} in the last inequality.
Since $\kappa_*-\kappa_{\rm RW}<0$, by choosing $\eta$ small enough, we get $\E[E_{\rm iii}]<Ce^{-ct}$.

Summing over the bounds on the three parts of \eqref{eq:cross-ratio-normalized}, the conclusion follows.
\end{proof}

We can now derive the exponential mixing of increments.
\begin{lemma}\label{lem:expdecay}
For every $\theta\ge 1$, we have $\sup_{m\ge1}\E|P_m|^\theta<\infty$.
Moreover, for any $s\ge0$ and
$t\ge1$,
\begin{equation}\label{eq:forgetting}
  \E\big| P_{s+t}-P_t^{(s)} \big|^2\le\E\xi_{s,t}^2< Ce^{-ct}.
\end{equation}
\end{lemma}

\begin{proof}
The uniform moment bound follows from \eqref{eq:forward-ratio-one}, and
the inequality $|x|^\theta \lesssim e^x+e^{-x}$ for any $x\in\R$ and $\theta \ge 1$.
For the second assertion, the definition of $\xi_{s,t}$ gives
\[
  P_{s+t}-P_t^{(s)}
  =
  \xi_{s,t}-\E\xi_{s,t},
\]
and thus
\[
  \E\big| P_{s+t}-P_t^{(s)} \big|^2\le\E\xi_{s,t}^2.
\]
By the Cauchy--Schwarz inequality and \eqref{eq:forward-ratio-one} again, we have
\[
\E e^{-\xi_{s,t}/2}
  \le
    \left(\E\frac{w_{0,s+t}(o,\mathcal D)}
             {w_{0,s+t-1}(o,\mathcal B)}\right)^{1/2}
    \left(\E\frac{w_{s,s+t-1}(\mathcal A,\mathcal B)}
             {w_{s,s+t}(\mathcal A,\mathcal D)}
  \right)^{1/2}
  \le\sqrt e.
\]
Consequently, using \Cref{lem:four-path} and the inequality $x^2
  \lesssim
 e^{-x/4}|e^x-1|^{1/2}$,
we get
\begin{equation}\label{eq:cross-ratio-L2}
  \E\xi_{s,t}^2
  \le
  C\E\left[
    e^{-\xi_{s,t}/4}
    |e^{\xi_{s,t}}-1|^{1/2}
  \right]
  \le
  C(\E e^{-\xi_{s,t}/2})^{1/2}
  (\E|e^{\xi_{s,t}}-1|)^{1/2}
  < Ce^{-ct},
\end{equation}
which finishes proving \eqref{eq:forgetting}.
\end{proof}

The exponential forgetting provides refined estimates on both increment means and covariances, and identifies the asymptotic
variance.

\begin{lemma}
\label{lem:covariance}
There is $b_{\rm ent}\in\R$ such that for any $m\in \N$,
\begin{equation}\label{eq:mean-stabilization}
  \abs{\E\log X_m(W_m)+\kappa_*m-b_{\rm ent}}
  < Ce^{-cm}.
\end{equation}
Furthermore, for any $s\ge 1$ and $t \ge 0$,
\begin{equation}\label{eq:epspst}
  |\E[P_sP_{s+t}]|< Ce^{-ct}.
\end{equation}
For every integer $j\ge0$, the limit
$e_j:=\lim_{t\to\infty}\E[P_tP_{t+j}]$
exists, and for $t\ge 1$,
\begin{equation}\label{eq:covariance-stabilization}
  |e_j|< Ce^{-cj},
  \qquad
  |\E[P_tP_{t+j}]-e_j|
  < Ce^{-c(t\vee j)}.
\end{equation}
Finally, we can define
\begin{equation}\label{eq:sigma}
  \sigma^2:=\lim_{n\to\infty}
  \frac1n
  \Var\left(\sum_{m=1}^nP_m\right)
  =
  e_0+2\sum_{j=1}^\infty e_j.
\end{equation}
\end{lemma}

\begin{proof}
For $s\ge 0$ and $t\ge1$, we set
\[
  Q_t^{(s)}
  :=-\log\frac{w_{s,s+t}(W_s,W_{s+t})}
  {w_{s,s+t-1}(W_s,W_{s+t-1})},
  \qquad Q_t:=Q_t^{(0)}.
\]
By the definition \eqref{eq:cross-ratio},
$Q_{s+t}-Q_t^{(s)}=\xi_{s,t}$ and $Q_t^{(s)}$ has the same distribution as $Q_t$.
By \eqref{eq:forgetting}, for $t'\ge t\ge1$,
\begin{equation}\label{eq:increment-mean-rate}
  \abs{\E Q_{t'}-\E Q_t}
  \le \bigl(\E\xi_{t'-t,t}^2\bigr)^{1/2}
  < Ce^{-ct}.
\end{equation}
Thus $\E Q_t$ converges as $t\to\infty$, with the rate in
\eqref{eq:increment-mean-rate}, and its limit is
\[
  \lim_{t\to\infty}\E Q_t
  =\lim_{t\to\infty}\E Q_{\lfloor t\rfloor}
  =\lim_{t\to\infty}\frac{1}{\lfloor t\rfloor}\sum_{m=1}^{\lfloor t\rfloor}\E Q_m
  =\lim_{t\to\infty}\frac{\E[-\log X_{\lfloor t\rfloor}(W_{\lfloor t\rfloor})]}{\lfloor t\rfloor}
  =\kappa_*,
\]
by \eqref{eq:kappa-star}.  In particular, with
$a_t:=\E Q_t-\kappa_*$, we have $|a_t|\le Ce^{-ct}$.
Thus by defining
$$b_{\rm ent}:=-\sum_{k=1}^{\infty}a_k,$$
we have
\[
  \abs{\E\log X_m(W_m)+\kappa_*m-b_{\rm ent}}
  =\abs{\sum_{k=m+1}^{\infty}a_k}
  < Ce^{-cm},
\]
which proves \eqref{eq:mean-stabilization}.
For $0\le t<1$, \eqref{eq:epspst} follows from the uniform second-moment bound in \Cref{lem:expdecay}, applied to $P_s$ and $P_{s+t}$. For $t\ge 1$, we have
\[
  |\E[P_sP_{s+t}]|
  =
  |\E[P_s(P_{s+t}-P_t^{(s)})]| \le (\E P_s^2)^{1/2} \big(\E |P_{s+t}-P_t^{(s)}|^2\big)^{1/2}
  < Ce^{-ct},
\]
where the last inequality follows from \Cref{lem:expdecay}.

For $t\ge1$ and $j\ge0$, put
\[
  \operatorname{Cov}_t(j)
  :=\operatorname{Cov}(P_t,P_{t+j})
  =\E[P_tP_{t+j}].
\]
For $r\ge0$, since $(P_t, P_{t+j})$ has the same joint law as $(P_t^{(r)}, P_{t+j}^{(r)})$, we can write
\[
\begin{split}
|\operatorname{Cov}_{t+r}(j)-\operatorname{Cov}_t(j)| &= |\E[P_{t+r}P_{t+r+j}] - \E[P_t^{(r)}P_{t+j}^{(r)}]| \\&\le |\E[(P_{t+r}-P_t^{(r)})P_{t+r+j}]| + |\E[P_t^{(r)}(P_{t+r+j}-P_{t+j}^{(r)})]|
\\
&\le (\E|P_{t+r}-P_t^{(r)}|^2)^{1/2}(\E P_{t+r+j}^2)^{1/2} + (\E|P_t^{(r)}|^2)^{1/2}(\E|P_{t+r+j}-P_{t+j}^{(r)}|^2)^{1/2}.
\end{split}
\]
Then by \Cref{lem:expdecay}, we can bound the above by $Ce^{-ct}$.
This implies the convergence of $\operatorname{Cov}_t(j)$ as $t\to\infty$, as well as $\abs{\operatorname{Cov}_t(j)-e_j}\le Ce^{-ct}$.
Then \eqref{eq:covariance-stabilization} follows from this together with \eqref{eq:epspst}.

Finally, \eqref{eq:sigma} holds true because
\begin{align*}
  \frac1n\Var\left(\sum_{m=1}^nP_m\right)
  &=\frac1n\sum_{m=1}^n\operatorname{Cov}_m(0)
    +\frac2n\sum_{j=1}^{n-1}\sum_{m=1}^{n-j}
      \operatorname{Cov}_m(j)\\
  &=e_0+2\sum_{j=1}^{n-1}\left(1-\frac jn\right)e_j
    +O(n^{-1})
  \to e_0+2\sum_{j\ge1}e_j, \quad \mbox{as }n\to \infty,
\end{align*}
where the error term $O(n^{-1})$ is due to
$\sum_{m\ge1}e^{-cm}
  +\sum_{m,j\ge1}e^{-c(m\vee j)}<\infty$.
\end{proof}
We can now deduce the central limit theorem for $-\log X_t(W_t)$; the fact that $\sigma^2>0$ will be shown in the next section.
\begin{theorem}\label{thm:tree-clt}
As $t\to\infty$, $\frac{-\log X_t(W_t)-\kappa_*t}{\sqrt t}$ converges in distribution to $\mathcal{N}(0, \sigma^2)$, the centered normal distribution with variance $\sigma^2$.
\end{theorem}
\begin{proof}
We prove this using Berk's central limit theorem with an increasing dependence range \cite{BerkMDependent}. We note that some other versions of the central limit theorem for dependent variables, such as \cite[Theorem~4.1]{JansonMDependent}, should be applicable as well.

Take any $n\in \N$ and $m=\lfloor (\log(n))^2 \rfloor$.
Then the sequence $P_m^{(k-m)}$ for $k=m+1, \ldots, n$ is $m$-dependent.
We next verify the conditions of the main theorem in \cite{BerkMDependent} for this sequence.

\begin{itemize}
    \item[(i)] The variables $P_m^{(k-m)}$ are centered and have uniformly bounded fourth moments by \Cref{lem:expdecay}.
    \item[(ii)] For $m<i<j\le n$ with $j-i\le m$,
\eqref{eq:forgetting}, \eqref{eq:epspst}, and the uniform second
moments give
\[
  \big|\E[P_m^{(i-m)}P_m^{(j-m)}]\big|
  < \big|\E[P_iP_j]\big|+Ce^{-cm}
  < Ce^{-c(j-i)}+Ce^{-cm}
  < Ce^{-c(j-i)}.
\]
The diagonal terms $\E[(P_m^{(i-m)})^2]$
are uniformly bounded, and for $j-i>m$, $\E[P_m^{(i-m)}P_m^{(j-m)}]=0$. Summing these bounds yields
\[
  \Var\left(\sum_{k=a}^bP_m^{(k-m)}\right)< C(b-a+1),
  \qquad \forall\; m<a\le b\le n.
\]
    \item[(iii)] By Minkowski's inequality and \Cref{lem:expdecay}, we have
\begin{multline}    \label{eq:PPmdiff}
\E\left[\left( \sum_{k=1}^n P_k - \sum_{k=m+1}^n P_m^{(k-m)} \right)^2\right]^{1/2}
\le \sum_{k=1}^m \E[P_k^2]^{1/2} + \sum_{k=m+1}^n \E[(P_k-P_m^{(k-m)})^2]^{1/2} \\ < Cm + Cne^{-cm} = o(\sqrt{n}).    
\end{multline}
Consequently, \eqref{eq:sigma} gives
$\Var\big(\sum_{k=m+1}^n P_m^{(k-m)}\big)/(n-m)\to\sigma^2$.
For now we assume that $\sigma^2>0$.
    \item[(iv)] We have $m^3/(n-m)\to0$.
\end{itemize}  
These verify all the hypotheses of the main theorem of \cite{BerkMDependent}
with parameter $\delta=2$.  Thus we conclude that
\[
  \frac{1}{\sqrt{n-m}}\sum_{k=m+1}^n P_m^{(k-m)} \law \mathcal N(0,\sigma^2).
\]
With \eqref{eq:PPmdiff} we then have that
\begin{equation}  \label{eq:Pkcovn}
  \frac{1}{\sqrt n}\sum_{k=1}^nP_k\law \mathcal N(0,\sigma^2).
\end{equation} 

In the case where $\sigma^2=0$, we would simply have $\frac{1}{\sqrt{n}}\sum_{k=1}^n P_k\to 0$ in distribution by 
\eqref{eq:sigma}. In particular, \eqref{eq:Pkcovn} still holds.

Now since $\sum_{k=1}^nP_k=-\log X_n(W_n)+\E\log X_n(W_n)$, by \eqref{eq:mean-stabilization}, we have
\[
\frac{1}{\sqrt n}(-\log X_n(W_n) -\kappa_*n )\to \mathcal N(0,\sigma^2),
\]
in distribution.
To extend this beyond convergence along integers, we note that for any $t>0$ and $n=\lfloor t\rfloor$, we have
\begin{equation}   \label{eq:integer-time-reduction}
  \E[ |\log X_t(W_t)-\log X_n(W_n)|^2]
  < C(1+e^{t-n})<C,
\end{equation}
which follows from \eqref{eq:forward-ratio-one} and the inequality $x^2<C(e^x+e^{-x})$. Thus Theorem \ref{thm:tree-clt} follows.
\end{proof}

\section{Two walks and positivity of the variance}
\label{sec:two-walkers}
In this section, we continue using the same setup as the previous section, and work on the $d$-regular tree $\T=(V,E)$. We complete the proof of the central limit theorems needed for the Gaussian cutoff profile.
This consists of two parts: 
in Section \ref{ssec:clt_two_walkers}, we extend the central limit theorem in Theorem \ref{thm:tree-clt} to two walks that are conditionally independent given the Poisson field $\Pi$ (as discussed after \Cref{thm:intro-tree-clt}); in Section \ref{ssec:positivity_variance}, we further prove that the limiting variance $\sigma^2$ is strictly positive.

\subsection{Joint Gaussian limit for two walks}\label{ssec:clt_two_walkers}
We consider two walks that are independent conditional on the same Poisson point process $\Pi$. The joint limit (of the logarithms of the weights, under appropriate scaling) consists of two independent normal random variables.
The key observation is that the main contribution to the weights at large times for the two walks comes from disjoint subtrees, on which they evolve independently.

We start by introducing restricted weights. Given a subset $H \subset V$, we define the restricted weights: for  $s<t$ and $u, v\in V$,
\begin{equation}\label{eq:restricted_kernel}
    w_{s,t}^H(u,v)=\sum_{\gamma\in \mathcal C_{s,t}[u,v]:\gamma(r)\in H,\, \forall\, r\in [s,t]}  w(\gamma).
\end{equation}
In other words, we only sum over the paths that lie entirely in the set $H$. The following lemma compares $w_{s,t}^H(u,v)$ with the full kernel using an escape-time bound. 

Recall that $W$ is the random walk generated by $\Pi$ with time interval $[0, \infty)$, starting from $o$.
\begin{lemma}\label{lem:separated-branches}
Take any $H\subset V$, and let
$\tau:=\inf\{r\ge0: W_r\notin H\}$ be its escape time from $H$. 
Then for all
$0<\vartheta<1$ and $t>0$, we have
\begin{equation}\label{eq:killed-kernel-markov}
  \P\left[
    w_{0,t}^H(o, W_t)<\vartheta  w_{0,t}(o, W_t) \right]
  \le
\frac{\P\left[\tau\le t\right]}{1-\vartheta}.
\end{equation}
\end{lemma}

\begin{proof}
The law of $W_t$ conditional on $\Pi$ is exactly equal to $w_{0,t}(o,\cdot)$.  Using this fact and applying \Cref{lem:sb-markov} with
$p(\cdot)=w_{0,t}(o,\cdot)$ and $q(\cdot)=w_{0,t}^H(o,\cdot)$, we get 
\[
\begin{split}
  \P\left[w_{0,t}^H(o, W_t)<\vartheta  w_{0,t}(o, W_t) \right]
   \le &
  \frac1{1-\vartheta}
  \E\sum_{v\in V}\bigl(w_{0,t}(o,v)-w_{0,t}^H(o,v)\bigr)
=  \frac{\P[\tau\le t]}{1-\vartheta},
\end{split}
\]
where we have used the definition in \eqref{eq:restricted_kernel} that  a path $\gamma$ contributes to $w_{0,t}(o,v)$ but not to $w_{0,t}^H(o,v)$ if and only if it exits $H$ before $t$. Thus \eqref{eq:killed-kernel-markov} is proved.
\end{proof}
We need another preparatory lemma on the separation of two walks.
\begin{lemma}  \label{lem:twosepe}
Let $u,v\in V$ with $D:=\dist(u,v)\ge 3$.  There are disjoint connected subsets $H_u,H_v\subset V$, containing $u,v$, respectively, such that
\begin{equation}\label{eq:branch-geometry}
  \dist(u,H_u^c), \dist(v,H_v^c)
  \ge\lfloor D/3\rfloor,
  \qquad
  \dist(H_u,H_v)\ge D/3-3.
\end{equation}
Moreover, let $W^u$ and $W^v$ be continuous-time random walks starting from $u$ and $v$, respectively, each with edge jump rate $1/(2d)$, and let $\tau_u:=\inf\{r\ge0: W_r^u\notin H_u\}$ and $\tau_v:=\inf\{r\ge0: W_r^v\notin H_v\}$. Then we have
\begin{equation}\label{eq:branch-exit}
  \P[\tau_u<\infty]
  +
  \P[\tau_v<\infty]
  \le C(d-1)^{-D/3}.
\end{equation}
\end{lemma}
\begin{proof}
Write the geodesic between $u,v$ as $[u,v]=(u_0,u_1,\ldots,u_D)$ and put
$m=\lfloor D/3\rfloor$.  Delete the edges
\[
  \{u_{m-1},u_m\},
  \qquad
  \{u_{D-m},u_{D-m+1}\},
\]
and let $H_u$ and $H_v$ be the components containing $u_0=u$ and
$u_D=v$.  Then
\[
  \dist(u,H_u^c)=\dist(v,H_v^c)=m,
  \qquad
  \dist(H_u,H_v)=D-2m+2,
\]
which proves \eqref{eq:branch-geometry}.
For $W^u$, until it hits $u_m$, put $\zeta_s:=\dist(W^u_s, u_m)$.  
It is then a biased random walk on $\Z_{\ge 0}$. By the gambler's ruin formula, we get
\[
  \P[\tau_u<\infty]
  =\P[\zeta_s=0 \text{ for some }s]
  =(d-1)^{-m}.
\]
The same holds for $\tau_v$, so that we get \eqref{eq:branch-exit}.
\end{proof}

With \Cref{lem:separated-branches,lem:twosepe}, we are now ready to prove the central limit theorem for two walks. 

\begin{theorem}\label{thm:two-walker-clt}
Let $W^1, W^2$ be random walks generated by $\Pi$ on the time interval $[0, \infty)$, starting from $o$, and independent conditional on $\Pi$.
Then as $t\to\infty$, 
$$\frac{1}{\sqrt{t}}\left(-\log X_t(W_t^1)-\kappa_*t, -\log X_t(W_t^2)-\kappa_*t\right)\law \mathcal{N}(0, \sigma^2I_2), $$  
which corresponds to two independent centered Gaussian random variables with variance $\sigma^2$.
\end{theorem}

\begin{proof}
For any $t>1$, take $L\in [\lfloor\log(t)^2\rfloor, \lfloor\log(t)^2\rfloor+1)$ such that $t-L\in \Z$.
The proof consists of two steps. We first prove 
the joint convergence with $(X_t(W_t^1), X_t(W_t^2))$ replaced by $(w_{L,t}(W_L^1, W_t^1), w_{L,t}(W_L^2, W_t^2))$. In the second step we show that this replacement has a negligible effect after scaling by $\sqrt{t}$.

\noindent\textbf{Step 1: restarted polymers.}
By \eqref{eq:distance-lyapunov2}, we have 
\begin{equation}    \label{eq:sLbd}
\P[\mathcal S_L]\ge 1- e^{-cL},
\end{equation}
where $\mathcal S_L:=\{\dist(W_L^1,W_L^2)> (d-2)L/(2d)\}$.
On $\mathcal S_L$, apply \Cref{lem:twosepe} to
$W_L^1,W_L^2$.  For large $t$, this gives $H_1,H_2\subset V$ with
\[
\dist(H_1, H_2)\ge 2,\qquad  \dist(W_L^i,H_i^c)\ge cL,
  \qquad i=1,2.
\]
For $i=1,2$, let $\tau_i=\inf\{r\ge L: W_r^i\not\in H_i\}$.
Then by \eqref{eq:branch-exit}, we have $\P[\tau_i \le t \,|\, W^1_L, W^2_L]<Ce^{-cL}$.
Applying
\Cref{lem:separated-branches} with $\vartheta=1/2$, we thus get that, under $\mathcal S_L$,
\begin{equation}\label{eq:killed-full-comparison}
  \P\left[
      \left| - \mathds{1}[\tau_i \geq t]\log w_{L,t}^{H_i}(W_L^i,W_t^i) + \log w_{L,t}(W_L^i,W_t^i) \right|>\log(2)
 \,\middle| W^1_L, W^2_L
  \right]
  \le Ce^{-cL}.
\end{equation}
By this and \Cref{thm:tree-clt} applied to the time interval $[L, t]$, we get that as $t\to\infty$,
\[
\frac{- \mathds{1}[\tau_i \geq t]\log w_{L,t}^{H_i}(W_L^i,W_t^i) - \kappa_*(t-L)}{\sqrt{t-L}} \,\bigg| W^1_L, W^2_L \law \mathcal{N}(0, \sigma^2), \quad i=1,2,
\]
with a uniform convergence rate over almost surely every pair $W^1_L, W^2_L$ under $\mathcal S_L$.
  Since the  variables on the left-hand side are independent (conditional on $W^1_L, W^2_L$) for $i=1,2$, we get that the joint limit for $i=1,2$ is $\mathcal{N}(0, \sigma^2I_2)$.
Then by \eqref{eq:killed-full-comparison} again, and also using \eqref{eq:sLbd}, we see that as $t\to\infty$,
\begin{equation}\label{eq:two-segment-clt}
  \frac1{\sqrt{t-L}}
  \left(
    -\log w_{L,t}(W_L^1, W_t^1)-\kappa_*(t-L),
    -\log w_{L,t}(W_L^2, W_t^2)-\kappa_*(t-L)
  \right)
  \law \mathcal N(0,\sigma^2I_2).
\end{equation}

\noindent\textbf{Step 2: bound for replacement.}
 For $i=1,2$, and $s\ge1$, let
\[
  Q_s^i
  :=
  -\log\frac{X_s(W_s^i)}{X_{s-1}(W_{s-1}^i)},
  \qquad
  Q_s^{i,(L)}
  :=
  -\log\frac{
    w_{L,L+s}(W_L^i,W_{L+s}^i)
  }{
    w_{L,L+s-1}(W_L^i,W_{L+s-1}^i)
  }.
\]
Like \eqref{eq:cross-ratio}, we define 
$\xi_{L,s}^i= Q_{L+s}^i-Q_s^{i,(L)}$.
We then write 
\begin{equation*}
\begin{split}
  \Delta_i
  &:=-\log X_t(W_t^i)+\log X_L(W_L^i)
    +\log w_{L,t}(W_L^i,W_t^i)
  =\sum_{k=1}^{t-L}
    \left(Q_{L+k}^i-Q_k^{i,(L)}\right)
  =\sum_{k=1}^{t-L}\xi_{L,k}^i.
\end{split}
\end{equation*}
By \Cref{lem:expdecay} and Minkowski's inequality, we have
\[
\big(\E\Delta_i^2\big)^{1/2} \le  \sum_{k=1}^{t-L} \big(\E (\xi_{L,k}^i)^2\big)^{1/2} \leq C \sum_{k=1}^{t-L} e^{-c k}  < C'.
\]
Thus $\Delta_i/\sqrt{t}$ converges to 0 in probability as $t\to\infty$.
Moreover, by \Cref{thm:tree-clt}, $(-\log X_L(W_L^i)-\kappa_*L)/\sqrt{t}$ also converges to 0 in probability.
Consequently,  with \eqref{eq:two-segment-clt} and
\[
  -\log X_t(W_t^i)-\kappa_*t
  =
  -\log w_{L,t}(W_L^i, W_t^i) -\kappa_*(t-L)-\log X_L(W_L^i)-\kappa_*L
  +\Delta_i,
\]
the conclusion follows.
\end{proof}

\subsection{Non-degeneracy of $\sigma^2$}
\label{ssec:positivity_variance}
Given the stabilization estimates in \Cref{lem:covariance}, it would actually suffice to show that the variance of $\log X_t(W_t)$ is unbounded as $t\to\infty$; this is the content of \Cref{lem:variance-unbounded} below.
The idea is to use a Mermin-Wagner type argument, where the Poisson rate in a tube is slightly tuned.

For this, we need an alternative description of the Poisson point process $\Pi$, and its generated walk $W$ with time interval $[0, \infty)$, starting from $o$.
For any path $f:[0,\infty)\to V$, set
\[
  \mathcal R(f)
  :=
  \left\{
    (e, t)\in E\times[0, \infty): f(t)\in e
  \right\}.
\]
\begin{lemma}\label{lem:palm-law}
Conditional on the walk $W:[0, \infty)\to V$,  $\Pi$ is the union of $\{(\{W_{s-}, W_s\}, s): W_{s-}\neq W_s\}$, and an inhomogeneous Poisson point process on $E\times [0,\infty)$, with rate $1/(2d)$ on $\mathcal R(W)$, and rate $1/d$ on $E\times [0,\infty)\setminus \mathcal R(W)$.
\end{lemma}

\begin{proof}
The restriction of $\Pi$ to $E\times [0,\infty)$ can be written as $\Pi^+\cup\Pi^-$, the union of two independent Poisson point processes on $E\times [0, \infty)$ (each with rate $1/(2d)$). 
We can construct $W$, such that for any $r>0$ with $W_r\neq W_{r-}$, it holds that $(\{W_r, W_{r-}\},r) \in \Pi^+$; and for any $r>0$ with $W_r=W_{r-}$, $(\{W_r, v\}, r)\not\in \Pi^+$ for any $v\sim W_r$. 
Thus given $W$, $\Pi^-$ and $\Pi^+$ are still independent, with $\Pi^-$ still being a rate $1/(2d)$ Poisson point process on $E\times [0, \infty)$, and $\Pi^+\setminus \mathcal R(W)$ being a rate $1/(2d)$ Poisson point process on $E\times [0, \infty)\setminus \mathcal R(W)$.
The set $\Pi^+\cap \mathcal R(W)$ is determined by $W$.
Thus the conclusion follows.
\end{proof}

We now use the planted description to couple the original clock rate to a
slightly slowed rate, to force an unbounded logarithmic
variance. 
Let $d_{\rm TV}$ denote the total variation distance between the laws of two random variables. Then using Pinsker's inequality, for any $0<\lambda_1<\lambda_2$, one can bound
\begin{equation}\label{eq:dtvmunu}
      d_{\rm TV}(\Poi(\lambda_1),\Poi(\lambda_2)) < C(\lambda_2-\lambda_1)/\sqrt{\lambda_1}.
\end{equation}
\begin{lemma}\label{lem:variance-unbounded} We have
$\sup_{t\in \mathbb{N}}
  \Var\bigl(\log X_t(W_t)\bigr)=\infty$.
\end{lemma}

\begin{proof}
We take four steps. First, we introduce three different joint laws of a random walk and Poisson clocks. Then in the second and third steps we compare these different laws. Finally we use these comparisons to deduce the unboundedness of variance.

\noindent\textbf{Step 1: the slowed and tube planted laws.}
Take any parameters $\alpha$ and $\beta$, satisfying
\[
  0<\alpha<\frac1{20},
  \qquad
  2\alpha<\beta<\frac14-\frac32\alpha.
\]
Take a large enough $t$, choose $\delta=\delta_t$ so
that
\[
  \delta\asymp t^{-1+\alpha},
  \qquad 2d\delta t\in\Z,
\]
and put
\[
  r=\lfloor\beta\log_{d-1}(t)\rfloor,
  \qquad
  \ell^{\rm tube}:=t^{1/2+\alpha}.
\]  
Using these parameters, we construct three joint laws of a random walk and Poisson clocks: original, slowed, and tube laws. Let
$(\Pi^\circ,W^\circ)$ have the same law as $(\Pi, W)$, and
$(\Pi^{\mathrm{slow}},W^{\mathrm{slow}})$ have the same law as
\[
\left(\{(e, s)\in E\times \R: (e, (1-2\delta d)s)\in \Pi\}, \quad s\mapsto W_{(1-2\delta d)s}\right).
\]
In words, $\Pi^{\mathrm{slow}}$ is a Poisson point process on $E\times \R$ with rate $1/d-2\delta$, and $W^{\mathrm{slow}}$ is generated by $\Pi^{\mathrm{slow}}$ on $[0, \infty)$, starting from $o$.

We next construct the tube law $(\Pi^{\rm tube}, W^{\rm tube})$, from the perspective provided by \Cref{lem:palm-law}.
For any  $v\in V$, define the space-time tube and its incident-edge
expansion by (recall tubes from \Cref{defn:splspt})
\begin{equation}\label{eq:variance-tube-definition}
  \mathcal V_v
  :=\mathcal V_{0,t;\ell^{\rm tube}}(o,v;2r),
  \qquad
  \mathcal U_v:=
  \left\{
    (e,s)\in E\times\R:
    \exists\,u\in e\text{ with }(u, s)\in\mathcal V_v
  \right\}.
\end{equation}
We set $W^{\rm tube}$ to be a random walk on $\T$ with jump
 rate $1/(2d)-\delta$ across each edge.  
Conditional on $W^{\mathrm{tube}}$, let $\Pi^{\rm tube}$ be the union of $\{(\{W_{s-}^{\mathrm{tube}}, W_s^{\mathrm{tube}}\}, s): W_{s-}^{\mathrm{tube}}\neq W_s^{\mathrm{tube}}\}$, and a Poisson point process on $E\times \R$ with rate
\begin{equation}\label{eq:tube-law-intensity}
  \begin{cases}
    1/(2d)-\delta,
      &\text{ in }\mathcal R(W^{\mathrm{tube}}),\\
    1/d-2\delta,
      &
       \text{ in } \mathcal U_{W_t^{\mathrm{tube}}}
        \setminus \mathcal R(W^{\mathrm{tube}}),\\
    1/d,
      & \text{ in } (E\times\R) \setminus (\mathcal U_{W_t^{\mathrm{tube}}}
        \cup \mathcal R(W^{\mathrm{tube}}) ).
  \end{cases}
\end{equation}
This auxiliary tube law $(\Pi^{\rm tube}, W^{\rm tube})$ is the intermediary in the comparison.

For $\bullet\in\{\circ,\mathrm{slow},\mathrm{tube}\}$, we further write $X_t^\bullet(v):=w_{0,t}^{\Pi^\bullet}(o,v)$.

\noindent\textbf{Step 2: original versus tube.}
For the walks $W^\circ$ and $W^{\mathrm{tube}}$, their numbers of jumps in time $[0, t]$ are $\Poi(t/2)$ and $\Poi(t/2-d\delta t)$, respectively.
The laws of $W^\circ|_{[0,t]}$ and $W^{\mathrm{tube}}|_{[0,t]}$, conditional on the same number of jumps, are the same.
Then by the  total-variation distance bound  \eqref{eq:dtvmunu}, we can couple them such that
\begin{equation}\label{eq:walker-rate-coupling}
  \P\big[
    W^\circ|_{[0,t]}\ne W^{\mathrm{tube}}|_{[0,t]}
  \big]
  \le C\delta\sqrt t.
\end{equation}
We next consider $\Pi^\circ$ and $\Pi^{\mathrm{tube}}$ on $\mathcal U_{W_t^{\mathrm{tube}}}$. 
We have that $\P[\dist(o, W_t^{\rm tube})>2t]<Ce^{-ct}$, using  Chernoff bounds.
Thus the total measure of $(\mathcal U_{W_t^{\rm tube}} \cup \mathcal R(W^{\mathrm{tube}}))\cap (E\times [0,t])$ (with counting measure on $E$ and Lebesgue measure on time) is
\[
\int_0^t
    \big|\big\{e:(e,s)\in\mathcal U_{W_t^{\rm tube}}\cup \mathcal R(W^{\mathrm{tube}}) \big\}\big|\,\dd s< C(t+\ell^{\rm tube}\dist(o, W_t^{\rm tube}))(d-1)^{2r},
\]
which, with probability $>1-Ce^{-ct}$, is bounded by $Ct^{3/2+\alpha+2\beta}$.
Then combining \Cref{lem:palm-law} for $\Pi^\circ$ and using \eqref{eq:dtvmunu} (applied separately to the Poisson point processes on 
 $\mathcal U_{W_t^{\rm tube}}\setminus\mathcal R(W^{\mathrm{tube}})$ and 
$\mathcal R(W^{\mathrm{tube}})$), under the coupling of \eqref{eq:walker-rate-coupling} and the event $\dist(o, W_t^{\rm tube})\le 2t$, we can couple $(\Pi^{\circ}, W^{\circ})$ and $(\Pi^{\rm tube}, W^{\rm tube} )$, such that
\begin{multline}   \label{eq:original-tube-tv}
\P\big[ (\Pi^{\circ}\cap (E\times [0,t]), W^{\circ}|_{[0,t]})\neq 
      (\Pi^{\rm tube}\cap (E\times [0,t]), W^{\rm tube}|_{[0,t]} ) \big] \\ < Ce^{-ct} + C\delta\sqrt{t} + C\delta\sqrt{t^{3/2+\alpha+2\beta}}<Ct^{-1/4+3\alpha/2+\beta}=o(1).
\end{multline}

\noindent\textbf{Step 3: slow versus tube.}
As $  W^{\mathrm{slow}}$ and $W^{\mathrm{tube}}$ have the same jump rates, we couple them so that
$W^{\mathrm{slow}}=W^{\mathrm{tube}}$. 
Let $\Pi^+$ be a rate $2\delta$ Poisson point process on $E\times \R$, independent of $(\Pi^{\rm slow}, W^{\rm slow})$.
Then from the construction of $\Pi^{\rm tube}$, and \Cref{lem:palm-law}, conditional on $W^{\mathrm{slow}}$, we have
$\Pi^{\mathrm{tube}}=\Pi^{\mathrm{slow}}\cup(\Pi^+\setminus (\mathcal U_{W_t^{\mathrm{slow}}} \cup \mathcal R(W^{\mathrm{slow}})))$ in distribution. 
We can couple them so that the equality holds almost surely.

For every $v\in V$, define the retained weight
\[
  X_t^{\rm ret}(v)
  :=
  \sum_{\substack{
    \gamma \in \mathcal C_{0,t}[o,v; \Pi^{\rm slow}] \\
    \gamma(s)\not\in e, \, \forall (e,s)\in \Pi^+ \setminus \mathcal U_v, \, 0\le s \le t
  }}
  w^{\rm slow}(\gamma)\le X_t^{\mathrm{slow}}(v),
  \qquad \text{thus }
  \sum_{v\in V}X_t^{\rm ret}(v)\le 1.
\]
At the endpoint
$W_t^{\mathrm{slow}}=W_t^{\mathrm{tube}}$, every retained path avoids $\Pi^+\setminus \mathcal U_{W_t^{\mathrm{slow}}}$, so it has the same weight in $\Pi^{\rm tube}$.  Hence
\begin{equation}\label{eq:retained-below-both}
  X_t^{\rm ret}(W_t^{\mathrm{slow}})
  \le
  X_t^{\rm tube}(W_t^{\mathrm{tube}}).
\end{equation}
Using the fact that $W^{\mathrm{slow}}$ is sampled from the weight $w^{\rm slow}$ and recalling the definition \eqref{eq:variance-tube-definition}, we have
\begin{equation}\label{eq:retained-mass-loss}
\begin{split}
  \E\Big[\sum_{v\in V}\bigl(X_t^{\rm slow}(v)-X_t^{\rm ret}(v)\bigr)\Big]
  &=  \P[\exists (e,s)\in \Pi^+\setminus\mathcal U_{W^{\rm slow}_t},\, 0\le s \le t:\, W^{\rm slow}_s\in e]
  \\
  &\le
  \E\Big[
    1-\exp\Big(
      -2d\delta\int_0^t
      \mathds{1}[(W_s^{\rm slow},s)
        \notin\mathcal V_{W_t^{\rm slow}}]
      \,\dd s
    \Big)
  \Big]\\
  &\le
  2d\delta\,
  \E\int_0^t
  \mathds{1}[(W_s^{\rm slow},s)
    \notin\mathcal V_{W_t^{\rm slow}}]
  \,\dd s.
\end{split}
\end{equation}
By the occupation
estimate \eqref{eq:tubeint}, applied after a deterministic time change, we can bound \eqref{eq:retained-mass-loss} by
\begin{equation}\label{eq:slow-tube-loss}
  C\delta t
  \big(
    e^{-c(\ell^{\rm tube})^2/t}+(d-1)^{-r/2}
  \big)
  \le
  Ct^\alpha
  \big(e^{-ct^{2\alpha}}+t^{-\beta/2}\big)
  =o(1).
\end{equation}
Applying \Cref{lem:sb-markov} conditionally on $\Pi^{\rm slow}$ and $\Pi^+$,
and then averaging, gives
\begin{equation}\label{eq:retained-markov}
\begin{split}
  &\P\Big[
    X_t^{\rm ret}(W_t^{\rm slow})
    <\frac12X_t^{\rm slow}(W_t^{\rm slow})
  \Big]
  \le
  2\E\Big[\sum_{v\in V}
  \bigl(X_t^{\rm slow}(v)-X_t^{\rm ret}(v)\bigr)\Big]
  =o(1).
\end{split}
\end{equation}

\noindent\textbf{Step 4: unbounded variance.}
Combining \eqref{eq:original-tube-tv},
\eqref{eq:retained-below-both}, and \eqref{eq:retained-markov} yields
\begin{equation}\label{eq:slow-coupling}
  \P[\log X_t^\circ(W_t^\circ) \ge \log X_t^{\rm slow}(W_t^{\rm slow})-\log2]=1-o(1),
\end{equation}
and in particular the probability is greater than $1/2$ for all
large $t$. Note that the slow law at time $t$ equals the original law at time
$
  (1-2d\delta)t.
$
Therefore, for integers $t$ and $(1-2d\delta)t$, \eqref{eq:mean-stabilization} gives
\begin{equation}\label{eq:slow-original-means}
\begin{split}
  \E\log X_t^\circ(W_t^\circ)
  &=-\kappa_*t+O(1),\\
  \E\log X_t^{\rm slow}(W_t^{\rm slow})
  &=-\kappa_*t+2d\kappa_*\delta t+O(1).
\end{split}
\end{equation}
On the event in
\eqref{eq:slow-coupling}, either
\[
  \log X_t^\circ(W_t^\circ)\ge -\kappa_*t+d\kappa_*\delta t-\log2
  \quad\text{or}\quad
  \log X_t^{\rm slow}(W_t^{\rm slow})\le -\kappa_*t+d\kappa_*\delta t.
\]
One of these events has probability at least $1/4$. 
Then by  \eqref{eq:slow-original-means}, we have
\[
\Var(\log X_t^\circ(W_t^\circ)) \vee \Var(\log X_t^{\rm slow}(W_t^{\rm slow})) \gtrsim (\delta t)^2 \asymp t^{2\alpha},
\]
which goes to $\infty$  along any diverging integer sequence of $t$. This finishes the proof.
\end{proof}

We now upgrade the unbounded variance result to the non-degeneracy of $\sigma^2$, using \Cref{lem:covariance}.

\begin{prop}\label{prop:sigma-positive}
The constant $\sigma^2$ in \eqref{eq:sigma} is strictly positive.
\end{prop}

\begin{proof}
We argue by contradiction, and assume that $\sigma^2=0$.  For $i\le j\in \N$, by \Cref{lem:covariance} we have
\[
  \E[P_iP_j]
  =
  e_{j-i}
  +
  O(e^{-c(i\vee(j-i))}).
\]
By summing over all $1\le i, j \le n$, we have 
\[
\Var\bigg(\sum_{i=1}^nP_i\bigg) = O(1) + ne_0+ \sum_{i=1}^{n-1} 2(n-i) e_i.
\]
Using the assumption that $\sigma^2=e_0+2\sum_{i=1}^\infty e_i=0$ and the bound $|e_j|\lesssim e^{-cj}$ from  \eqref{eq:covariance-stabilization}, we have
\[
\Var\bigg(\sum_{i=1}^nP_i\bigg) < O(1) + \sum_{i=1}^\infty 2i|e_i| = O(1).
\]
Thus we conclude that
\[
  \sup_{n\in \N}\Var( \log X_n(W_n)  )
  =
  \sup_{n\in \N}\Var\bigg(\sum_{i=1}^nP_i\bigg)<\infty,
\]
which contradicts \Cref{lem:variance-unbounded}.
\end{proof}

\section{From the tree to a finite \texorpdfstring{$d$}{d}-regular graph}
\label{sec:cover}
In this section we compare repeated averaging on a large finite simple $d$-regular graph and the infinite $d$-regular tree $\T=\T_d=(V, E)$, with root $o$. Let $G_n=(V_n,E_n)$ be a connected simple $d$-regular graph with
$n=|V_n|$, and fix a root $o_n$.   Throughout this section we consider $t \asymp \log n$, or more specifically,
\begin{equation}\label{eq:finite-comparison-window}
  \frac{\log(n)}{2\kappa_{\rm RW}}
  \le t\le
  \frac{2\log(n)}{\kappa_{*}}.
\end{equation}
Let $\cP:\T\to G_n$
be the universal covering map with $\cP(o)=o_n$.  We consider the following Poisson fields:
\begin{enumerate}
\item $\Pi^{\T}$: rate-$1/d$ Poisson point process on $E\times\R$;
\item $\Pi^n$: rate-$1/d$ Poisson point process on $E_n\times\R$;
\item $\widetilde\Pi$: the pullback of $\Pi^n$ through $\cP$.
\end{enumerate}
Due to multiple lifts of one edge through $\cP$, the set $\widetilde\Pi$ is no longer an independent Poisson point process on $E\times \R$, and its intersections with different parts of $\T$ are correlated.

For any one of these environments $\Xi$, let $W^\Xi$ be the walk generated by $\Xi$ on the time interval $[0, \infty)$, starting from $o_\Xi$. 
Here and below, we let $o_{\Pi^n}=o_n$, and $o_{\Pi^{\T}}=o_{\widetilde\Pi}=o$. We then write
\[
  X_t^*(\Xi):=w_{0,t}(o_\Xi,W_t^\Xi;\Xi).
\]
We will also consider the two-walk case: let $W^{1,\Xi},W^{2,\Xi}$ be independent copies of $W^{\Xi}$, conditional on $\Xi$. We then write, for $i=1, 2$,
\[
  X_t^{*,i}(\Xi):=w_{0,t}(o_\Xi,W_t^{i,\Xi};\Xi).
\]
We want to compare $(X_t^{*,1}(\Pi^\T), X_t^{*,2}(\Pi^\T))$ with  $(X_t^{*,1}(\Pi^n), X_t^{*,2}(\Pi^n))$, and $(X_t^{*,1}(\widetilde\Pi), X_t^{*,2}(\widetilde\Pi))$ will act as the intermediary process.

\subsection{Setup and preliminary estimates}   \label{ssec:setupprees}

In this subsection, we introduce some basic notation and the setup, and prove estimates to be used repeatedly in later subsections. 

\smallskip

\noindent\textbf{General graph notations.}
For a connected graph $H=(V^H, E^H)$, and $v\in V^H$, $m>0$, we use $B_m(v)$ to denote the subgraph generated by all vertices $u$ with $\dist(u,v)\le m$.
For any $\mathcal A\subset V^H\times \R$, we denote its incident space-time edge set by
\[
  \mathcal R(\mathcal A)
  :=
  \left\{
    (e,r)\in E^H\times\R:
    \exists\,v\in e\text{ with }(v,r)\in\mathcal A
  \right\}.
\]
For $f:I\to V^H$ where $I\subset \R$ is an interval, we also write
$\mathcal R(f)=\mathcal R(\{(f(s),s): s\in I\})$.
For $\mathcal A\subset V^H\times \R$ or $\mathcal A\subset E^H\times \R$, we use $\mathcal A_s$ to denote its time-$s$ slice,  i.e., $$ 
\mathcal A_s=\{v\in V^H\, (\mbox{or }E^H): (v,s)\in \mathcal A\}.
$$
If $H$ is a finite graph, write $\tx(H)=|E^H|-|V^H|+1$ for its tree excess.
\smallskip

\noindent\textbf{Parameters and assumptions.}
We next provide some setup that will be used in proofs in later subsections. The choices of many parameters here are motivated by later proofs. Readers may skim this setup for now and revisit it while reading the next two subsections.

Take four intermediate rates
\begin{equation}\label{eq:intermediate-rates}
  \kappa_i
  :=
  \frac{i\kappa_{\rm RW}+(5-i)\kappa_*}{5},
  \qquad i=1,2,3,4.
\end{equation}
Since $\kappa_*<\kappa_{\rm RW}$, we have $\kappa_*<\kappa_1<\kappa_2<\kappa_3<\kappa_4
  <\kappa_{\rm RW}$.
Fix a constant $\lambda_{\rm sp}>0$, the uniform lower
bound on the combinatorial Laplacian gap. Let $c_*>0$ be a small constant and define:
\begin{equation}\label{eq:detour-times}
\begin{split}
  L_n=\lfloor c_*\log(n)\rfloor,
  \qquad
  T_{\rm ini}=100L_n,
  \qquad
  &T_1=L_n/10,
    \qquad
  T_2^-=\frac{\log(n)}{\kappa_4},
  \qquad
  T_2=\frac{\log(n)}{\kappa_3},
\\
    R_n=\left\lfloor\frac15\log_{d-1}(n)\right\rfloor,&
  \qquad
  K_n=\left\lfloor8(\log(n))^{0.49}\right\rfloor.
\end{split}
  \end{equation}
Notice that $T_2^-<T_2$.
We choose $c_*$ sufficiently small depending only on $d$, such that
\[
  T_{\rm ini}+T_1<\frac{\log(n)}{2\kappa_{\rm RW}}\le t,
  \qquad
  T_{\rm ini}<R_n/2.
\]
We also fix constants
\begin{equation}\label{eq:coarse-interval-scale}
\begin{gathered}
  C_{\rm cp}>10+10\log(d-1),\\
  0<\varrho<\frac18\min\left\{
    \frac{c_*}{10},
    \frac{1}{\kappa_3}-\frac{1}{\kappa_4},
    \frac{1}{\kappa_2}-\frac{1}{\kappa_3},
    \frac{(\kappa_2-\kappa_1)c_*}{10C_{\rm cp}}
  \right\}.
\end{gathered}
\end{equation}
Implicit constants and $C,c>0$ in the remainder of this section may depend on the constants $\lambda_{\rm sp}, c_*, C_{\rm cp}, \varrho$.

\begin{definition}\label{def:good-graph}
For $\delta>0$, the simple $d$-regular graph $G_n$ with root $o_n$ is
$\delta$-good if:
\begin{enumerate}[label=(\arabic*)]
\item $\tx(B_{R_n}(v))\le1$ for every $v\in V_n$;
\item $B_{K_n}(o_n)$ is a tree, i.e., $\tx(B_{K_n}(o_n))=0$;
\item $\lambda_{2,n}\ge\lambda_{\rm sp}$, where $\lambda_{2,n}$ is the second-smallest eigenvalue of the combinatorial Laplacian;
\item the time $s$ random walk transition probability $p^n_s$ of the random walk on $G_n$ (with edge jump rate $1/(2d)$) satisfies
\begin{align}
\sum_{u\in V_n}\bigl(p_s^n(v,u)-1/n\bigr)\vee 0
  &\le\delta,
  &&\forall\, v\in V_n,\quad s\ge T_2^-,
  \label{eq:TV-mixing}\\
\sum_{u \in V_n}
  \bigl(p_s^n(v,u)-e^{-\kappa_2s}\bigr)\vee 0
  &\le\delta,
  &&\forall\, v\in V_n,\quad T_1< s\le T_2+2\varrho \log n+4.
  \label{eq:entropic-mixing}
\end{align}
\end{enumerate}
\end{definition}
These conditions are tailored for random $d$-regular graphs (see \Cref{prop:entropic-heat-kernel} below). 
Roughly speaking, the first two conditions state that $G_n$ is locally tree-like, and the third condition gives the uniform spectral gap.
The last condition concerns the global geometry: the random walk transition probability decays at an exponential rate of at least $\kappa_2$ until it is nearly uniform. 
The fact that $\kappa_*<\kappa_{\rm RW}$ (proved in Proposition \ref{prop:kappa-star}) is essential for such a decay condition.

\smallskip

\noindent\textbf{Preliminary random walk estimates.}
We then give some results on random walks on $d$-regular graphs that will be used repeatedly.
Below we always assume conditions (1) and (2) in \Cref{def:good-graph}, and that $t$ satisfies
\eqref{eq:finite-comparison-window}.

Because $T_{\rm ini}<R_n/2$, item~(1) of
\Cref{def:good-graph} implies that $B_{T_{\rm ini}}(o_n)$ contains
at most one cycle.  If that cycle exists, choose on it a vertex
$u_*$ with maximum distance to $o_n$.  Let
\begin{equation}\label{eq:initial-tree-component}
  \mathcal T_{\rm ini}
  :=
  \begin{cases}
  \text{the component containing $o_n$ of }
  B_{T_{\rm ini}}(o_n)\setminus\{u_*\},
    &\text{if the cycle exists},\\
  B_{T_{\rm ini}}(o_n),&\text{otherwise}.
  \end{cases}
\end{equation}
We then
let $\widetilde{\mathcal T}_{\rm ini}$ be the connected component containing
$o$ of $\cP^{-1}(\mathcal T_{\rm ini})$. 
Item~(2) in \Cref{def:good-graph} gives $\dist(o_n,u_*)>K_n$ whenever $u_*$ is present.
 The map $\cP$ restricted
to $\widetilde{\mathcal T}_{\rm ini}$ is an isomorphism onto
$\mathcal T_{\rm ini}$.

Denote
\[
  x_n:=(\log(n))^{0.51},
  \qquad
y_n:=K_n/128=\left\lfloor8(\log(n))^{0.49}\right\rfloor/128.
\] 
For $t>0$ and $v\in V$, define (recall paths contained in a tube from \Cref{defn:splspt})
\[
\mathcal V_t^{\rm str}(v) = \mathcal V_{0,t;x_n}(o,v;2y_n) \setminus \{ (u, s): u\not\in\widetilde{\mathcal T}_{\rm ini},\; 0\le s\le T_{\rm ini}+T_1 \}.
\]
For
$\Xi\in\{\Pi^\T,\widetilde\Pi\}$, set
\begin{equation}\label{eq:structural-path-set}
\mathcal C_t^{\rm str}[v;\Xi]:= \{\gamma\in \mathcal C_{0,t}[o,v;\Xi]: (\gamma(r), r)\in\mathcal V_t^{\rm str}(v),\; \forall r\in [0,t] \},
\end{equation}
\[
  S_t(v;\Xi):=\sum_{\gamma\in\mathcal C_t^{\rm str}[v;\Xi]}
    w(\gamma;\Xi),
  \qquad
  S_t^{*}(\Xi):=S_t(W_t^{\Xi};\Xi),
  \qquad
  S_t^{*,i}(\Xi):=S_t(W_t^{i,\Xi};\Xi),
  \quad i\in\{1,2\}.
\]
Our estimates are as follows.
\begin{itemize}
\item By considering the  distance of the walk to the boundary of 
$\mathcal T_{\rm ini}$, we have
\begin{equation}   \label{eq:WsPi}
        \P[B_{K_n/2}(W_s^{\Pi^n}) \subset {\mathcal T}_{\rm ini},\, \forall\, 0\le s\le T_{\rm ini}+T_1]> 1- Ce^{-cK_n}.
\end{equation}
\item 
For each $v\in V_n$, let $W^v$ be a walk on $G_n$ with rate $1/(2d)$ over each edge, started at $v$. Since we have assumed that every $R_n$-ball has tree excess at most one, for
$r<R_n/10$,  the following two estimates hold:
\begin{align}
  \sup_{v\in V_n}\P\bigl[B_{r}(W^v_{8r})\text{ contains a cycle}\bigr]
  &< Ce^{-cr},
  \label{eq:root-after-burnin}\\
 \sup_{v\in V_n: B_{4r}(v) \textrm{ is a tree}} \P\left[
    \begin{array}{c}
    B_r(W_s^v)\text{ contains a cycle for some }s\le R_n/2,\\
    \text{or }B_{4r}(W_{R_n/2}^v)\text{ contains a cycle}
    \end{array}
  \right]
  &< Ce^{-cr},
  \label{eq:root-propagation}
\end{align}
These two estimates can be proved in the same way as in \cite[Lemma 3.2]{LS}, by analyzing the birth-death chain evolution of the distance to the nearest cycle and then applying Chernoff bounds.
Applying  \eqref{eq:root-after-burnin} with $r=8L_n$ (the choices in \eqref{eq:detour-times} ensure that $8L_n<R_n/10$), then \eqref{eq:root-propagation} with
$r=2L_n$, and iterating over the $O(1)$ intervals of length $R_n/2$
covering $[T_{\rm ini},t]$, we get
\begin{equation}  \label{eq:B2LWs}
    \P \left[  B_{2L_n}(W_s^{\Pi^n})\text{ is a tree for every }
  T_{\rm ini}\le s\le t\right] > 1-Ce^{-cL_n}.
\end{equation}
    \item We have
\begin{equation}  \label{eq:strone}
\P[W^{\Pi^\T}|_{[0,t]}\in\mathcal C_t^{\rm str}[W_t^{\Pi^\T};\Pi^\T]] > 1-Ce^{-c(\log(n))^{0.02}}.
\end{equation}
Indeed, by \eqref{eq:WsPi} and \Cref{lem:tube}, we get that the left-hand side of \eqref{eq:strone} is at least
\[
1 - C (e^{-cx_n^2/t}+Ct (d-1)^{-y_n/2}+C e^{-cK_n} )> 1- C e^{-c (\log(n))^{0.02}}.
\]
\end{itemize}

\subsection{Independent tree field versus pullback field}\label{ssec:independent_pullback}

In this subsection, we compare walks generated from $\Pi^\T$ and $\widetilde\Pi$, including multiple walks that are independent conditional on $\Pi^\T$ and $\widetilde \Pi$.
We next show that the restricted weights induced from  $\Pi^{\T}$ and $\widetilde \Pi$ can be coupled with high probability.  

\begin{lemma}
\label{lem:structural-corridors}
Suppose that $G_n$ with root $o_n$ satisfies (1) and (2) of
\Cref{def:good-graph} and that $t$ satisfies
\eqref{eq:finite-comparison-window}.  One may couple $\Pi^\T$ and
$\widetilde\Pi$, together with two walks $W^{1,\Pi^\T}, W^{2,\Pi^\T}$ and $W^{1,\widetilde\Pi}, W^{2,\widetilde\Pi}$, such that
\begin{equation}\label{eq:structural-restricted-coupling}
  \P\left[
    \exists i\in\{1,2\}:
    S_t^{*,i}(\Pi^\T)\ne S_t^{*,i}(\widetilde\Pi)
  \right]
  \le Ce^{-c(\log(n))^{0.02}}.
\end{equation}
\end{lemma}

\begin{proof}
Via a straightforward generalization of Lemma \ref{lem:palm-law} to the two-walk case, for $\Xi\in\{\Pi^\T,\widetilde\Pi\}$, one can first sample the (coupled) walks $W^{1,\Xi}, W^{2,\Xi}$, then the Poisson point process $\Xi$ conditional on the two walks.  To prove the statement, it suffices to work on $\Pi^\T$ and show that the event
\begin{multline}\label{eq:two-corridor-compatibility-event}
  \mathcal E_{12}:=\Big\{
  (e,s)\mapsto(\cP(e),s)\text{ is injective on }
  \mathcal R(\mathcal V_t^{\rm str}(W_t^{1,\Pi^\T}) \cup
             \mathcal V_t^{\rm str} (W_t^{2,\Pi^\T})  ), \\ W^{i,\Pi^\T}|_{[0,t]}\in\mathcal C_t^{\rm str}[W_t^{i,\Pi^\T};\Pi^\T],\; i=1,2
  \Big\}
\end{multline}
has probability at least $1-Ce^{-c(\log(n))^{0.02}}$.
Indeed, if we let $\widetilde{\mathcal E}_{12}$ be the same event with $\Pi^\T$ replaced by $\widetilde\Pi$, then $\P[\mathcal E_{12}]=\P[\widetilde{\mathcal E}_{12}]$, and $W^{1,\Pi^\T}, W^{2,\Pi^\T}$ under $\mathcal E_{12}$ and $W^{1,\widetilde\Pi}, W^{2,\widetilde\Pi}$ under $\widetilde{\mathcal E}_{12}$ have the same law, so we can couple them to make them equal almost surely. Therefore the Poisson processes in these sets, i.e., $\Pi^\T\cap \mathcal R(\mathcal V_t^{\rm str}(W_t^{1,\Pi^\T}) \cup \mathcal V_t^{\rm str} (W_t^{2,\Pi^\T})  )$ and $\widetilde\Pi \cap\mathcal R(\mathcal V_t^{\rm str}(W_t^{1,\widetilde\Pi}) \cup \mathcal V_t^{\rm str} (W_t^{2,\widetilde\Pi})  )$, can further be coupled to be equal almost surely.

Below we lower bound $\P[\mathcal E_{12}]$. 
Denote by $\mathcal E^{i}_{\rm str}$ the event in the second line of \eqref{eq:two-corridor-compatibility-event}, for each $i=1, 2$.
Then by \eqref{eq:strone}, we have $\P[\mathcal E^{i}_{\rm str}]> 1-Ce^{-c(\log(n))^{0.02}}$.

It now suffices to consider the injectivity event in \eqref{eq:two-corridor-compatibility-event}.
The plan is as follows. We first let $\mathcal{E}^i_{\rm tree}$ be the event where 
$B_{2L_n}(\mathcal P(W_s^{i,\Pi^\T}))$ is a tree for every $T_{\rm ini}\le s\le t$.
Since $\mathcal P(W^{i,\Pi^\T})$ is a random walk on $G_n$, by \eqref{eq:B2LWs} we have that $\P[\mathcal{E}^i_{\rm tree}]> 1-Ce^{-cL_n}$.
Then using the confinement event (i.e., the second line of \eqref{eq:two-corridor-compatibility-event}), $\mathcal{E}^i_{\rm tree}$ implies that $\mathcal P$ is injective on any time slice of the tube $\mathcal V_t^{\rm str}(W_t^{i,\Pi^\T})$.
Finally, we show that the projections of $W^{1,\Pi^\T}$ and $W^{2,\Pi^\T}$ are well separated after $T_{\rm ini}$.

We next provide the details.

\smallskip

\noindent\textbf{Step 1: looplessness of time slices for each walk.}
Consider the event  
$$
\mathcal{E}^i_{\rm conf}:= \mathcal{E}^i_{\rm tree}\cap  \mathcal{E}^i_{\rm str}
\cap \{N_{W^i}(I) \leq 4 x_n, \, \forall \mbox{ interval }I \subset [0,t] \mbox{ s.t. } |I|\leq 4x_n \},
$$
where $N_{W^i}(I)$ is the number of jumps of walk $W^i$ in the time interval $I$.  By \eqref{eq:strone} and the above stated inequality $\P[\mathcal{E}^i_{\rm tree}]> 1-Ce^{-cL_n}$, as well as  a Poisson tail bound, we find that
\begin{equation}\label{eq:omega_conf_bd}
    \P[\mathcal{E}^i_{\rm conf}]> 1 - C e^{-c (\log(n))^{0.02}}.
\end{equation}
On $\mathcal{E}^i_{\rm conf}$, the
triangle inequality gives
\[
  \dist\bigl(\cP(v),\cP(W_s^{i,\Pi^\T}) \bigr)
  \le4x_n+4y_n=o(L_n),\; \forall \, s\geq T_{\rm ini} \mbox{ and } v\in \mathcal V_t^{\rm str}(W_t^{i,\Pi^\T})_s.
\]
 Consequently, on $\mathcal{E}^i_{\rm conf}$,  for $T_{\rm ini}\le s \le t$ the projection of each 
time-$s$ slice, i.e., $\cP(\mathcal V_t^{\rm str}(W_t^{i,\Pi^\T})_s)$, lies in
$B_{L_n}(\cP(W_s^{i,\Pi^\T}))$, which is a tree by $\mathcal{E}^i_{\rm tree}$.   Moreover, $\mathcal{E}^i_{\rm str}$ ensures that for $0\le s\leq T_{\rm ini}+T_1$, the projection $\cP(\mathcal V_t^{\rm str}(W_t^{i,\Pi^\T})_s)$ is contained in $\mathcal T_{\rm ini}$, which is also a tree. 

\smallskip
\noindent\textbf{Step 2: disjointness of two tube projections.} It remains to prove that, with high probability, time slices of projections  $\mathcal P (\mathcal V_t^{\rm str}(W_t^1)_s)$ and $\mathcal P(\mathcal V_t^{\rm str}(W_t^2)_s)$  are well separated for $T_{\rm ini}+T_1\le s\le t$. To this end, put $$m_n:=10x_n+10 y_n.$$
For each $0\le s \le t$, we denote $D(s):=\mbox{dist}(\cP(W_{s}^{1,\Pi^\T}), \cP(W_{s}^{2,\Pi^\T}))$.
We next bound the probability that $D(s)\le m_n$ for some $T_{\rm ini}+T_1\le s \le t$.

Under $\mathcal{E}^1_{\rm str}\cap \mathcal{E}^2_{\rm str}$, $\cP(W_{s}^{1,\Pi^\T}), \cP(W_{s}^{2,\Pi^\T})\in \mathcal T_{\rm ini}$ for all $0\le s\le T_{\rm ini}+T_1$.
By considering the evolution of the distance $D(s)$, and the distance from $\cP(W_{s}^{1,\Pi^\T}), \cP(W_{s}^{2,\Pi^\T})$ to the vertex $u_*$, and using Chernoff bounds, we have
\[
\P[\{D(T_{\rm ini}+T_1) \le 2m_n\}\cap \mathcal{E}^1_{\rm str}\cap \mathcal{E}^2_{\rm str}] < Ce^{-c\log(n)}.
\]
Also by using Chernoff bounds, for any $0\le s\le t$ we have
\[
\P\Big[\Big\{D(s+L_n )\leq 2 m_n \mbox{ or } \inf_{s\leq r\leq s+L_n}  D(r) \leq  m_n  \Big\} \cap \{D(s) \geq 2 m_n\} \cap \mathcal{E}_{\rm tree}^1 \cap  \mathcal{E}_{\rm tree}^2 \Big] < C e^{-c m_n}.
\] 
Consequently, by iterating this estimate $O(1)$ times, we find that 
$$
\P\Big[\Big\{\inf_{T_{\rm ini}+T_1 \leq s\leq t} D(s) \leq m_n\Big\} \cap \mathcal{E}^1_{\rm conf}\cap \mathcal{E}^2_{\rm conf} \Big]< C e^{-c (\log(n))^{0.49}}. 
$$
Combining this with the bound \eqref{eq:omega_conf_bd}, the conclusion follows.
\end{proof}

\begin{prop}
\label{prop:tree-pullback}
Under the setup of \Cref{lem:structural-corridors},
\[
  \P\bigg[
    \exists i\in\{1,2\}:
    \frac{X_t^{*,i}(\Pi^\T)}
         {X_t^{*,i}(\widetilde\Pi)}
    \notin(1/2,2)
  \bigg]
  < Ce^{-c(\log(n))^{0.02}}.
\]
\end{prop}

\begin{proof}
For $\Xi \in \{\Pi^\T,\,  \widetilde \Pi\}$, the construction of the walk $W^\Xi$ gives
\[
  \E\Big[\sum_{v\in V}\bigl(w_{0,t}^\Xi(o,v)-S_t(v;\Xi)\bigr)\Big]
  =\P[W^{\Xi}|_{[0,t]}\not\in \mathcal C_t^{\rm str}[W^{\Xi}_t;\Xi]].
\]
By \eqref{eq:strone} (which is also valid for $\widetilde \Pi$), the right-hand side is bounded by
$Ce^{-c(\log(n))^{0.02}}$.  Applying \Cref{lem:sb-markov} with
$\alpha=1/2$ shows that,
outside an event of probability $Ce^{-c(\log(n))^{0.02}}$, 
\[
  S_t^{*,i}(\Xi)\le X_t^{*,i}(\Xi)
  \le 2S_t^{*,i}(\Xi),
  \qquad i\in\{1,2\}.
\]
Combining this with Lemma \ref{lem:structural-corridors} finishes the proof.
\end{proof}

\subsection{Pullback weight versus finite-graph weight}\label{ssec:pullback_finite}

In this subsection, we assume that the graph $G_n$ with root $o_n$ is $\delta$-good in the sense of Definition \ref{def:good-graph}.
We compare $X_t^{*,i}(\widetilde\Pi)$ and $X_t^{*,i}(\Pi^n)$ for $i=1,2$. 
Unlike the previous subsection, where one needs to construct a coupling between $W^{1,\Pi^\T}, W^{2,\Pi^\T}, \Pi^\T$ and $W^{1,\widetilde\Pi}, W^{2,\widetilde\Pi}, \widetilde\Pi$, there is a natural coupling between $\widetilde\Pi$ and $\Pi^n$ and the walks, through the projection map $\cP$. 
We work under this coupling, and it suffices to consider one walk, i.e., compare $X_t^{*}(\widetilde\Pi)$ and $X_t^{*}(\Pi^n)$ under this coupling, since the comparison between two walks obviously follows by a union bound.

Note that under the coupling between $\widetilde\Pi$ and $\Pi^n$ through $\cP$, we always have $X_t^*(\widetilde\Pi)\le X_t^*(\Pi^n)$.
The  reverse comparison is then the more difficult and technically demanding part. 
We will prove that  $X_t^*(\Pi^n)$ can be (almost) bounded by $2X_t^*(\widetilde \Pi)$, with probability at least $1-C\delta-o(1)$, when $G_n$ is a $\delta$-good rooted graph.

For this, we rewrite the difference $X_t^*(\Pi^n)-X_t^*(\widetilde\Pi) $ as follows. 
    For $\gamma\in \mathcal C_{0,t}[o_n, W_t^{\Pi^n}; \Pi^n]$, its \emph{lift} $\widetilde\gamma: [0, t]\to V$ is the unique function such that $\widetilde\gamma(0)=o$, $\cP(\widetilde\gamma)=\gamma$, and  for any $0< s\le t$, we have $\dist(\widetilde\gamma(s-),\widetilde\gamma(s))\le 1$. 
The fact that $\widetilde\gamma$ is uniquely determined by the conditions relies on the fact that $G_n$ is a simple graph.
We then have
\begin{equation}   \label{eq:diffXnt}
X_t^*(\Pi^n)-X_t^*(\widetilde\Pi) = \sum_{\gamma\in \mathcal C_{0,t}[o_n, W_t^{\Pi^n}; \Pi^n], \widetilde\gamma(t)\neq W_t^{\widetilde\Pi}} w(\gamma; \Pi^n).
\end{equation}
To control the right-hand side, we introduce a classification of all $\gamma\in \mathcal C_{0,t}[o_n, W_t^{\Pi^n}; \Pi^n]$, in terms of the sizes of its \textit{gaps} away from the \textit{spine}, i.e., $W^{\Pi^n}$.

\begin{definition}\label{def:gap}
For each $\gamma \in \mathcal C_{0,t}[o_n, W_t^{\Pi^n}; \Pi^n]$,
record each maximal interval
$[a,b)$ in which  $\gamma$ differs from $W^{\Pi^n}$, and call
  it a \emph{gap}.  Its  endpoints $W_{a-}^{\Pi^n}$ and $W_b^{\Pi^n}$  are
   called \emph{divergence endpoint} and \emph{reunion endpoint}, respectively. 
  We call the edge $e_a=\{W_a^{\Pi^n}, \bar W_a^{\Pi^n}\}$ such that $(e_a, a)\in \Pi^n$ the \emph{divergence edge}.  
\end{definition}

Define the base tube $ \mathcal V$ and its enlargement $  \mathcal V^+$ by
\[
  \mathcal V
  :=\mathcal V_{0,t;x_n}(o, W_t^{\widetilde\Pi};2y_n),\qquad
  \mathcal V^+
  :=\left\{
    (v,s):\dist\bigl(v,\mathcal V_s\bigr)\le L_n
  \right\},\qquad
  \mathcal U^+:=\mathcal R(\mathcal V^+).
\]
For $\gamma \in \mathcal C_{0,t}[o_n, W_t^{\Pi^n}; \Pi^n]$, its \emph{piecewise lift} $\hat\gamma:[0, t]\to V$ is the unique function such that $\cP(\hat\gamma)=\gamma$, and  $\hat\gamma(s)=W^{\widetilde\Pi}_s$ whenever $\cP(\hat\gamma(s))=\cP(W^{\widetilde\Pi}_s)$, and for any other $0\le s \le t$, we have $\dist(\hat\gamma(s-), \hat\gamma(s))\le 1$.
In words, $\hat\gamma$ is the lift which restarts from $W^{\widetilde\Pi}$ whenever $\cP(\hat\gamma)$ hits $W^{\Pi^n}$.

\begin{definition}\label{defn:gammaclass}
     For $\gamma \in \mathcal C_{0,t}[o_n, W_t^{\Pi^n}; \Pi^n]$, a gap $[a,b)$ of it is classified as
    \begin{itemize}
        \item \emph{initial}: if $\exists \, a\leq s\leq b \wedge (T_{\rm ini}+T_1) $ such that $\gamma(s)\notin \mathcal T_{\rm ini}$;
        \item \emph{short}: if $b-a\leq T_1$ and this gap is not initial; moreover, it can be further classified as
        \begin{itemize}
            \item \emph{bad short}: if 
        $(\hat \gamma (s), s) \notin \mathcal V^+ $ for some $s\in [a,b)$;
            \item \emph{good short}: otherwise;
        \end{itemize}
        \item \emph{medium}: if $b-a \in (T_1, T_2]$ and this gap is not initial;
        \item \emph{long}: if $b-a>T_2$ and this gap is not initial.
    \end{itemize}
    We also classify $\gamma$ as follows. If all its gaps are good short, we classify $\gamma$ to be \emph{ordinary}.
    Otherwise, consider its first gap that is either initial, bad short, medium, or long. We classify $\gamma$ to be either
    \emph{initial}, \emph{bad short}, \emph{medium}, or \emph{long}, accordingly.
\end{definition}
    For brevity, we also write ord/ini/bsh/med/long for the classes.
For  $\bullet \in \{\mathrm{ord},\allowbreak\ \mathrm{ini},\allowbreak\ \mathrm{bsh},\allowbreak
    \, \mathrm{med},\allowbreak\ \mathrm{long}\} $, we denote by $X_t^\bullet$ the total weight of paths in the class $\bullet$, i.e.,
$$
X_t^\bullet=\sum_{\gamma\in \mathcal C_{0,t}[o_n, W_t^{\Pi^n}; \Pi^n], \gamma \textrm{ is in class }\bullet} w(\gamma; \Pi^n).
$$
To bound \eqref{eq:diffXnt}, we analyze contributions from these five classes respectively, and our strategy is as follows.
\begin{itemize}
    \item For $\gamma$ from ord, unless $W^{\Pi^n}$ behaves atypically, it would have $\widetilde\gamma(t)=W_t^{\widetilde\Pi}$, so that $\gamma$ does not contribute to $X_t^*(\Pi^n)-X_t^*(\widetilde\Pi)$. The probability of such atypical behavior is bounded using the preliminary estimates in \Cref{ssec:setupprees}.
    \item For the two terms $X_t^{\rm ini}$ and $X_t^{\rm bsh}$, using \Cref{lem:sb-markov}, bounds on them would also be reduced to random walk estimates in \Cref{ssec:setupprees}.
    \item For $X_t^{\rm med}$, the idea is to consider the \emph{off-spine weight} 
\[
D_{a,b}
  :=
  \frac12
  \sum_{\substack{\gamma\in \mathcal{C}_{a,b}[\bar W_a^{\Pi^n}, W_b^{\Pi^n};\Pi^n],\\
       \gamma(r)\ne W_r^{\Pi^n},\ a\le r<b}}
  w(\gamma;\Pi^n),
\]
   for any $\gamma\in \mathcal C_{0,t}[o_n, W_t^{\Pi^n}; \Pi^n]$ and its medium gap $[a,b)$. Here $\bar W_a^{\Pi^n}$ is the other endpoint of the divergence edge $e_a$, and the factor $1/2$ here comes from the branch chosen at the divergence endpoint.\\
We will show that the off-spine weight $D_{a,b}$ is (almost) bounded by the random walk transition probability. Heuristically speaking, this bound is due to the disjointness, i.e., $\gamma(r)\ne W_r^{\Pi^n}$ for $a\le r<b$, so that conditioning on $W^{\Pi^n}$ would not significantly increase $w(\gamma;\Pi^n)$. \\
With the transition probability decay  \eqref{eq:entropic-mixing}, we then get that $D_{a,b}$ is much smaller than the corresponding weight $w_{a-,b}^{\Pi^n} (W_{a-}^{\Pi^n}, W_b^{\Pi^n})$ along the spine, since $\kappa_*<\kappa_2$.
From there we conclude that
$X_t^{\rm med}$ is much smaller than $X_t^*(\Pi^n) $.
\item For $X_t^{\rm long}$, we also bound the off-spine weights, but now for long gaps. 
As in the medium gap case, we use the disjointness and \eqref{eq:TV-mixing}, and bound the long gap off-spine weights by order $n^{-1+o(1)}$.
\end{itemize}
Bounds for ord, ini, and bsh will be given in \Cref{sssec:geometric_setup}.
Sections \ref{sssec:medium_gaps} and \ref{sssec:long_gaps} are then devoted to med and long, respectively.

\smallskip

\noindent\textbf{Geometric events.}
We introduce the following events, which will be used repeatedly later.
Take the cycle-cut tree component $\mathcal T_{\rm ini}$ from
\eqref{eq:initial-tree-component}. 
Let the
tube-confinement event $\mathcal{E}_{\rm tube}$ be
\[
  \{W^{\widetilde\Pi}|_{[0,t]}\in\mathcal C_t^{\rm str}[W_t^{\widetilde\Pi};\widetilde\Pi]\}
  =\{(W_s^{\widetilde\Pi},s)\in\mathcal V,
  \, \forall\, 0\le s\le t\}\cap 
  \{  \mathcal P(W_s^{\widetilde\Pi})\in {\mathcal T}_{\rm ini}, \, \forall\, 0\le s\le T_{\rm ini}+T_1\},
\]
where we recall \eqref{eq:structural-path-set} for the path set $\mathcal C_t^{\rm str}[W_t^{\widetilde\Pi};\widetilde\Pi]$.
Let $N_{\rm sp}(I)$ count Poisson points incident to $W^{\Pi^n}$ in a
time interval $I$, i.e., 
\[
N_{\rm sp}(I) = | \mathcal R(W^{\Pi^n}|_I) \cap \Pi^n |.
\]
Fix $C_{\rm sp}$ and $C_{\rm pad}$ sufficiently
large and set
\begin{equation}\label{eq:padding-mark-scale}
  m_n^{\rm pad}
  :=\left\lceil
    C_{\rm pad}\frac{\log\log(n)}{\log\log\log(n)}
  \right\rceil.
\end{equation}
Our main event of interest is
\begin{equation}\label{eq:finite-geometry-event}
\begin{split}
  \mathcal{E}_{\rm geom}:&=\mathcal{E}_{\rm tube}\cap  \biggl\{
      \sup_{\substack{0\le u,v\le t\\|u-v|\le4x_n}}
    \dist(W_u^{\Pi^n},W_v^{\Pi^n})\le10x_n ,\;\; \dist(o,W_t^{\widetilde\Pi})\le2t,\\
  &  B_{K_n/2}(W_s^{\Pi^n})\text{ is a tree }, \forall\, 
       0\le s\le T_{\rm ini},\;\;  B_{2L_n}(W_{s'}^{\Pi^n})\text{ is a tree }, \forall \,
       T_{\rm ini}\le s' \le t,\, 
  \\
  &N_{\rm sp}([0,t])\le C_{\rm sp}\log(n),\;\; \max_{s\in\Z}
    N_{\rm sp}\bigl((s-2,s+2]\cap[0,t]\bigr)
    \le m_n^{\rm pad}
  \biggr\}.
\end{split}
\end{equation}
We have that\begin{equation}\label{eq:intrinsic-geometry-probability}
  \P[\mathcal{E}_{\rm geom}^c]=o(1).
\end{equation}
Indeed, \eqref{eq:strone} gives that $\P[\mathcal{E}_{\rm tube}^c]=o(1)$;
\eqref{eq:WsPi} and \eqref{eq:B2LWs} bound the events in the second line of \eqref{eq:finite-geometry-event}; and the remaining events are bounded using Poisson estimates.

\subsubsection{Initial and short gaps}\label{sssec:geometric_setup}
We start with ordinary $\gamma$.
\begin{lemma}    \label{lem:gap-class-decomposition}
Under $\mathcal{E}_{\rm geom}$, every ordinary $\gamma\in \mathcal C_{0,t}[o_n, W_t^{\Pi^n}; \Pi^n]$ must have $\widetilde\gamma(t)=W_t^{\widetilde\Pi}$; therefore  
\begin{equation}
  0\le X_t^*(\Pi^n)-X_t^*(\widetilde\Pi)
  \le X_t^{\rm ini}+X_t^{\rm bsh}
     +X_t^{\rm med}+X_t^{\rm long}.
\end{equation}
\end{lemma} 
\begin{proof}
Now it suffices to show that, under $\mathcal{E}_{\rm geom}$, any ordinary $\gamma\in \mathcal C_{0,t}[o_n, W_t^{\Pi^n};\Pi^n]$ must have $\widetilde\gamma=\hat\gamma$, since then we must have $\widetilde\gamma(t)=\hat\gamma(t)=W_t^{\widetilde\Pi}$.

In words, we want to show that $\hat\gamma$ never restarts. 
Since $\gamma$ has no initial gap, for any $0\le s\le T_{\rm ini}$ we must have $\widetilde\gamma(s)\in \widetilde{\mathcal T}_{\rm ini}$. As $\cP$ is a bijection on $\widetilde{\mathcal T}_{\rm ini}$,  $\hat\gamma$ cannot restart in $[0, T_{\rm ini}]$.

We next consider time in $[T_{\rm ini}, t]$. Note that any restarting of $\hat\gamma$ must happen at the reunion endpoint $b$ of some gap $[a,b)$.
This gap must be good short since $\gamma$ is ordinary, so we have $\hat\gamma(b-), \hat\gamma(b)\in \mathcal V^+_b$.
However, we claim that under the event $\mathcal{E}_{\rm geom}$, $\cP$ is a injective on $\mathcal U^+_s=\mathcal R(\mathcal V^+)_s$ for any $T_{\rm ini}\leq s\leq t$.  This in particular implies that $\hat\gamma$ cannot restart at $b$, and the conclusion follows.

We next prove the claim.
Under the event $\mathcal{E}_{\rm tube}\cap \{\dist(o,W_t^{\widetilde\Pi})\le2t\}$, for any $0\le s \le t$, we have $\mathcal V_s\subset B_{L_n-1}(W_s^{\widetilde\Pi})$, thus $\mathcal V^+_s\subset B_{2L_n-1}(W_s^{\widetilde\Pi})$.
Since under the event $\mathcal{E}_{\rm geom}$,
$B_{2L_n}(W_s^{\Pi^n})$ is a tree for all $T_{\rm ini}\leq s\leq t$, we must have that $\cP$ is injective on $\mathcal U^+_s$.
\end{proof}

Having dealt with ordinary $\gamma$ under $\mathcal E_{\rm geom}$, we now bound $X_t^{\rm ini}$ and $X_t^{\rm bsh}$.
\begin{lemma}\label{lem:short-gaps}
We have
$\P[X_t^{\rm ini}+X_t^{\rm bsh}\leq 0.2X_t^*(\Pi^n)]>1-o(1)$.
\end{lemma}

\begin{proof}
For each $v\in V_n$, let $X_t^{\rm ini}(v)$  and $X_t^{\rm bsh}(v)$ be the sum of $w(\gamma; \Pi^n)$ for $\gamma\in \mathcal C_{0,t}[o_n, v; \Pi^n]$, satisfying that $\gamma(s)\not\in\mathcal T_{\rm ini}$ for some $0\le s\le T_{\rm ini}+T_1$ , or that $\dist(\widetilde\gamma(s), \widetilde\gamma(r))\ge L_n/2$ for some $0\le s<r\le t$, $r-s\le T_1$, respectively.
By Lemma \ref{lem:sb-markov} we have
\[
        \P[ X_t^{\rm ini}\ge 0.1 X_t^*(\Pi^n) ]
\le 10 \E \Big[\sum_{v\in V_n} X_t^{\rm ini}(v)\Big] \\
\le 10\P[W_s^{\Pi^n} \not\in  \mathcal T_{\rm ini},\, \exists\, 0\le s\le T_{\rm ini}+T_1]
=o(1),\]
where the last equality is by \eqref{eq:WsPi}.

We next bound $X_t^{\rm bsh}$.
For any $\gamma$ in the bad short class, there exists a gap $[a,b)$ with $b-a\le T_1$,
such  that its piecewise lift $\hat{\gamma}$ exits $\mathcal V^+$ at some time $s\in [a,b)$.
On the event $\mathcal{E}_{\rm geom}$, we have $\hat\gamma(a)\in\mathcal V_a$, and $\dist(\mathcal V_a, \mathcal V_s)\leq 2(s-a)\leq 2T_1$. These imply that $\dist(\hat{\gamma}(a), \mathcal V_s)\leq 2T_1+10(x_n+y_n)\le L_n/2$.
On the other hand, since $(\hat\gamma(s),s)\not\in\mathcal V^+$, from the  definition of $\mathcal V^+$ we have that $\dist (\hat{\gamma}(s), \mathcal V_s)\geq L_n$.
Thus we have $\dist(\widetilde\gamma(s), \widetilde\gamma(a))=\dist(\hat\gamma(s), \hat\gamma(a))\ge L_n/2$.
Now using Lemma \ref{lem:sb-markov} we have
\begin{multline*}
        \P[ X_t^{\rm bsh}\ge 0.1 X_t^*(\Pi^n) ]
\le 10 \E \Big[\sum_{v\in V_n} X_t^{\rm bsh}(v)\Big]+\P[\mathcal{E}_{\rm geom}^c] \\
\le 10\P\big[ \dist(W_{s_1}^{\widetilde\Pi}, W_{s_2}^{\widetilde\Pi})\ge L_n/2, \, \exists\, 0\le s_1 < s_2\le t,\, s_2-s_1\le T_1\big]+\P[\mathcal{E}_{\rm geom}^c]
=o(1),
\end{multline*}
where the last equality is by Poisson estimates.
\end{proof}

\subsubsection{Medium gaps}\label{sssec:medium_gaps}
Following the strategy stated above of comparing off-spine and total weights, we start with the following weight lower bound.
\begin{lemma}
\label{lem:bridge-lower}
The following holds true with $1-o(1)$ probability:
\begin{equation}\label{eq:medium_weight_lb}
  \forall\,\gamma\in \mathcal C_{0,t}[o_n, W_t^{\Pi^n}; \Pi^n] \mbox{ and its medium gap }[a,b),\quad w_{a-,b}^{\Pi^n} (W_{a-}^{\Pi^n}, W_b^{\Pi^n}) > e^{-\kappa_1 (b-a)-\sqrt{\log n}}.
\end{equation}
\end{lemma}
\begin{proof}
We first show that, uniformly over all pairs
$s, r\in \Z\cap [0, t]$, with $r-s\ge T_1-2$,
\begin{equation} \label{eq:spine_poly_lowerbd}
  \P\big[
    \mathcal{E}_{\rm geom}\cap\big\{w_{s,r}^{\widetilde\Pi}(W_s^{\widetilde\Pi},W_r^{\widetilde\Pi})
    <
    e^{-\kappa_1(r-s)}\big\}
  \big]
  <
  Ce^{-c(r-s)}+Ce^{-c(\log(n))^{0.02}}.
\end{equation}
For this, we note that on $\mathcal{E}_{\rm geom}$, $B_{K_n/2}(W^{\Pi^n}_s)$ is a tree for all $0\le s\leq t$. 
Conditional on $B_{K_n/2}(W^{\Pi^n}_s)$ being a tree, by \Cref{prop:tree-pullback} (which can be easily adapted for the current application, with a smaller tree radius and a shorter time interval), we can couple $w_{s,r}^{\widetilde\Pi}(W_s^{\widetilde\Pi},W_r^{\widetilde\Pi})$ with $X_{r-s}^*(\Pi^\T)$, such that their ratio is in $(1/2, 2)$, with probability $>1-Ce^{-c(\log(n))^{0.02}}$.
Then by \eqref{eq:typical-weight} and $\kappa_1>\kappa_*$, we get \eqref{eq:spine_poly_lowerbd}.

By taking a union bound of \eqref{eq:spine_poly_lowerbd}, using \eqref{eq:intrinsic-geometry-probability}, and noting that $w_{s,r}^{\widetilde\Pi}(W_s^{\widetilde\Pi},W_r^{\widetilde\Pi})\le w_{s,r}^{\Pi^n}(W_s^{\Pi^n},W_r^{\Pi^n})$, we have 
\[
  \P\bigg[
    \mathcal{E}_{\rm geom}\cap\bigcap_{s,r\in\Z\cap [0, t], r-s\ge T_1-2}\big\{w_{s,r}^{\Pi^n}(W_s^{\Pi^n},W_r^{\Pi^n})
    \ge 
    e^{-\kappa_1(r-s)}\big\}
  \bigg] = 1-o(1).
\]
For any $\gamma\in \mathcal C_{0,t}[o_n, W_t^{\Pi^n}; \Pi^n]$ and its medium gap $[a,b)$, take $s=\lceil a\rceil$ and $r=\lfloor b\rfloor$.
On $\mathcal{E}_{\rm geom}$, $N_{\rm sp}(I)\leq m_n^{\rm pad}$ for each  interval $I\subset [0,t]$ (with integer endpoints) of length $2$. This implies that  
$$
w^{\Pi^n}_{a-, s}(W_{a-}^{\Pi^n}, W_{s}^{\Pi^n}) \wedge 
w^{\Pi^n}_{r, b }(W_{r}^{\Pi^n}, W_{b }^{\Pi^n})
\geq 2^{-m_n^{\rm pad}}> e^{-\sqrt{\log n}/2}.
$$ Then (by \eqref{eq:chapman-kolmogorov})  
\begin{equation*}
\begin{split}
  w_{a-,b}^{\Pi^n}(W_{a-}^{\Pi^n},W_b^{\Pi^n})
&  \ge  w_{a-,s}^{\Pi^n}(W_{a-}^{\Pi^n},W_{s}^{\Pi^n})
 w_{s,r  }^{\Pi^n}(W_{s}^{\Pi^n},W_{r}^{\Pi^n})
   w_{r,b  }^{\Pi^n}(W_{r}^{\Pi^n}, W_b^{\Pi^n})\\
&  > e^{-\sqrt{\log n}}
  w_{s,r}^{\Pi^n}
  (W_{s}^{\Pi^n},W_{r}^{\Pi^n}).
  \end{split}
\end{equation*}
Thus the conclusion follows.
\end{proof}

We next upper bound the off-spine weight $D_{a,b}$ for any medium gap, using random walk transition probabilities.
For this, we introduce the notion of \emph{padded bubble kernel}.  Consider
\[
  V_n^{2,\rm off}:=\{(v,u)\in V_n^2: v\ne u\}.
\]
The \emph{split kernel} $\mathsf S: V_n\times V_n^{2,\rm off}\to\R_{\ge 0}$ and \emph{reunion kernel} $\mathsf C: V_n^{2,\rm off}\times V_n\to\R_{\ge 0}$ are defined as
\[
\begin{split}
  \mathsf S(u, (u,v))=\mathsf S(u, (v,u))&=\frac1{4d},
    \quad u\sim v,\\
  \mathsf C((u,v), u)=\mathsf C((u,v), v)&=\frac1{4d},
    \quad u\sim v,
\end{split}
\]
and all other entries vanish.
For any $h>0$ and $s\ge0$, we define the restricted transition kernel as
\[
  \mathsf U_s^{(h)}(u, v)
  :=\P\big[W^u_s=v,\, |\mathcal R(W^u|_{[0,s]}) \cap \Pi^n|\le h\big],
\] 
where $W^u$ is a random walk generated by $\Pi^n$ on the time interval $[0, \infty)$, starting from $u$.
Also, let  
$\mathsf Q_s$ be the killed transition kernel of two walks:
$$
\mathsf Q_s((u_1,u_2), (v_1, v_2)):=\P\big[ W^{u_1}_s=v_1, W^{u_2}_s=v_2, W^{u_1}_r\neq W^{u_2}_r, \, \forall \,0\leq r\leq s\big],
$$
where $W^{u_1}$ and $W^{u_2}$ are random walks generated by $\Pi^n$ on the time interval $[0, \infty)$, starting from $u_1$ and $u_2$ respectively, and they are independent conditional on $\Pi^n$.

For real numbers $s_1,s_3\ge0$ and $s_2>0$, define the padded bubble kernel
\begin{equation}\label{eq:padded-bubble-definition}
  \mathcal J_{s_1,s_2,s_3}^{\rm pad}
  :=\mathsf U_{s_1}^{(m_n^{\rm pad})}\mathsf S
    \mathsf Q_{s_2}\mathsf C
    \mathsf U_{s_3}^{(m_n^{\rm pad})}.
\end{equation}
In words, $\mathcal J_{s_1,s_2,s_3}^{\rm pad}$ is
the expected weight of one \emph{padded bubble}.
Also recall that $p^n_s$ denotes the  transition probability from time 0 to $s$ for a random walk on $G_n$, with edge jump rate $1/(2d)$.

\begin{lemma}\label{lem:offspine-first-moment}
For $0\le s_1,s_3\le 1$ and $T_1-2< s_2\le T_2$,
we have
\begin{equation}\label{eq:weighted-bridge-bubble}
  \max_{u\in V_n}\sum_{v\in V_n}
  \frac{\mathcal J_{s_1,s_2,s_3}^{\rm pad}(u,v)}
       {p_{s_1+s_2+s_3+2}^n(u,v)}
  <  (\log(n))^C.
\end{equation}
Then with probability $>1-o(1)$ the following is true: for any $\gamma\in \mathcal C_{0,t}[o_n, W_t^{\Pi^n}; \Pi^n]$ and its medium gap $[a,b)$,
\begin{equation}\label{eq:dab_small_p}
      D_{a,b}
  <
  n^{o(1)}
  p_{(\lceil b\rceil \wedge t)-\lfloor a\rfloor+2}^n
  (W_{\lfloor a\rfloor}^{\Pi^n},W_{\lceil b\rceil \wedge t}^{\Pi^n}).
\end{equation}
\end{lemma}

\begin{proof}
We first prove \eqref{eq:weighted-bridge-bubble}.
Replace $\mathsf Q$ inside $\mathcal J^{\rm pad}_{s_1,s_2,s_3}$ with an unordered version: let $ V_n^{2, \rm un}:=\{\{u,v\}\subset V_n:u\ne v\}$, and for any $s>0$,
\[
  \mathsf Q_s^{\rm un}(\{u_1,u_2\},\{v_1,v_2\})
  :=\mathsf Q_s((u_1,u_2),(v_1,v_2))+\mathsf Q_s((u_1,u_2),(v_2,v_1)).
\]
By restricting to the survival event, this is bounded by the transition probability of a two-particle symmetric exclusion process  with edge rate $1/(2d)$.
Further using the comparison inequality of symmetric exclusion processes with  independent random walks (see e.g., \cite[Corollary 1.9 of Chapter VIII]{Liggett85}),
we conclude that 
\begin{equation}\label{eq:Qsbd}
    \mathsf Q_{s}((u_1,u_2),(v_1,v_2)) \le \mathsf Q_s^{\rm un}(\{u_1,u_2\},\{v_1,v_2\})
    \leq   \sum_{i_1,i_2 \in \{1,2\}} 
    p^n_{s}(u_{i_1},v_{1})p^n_{s}(u_{i_2},v_{2}) .
\end{equation}
For the restricted transition kernel 
$\mathsf U_{s_1}^{(m_n^{\rm pad})}(u,v)$,
we note that it equals 0 if $\dist(u,v)>m_n^{\rm pad}$. For $\dist(u,v)\leq m_n^{\rm pad}+3$ and $0\leq s_1 \leq 2$, we have the bound 
\[
  p_{s_1+1}^n(u,v)
  \ge
  e^{-4}
  \frac{1}{(2d)^{m_n^{\rm pad}+3}(m_n^{\rm pad}+3)!} 
  \ge (\log(n))^{-C}.
\]
The same bounds hold for $p_{s_3+1}^n$. Thus $\mathcal J_{s_1,s_2,s_3}^{\rm pad}(u,v) 
        $ is bounded by
\begin{equation*}
    \begin{split}
       &  \sum_{u_1,v_1:\dist (u,u_1)\vee \dist (v_1,v)\leq m_n^{\rm pad} } \sum_{u_2\sim u_1; v_2\sim v_1} \sum_{i_1,i_2 \in \{1,2\}}  p^n_{s_2}(u_{i_1},v_{1})p^n_{s_2}(u_{i_2},v_{2}) \\
       < & \sum_{u_1,u_2,v_1,v_2\in V_n} (\log(n))^C
       p^n_{s_1+1}(u,u_1)p^n_{s_1+1}(u,u_2) p^n_{s_2}(u_1,v_1)p^n_{s_2}(u_2,v_2)p^n_{s_3+1}(v_1,v)p^n_{s_3+1}(v_2,v)\\
       = & (\log(n))^C p_{s_1+s_2+s_3+2}^n(u,v)^2,
    \end{split}
\end{equation*}
which proves \eqref{eq:weighted-bridge-bubble}.

Next we consider medium gaps. 
For any $s, r\in (\Z\cap [0, t])\cup \{t\}$ with $r-s\ge T_1$, define
\[
  \mathcal G_{s,r}:=\{[a,b):[a,b)\text{ is a medium gap for some }\gamma\in \mathcal C_{0,t}[o_n, W_t^{\Pi^n}; \Pi^n],\
    \lfloor a\rfloor=s,\ \lceil b\rceil \wedge t=r\}
\]
and the normalized sum
\begin{equation}\label{eq:normalized-gap-sum}
  Z_{s,r}:=
  \sum_{[a,b)\in\mathcal G_{s,r}}
  \frac{
    D_{a,b}
  }{
    p_{r-s+2}^n(W_s^{\Pi^n},W_r^{\Pi^n})
  }.
\end{equation}
Under $\mathcal{E}_{\rm geom}$, for any gap $[a,b)\in   \mathcal G_{s,r}$, both $|\mathcal R(W^{\Pi^n}|_{[s,a]})\cap \Pi^n|$ and $|\mathcal R(W^{\Pi^n}|_{[b,r]})\cap \Pi^n|$ are bounded by $m_n^{\rm pad}$.
Thus, by conditioning on $W_s^{\Pi^n}$ and $W_r^{\Pi^n}$, we get \[\E\big[Z_{s,r}\mathds{1}[\mathcal{E}_{\rm geom}]\,\big|\, W_s^{\Pi^n}=u\big]
  \le
  \int_{s}^{s+1}\int_{r-1}^r
  \sum_{v\in V_n}   \frac{
    \mathcal J_{a-s,b-a,r-b}^{\rm pad}(u,v)
  }{
    p_{r-s+2}^n(u,v)
  }
  \,\dd b\,\dd a \leq (\log(n))^C,
\]
where we used
\eqref{eq:weighted-bridge-bubble} in the last inequality. 
By using Markov's inequality to bound the upper tail of $Z_{s,r}\mathds{1}[\mathcal{E}_{\rm geom}]$, taking a union bound over all such $s, r$, and applying \eqref{eq:intrinsic-geometry-probability}, the conclusion follows.
\end{proof}

Based on the previous two lemmas, we can now show that the off-spine weight is much smaller than the total weight.
Recall that $G_n$ is assumed to be $\delta$-good.

\begin{lemma}\label{lem:medium-gap}
Outside an event of probability
$<O(\delta)+o(1)$, for every $\gamma\in \mathcal C_{0,t}[o_n, W_t^{\Pi^n}; \Pi^n]$ and its medium gap $[a,b)$, 
\begin{equation}\label{eq:dab_ratio_wab}
      \frac{D_{a,b}}{w_{a-,b}^{\Pi^n}(W_{a-}^{\Pi^n},W_b^{\Pi^n})}
  < e^{-c\log(n)}.
\end{equation}
Consequently, we have $\P[X_t^{\rm med}>e^{-c\log(n)}X_t^*(\Pi^n)]<O(\delta)+o(1)$.
\end{lemma}

\begin{proof}
By Lemmas \ref{lem:bridge-lower} and \ref{lem:offspine-first-moment}, to validate \eqref{eq:dab_ratio_wab} it suffices to show that, outside an event of probability
$<O(\delta)+o(1)$, 
\begin{equation}\label{eq:transprob_wswr}
 \forall\, s, r\in (\Z\cap [0,t])\cup\{t\}\text{ with } r-s\in (T_1, T_2+2],\quad 
    p_{r-s+2}^n(W_{s}^{\Pi^n}, W_{r}^{\Pi^n})\leq e^{-c\log n-\kappa_1 (r-s)}.
\end{equation}
We prove \eqref{eq:transprob_wswr} in two steps.
First, we claim that, with probability $>1-o(1)$,
\begin{equation}\label{eq:coarse_compare}
  p_{r-s}^n(W_s^{\Pi^n},W_r^{\Pi^n})
  \ge n^{-C_{\rm cp}\varrho}
\end{equation}
simultaneously for all $0\le s\le r\le t$ with $r-s\le\varrho\log n$,
with one of $s, r$ in $(\Z\cap [0,t])\cup\{t\}$ and the other in $(\varrho\log(n)\Z\cap[0,t])\cup\{t\}$.
Indeed, conditional on $W_s^{\Pi^n}=u$, the endpoint $W_r^{\Pi^n}$ has law
$p_{r-s}^n(u,\cdot)$, and the number of jumps between $s$ and $r$ is
stochastically dominated by a Poisson random variable with mean
$\varrho\log(n)/2$.  Hence
\begin{multline*}
\P\left[
    p_{r-s}^n(u,W_r^{\Pi^n})<n^{-C_{\rm cp}\varrho}
    \,\big|\, W_s^{\Pi^n}=u
  \right]\\ \le
  \P\left[
    \dist(u,W_r^{\Pi^n})>\varrho\log(n)
    \mid W_s^{\Pi^n}=u
  \right]
  +|B_{\varrho\log(n)}(u)|n^{-C_{\rm cp}\varrho}< Cn^{-c\varrho}.    
\end{multline*}
Here we used a Poisson tail bound,
$|B_{\varrho\log(n)}(u)|\le d(d-1)^{\varrho\log(n)-1}$, and
$C_{\rm cp}>10+10\log(d-1)$ from \eqref{eq:coarse-interval-scale}.
There are only $O(\log(n))$ such pairs of $s, r$, so a union bound proves
\eqref{eq:coarse_compare}.

Second, we consider $p_{r-s+2}^n(W_s^{\Pi^n}, W_r^{\Pi^n})$ for $s, r \in (\varrho\log(n)\Z\cap[0,t])\cup\{t\}$ with $T_1<r-s\le T_2+2\varrho\log(n)+2$.
For any $u\in V_n$, we have
\[
\P\big[
    p_{r-s+2}^n(u,W_r^{\Pi^n})
    >2e^{-\kappa_2(r-s+2)}
    \,\big|\, W_s^{\Pi^n}=u
  \big]= \sum_{v\in V_n}
  p_{r-s}^n(u,v)  \mathds{1}[p_{r-s+2}^n(u,v)>2e^{-\kappa_2(r-s+2)}].
\]
Noting that $p_{r-s}^n(u,v)\le e\,p_{r-s+2}^n(u,v)$, this is further bounded by
\[
\sum_{v\in V_n}
  e p_{r-s+2}^n(u,v)  \mathds{1}[p_{r-s+2}^n(u,v)>2e^{-\kappa_2(r-s+2)}] \le \sum_{v\in V_n}
  2e\big(p_{r-s+2}^n(u,v)-e^{-\kappa_2(r-s+2)}\big)\vee 0
  \le 2e\delta,
\]
where the last inequality is by \eqref{eq:entropic-mixing}.
Thus by a union bound over $O(\varrho^{-2})$ many such pairs, we have that,  outside an event of probability $O(\delta)$,
\begin{equation}\label{eq:prs_upperbd}
  p_{r-s+2}^n(W_s^{\Pi^n},W_r^{\Pi^n})
  \le 2e^{-\kappa_2(r-s+2)},
\end{equation}
simultaneously for all $s, r \in (\varrho\log(n)\Z\cap[0,t])\cup\{t\}$ with $T_1<r-s\le T_2+2\varrho\log(n)+2$.

We now finish proving \eqref{eq:transprob_wswr}.
For $s, r$ as in \eqref{eq:transprob_wswr}, we take $s', r'\in (\varrho\log(n)\Z\cap[0,t])\cup\{t\}$ such that 
\[
  0\le s-s',\,r'-r\le\varrho\log n,\qquad
  T_1<r'-s'+2\le T_2+2\varrho\log n+4.
\]
By the semigroup property and \eqref{eq:coarse_compare},
\[
\begin{split}
  p_{r'-s'+2}^n(W_{s'}^{\Pi^n},W_{r'}^{\Pi^n})
  &\ge
  p_{s-s'}^n(W_{s'}^{\Pi^n},W_s^{\Pi^n})
  p_{r-s+2}^n(W_s^{\Pi^n},W_r^{\Pi^n})
  p_{r'-r}^n(W_r^{\Pi^n},W_{r'}^{\Pi^n})\\
  &\ge n^{-2C_{\rm cp}\varrho}
  p_{r-s+2}^n(W_s^{\Pi^n},W_r^{\Pi^n}).
\end{split}
\]
Combining this with \eqref{eq:prs_upperbd} gives
\[
  p_{r-s+2}^n(W_s^{\Pi^n},W_r^{\Pi^n})
  \le 2n^{2C_{\rm cp}\varrho}e^{-\kappa_2(r'-s'+2)}
  < 2e^{2C_{\rm cp}\varrho\log(n)-\kappa_2(r-s)}
  \le e^{-c\log(n)-\kappa_1(r-s)}.
\]
The last inequality follows from $r-s>T_1=(c_*/10+o(1))\log(n)$ and
\eqref{eq:coarse-interval-scale}.
This proves \eqref{eq:transprob_wswr}, and hence \eqref{eq:dab_ratio_wab},
outside an event of probability $O(\delta)+o(1)$.

We now prove the second statement on $X_t^{\rm med}$.
  Denote by $X_t^{\rm med}(a,b)$ the total weight of paths $\gamma$ with a medium gap $[a,b)$. Then on the event    of  \eqref{eq:dab_ratio_wab}, 
$$
X_t^{\rm med}(a,b)< e^{-c\log n } w^{\Pi^n}_{0,a-}(o_n, W^{\Pi^n}_{a-}) w_{a-,b}^{\Pi^n}(W_{a-}^{\Pi^n},W_b^{\Pi^n})
w_{b,t}^{\Pi^n}(W_{b}^{\Pi^n},W_t^{\Pi^n})\leq e^{-c\log(n)} w^{\Pi^n}_{0,t}(o_n, W^{\Pi^n}_t),
$$
  and hence by further assuming $\mathcal{E}_{\rm geom}$, which implies that $N_{\rm sp}([0,t])\le C_{\rm sp}\log(n)$, we have 
$$
X_t^{\rm med}<(C_{\rm sp}\log(n))^2 e^{-c\log(n)}X_t^*(\Pi^n).
$$
Thus the second statement of Lemma \ref{lem:medium-gap} follows.
\end{proof}

\subsubsection{Long gaps and completion of the finite-graph comparison}
\label{sssec:long_gaps}
We now estimate the final piece $X_t^{\rm long}$, and then finish the comparison between $X_t^*(\widetilde \Pi)$ and $X_t^*(\Pi^n)$.

For integers $1\le i<j\le 1+\lfloor t/(\varrho\log(n))\rfloor$ with
$(j-i-1)\varrho\log(n)>T_2^-$, define
\[
  \mathcal L_{i,j}
  :=\sum_{\substack{
    \gamma\in\mathcal C_{0,t}[o_n,W_t^{\Pi^n};\Pi^n]:\\
    \gamma\text{ has a long gap }[a,b),\\
    \lfloor a/(\varrho\log(n))\rfloor=i-1,\,
    \lfloor b/(\varrho\log(n))\rfloor=j-1
  }}w(\gamma;\Pi^n).
\]
We call such pairs $i,j$ admissible.

\begin{lemma}\label{lem:long-gap}
We have
\begin{equation}\label{eq:coarse-long-palm-bound}
  \P\left[
    \mathcal L_{i,j}>\delta^{-2}(\varrho\log(n))^4/n
    \text{ for some admissible }i,j
  \right]
  < C\delta.
\end{equation}
Consequently, we have $\P[X_t^{\rm long}>C\delta^{-2}(\log(n))^4/n]<C\delta$.
\end{lemma}

\begin{proof}
For $1\le i\le1+\lfloor t/(\varrho\log(n))\rfloor$, let
\[
  \mathcal N_i:=
  \left\{(e,s)\in\mathcal R(W^{\Pi^n})\cap\Pi^n:
    0\le s\le t,\ \lfloor s/(\varrho\log(n))\rfloor=i-1
  \right\}
\] 
be the Poisson marks incident to the planted walk in the $i$-th time interval.
For any $(\{u_1,u_2\},s),\allowbreak (\{v_1,v_2\},r)\in\Pi^n$, set
\[
\begin{aligned}
  \mathsf M(\{u_1,u_2\}, s, \{v_1,v_2\}, r)
  &:=p^n_{r-s}(u_1,v_1)+p^n_{r-s}(u_2,v_1)\\
  &\quad+p^n_{r-s}(u_1,v_2)+p^n_{r-s}(u_2,v_2),
\end{aligned}
\]
and define
\begin{equation}\label{eq:coarse-long-mixture-sum}
  \mathsf M_{i,j}:=
  \sum_{(\{u_1,u_2\},s)\in\mathcal N_i}\sum_{(\{v_1,v_2\},r)\in\mathcal N_j}
  \mathsf M(\{u_1,u_2\},s,\{v_1,v_2\},r).
\end{equation}

We first bound $\mathsf M_{i,j}$.  Conditional on the history through
$i\varrho\log(n)$, the mixing bound \eqref{eq:TV-mixing} and
$(j-i-1)\varrho\log(n)>T_2^-$ allow us to couple the planted walk at
$(j-1)\varrho\log(n)$ with a uniform vertex, with mismatch probability
at most $\delta$.  Start a replacement walk from this uniform vertex,
using the same subsequent clocks and choices to stay or cross at each clock ring.  Write
$\widehat{\mathcal N}_j$ for its incident marks in the $j$-th interval
and $\widehat{\mathsf M}_{i,j}$ for the sum obtained by replacing
$\mathcal N_j$ with $\widehat{\mathcal N}_j$.  Then
\begin{equation}\label{eq:cp_long}
  \P[\mathsf M_{i,j}=\widehat{\mathsf M}_{i,j}]\ge1-\delta.
\end{equation}
For the replacement walk, incident marks arrive at rate one and both
endpoints of each ringing edge have uniform marginals.  Hence
\[
  \E[\widehat{\mathsf M}_{i,j}\mid\mathcal N_i]
  \le4|\mathcal N_i|\varrho\log(n)/n.
\]
Since $\E[|\mathcal N_i|]\le\varrho\log(n)$, Markov's inequality and
\eqref{eq:cp_long} give
\[
  \P[\mathsf M_{i,j}>\delta^{-1}(\varrho\log(n))^2/n]<C\delta.
\]

We have (whenever $\mathsf M_{i,j}>0$)
\[
  \frac{\mathcal L_{i,j}}{\mathsf M_{i,j}}
  \le
  \sum_{(\{u_1,u_2\},s)\in\mathcal N_i}\sum_{(\{v_1,v_2\},r)\in\mathcal N_j}
  \frac{\sum_{\gamma\in\mathcal C_{0,t}[o_n,W_t^{\Pi^n};\Pi^n]: [s,r) \text{ is a gap of } \gamma}  w(\gamma; \Pi^n) }
       {\mathsf M(\{u_1,u_2\},s,\{v_1,v_2\},r)}.
\]
Its expectation is bounded by
\[
\E\Bigg[ \sum_{\substack{(\{u_1,u_2\},s), (\{v_1,v_2\},r) \in \Pi^n,\, 0\le s<r\le t,\\  \lfloor s/(\varrho\log(n))\rfloor=i-1, \lfloor r/(\varrho\log(n))\rfloor=j-1  }}  \frac{ \sum_{\substack{v\in V_n, \; \gamma, \gamma'\in \mathcal{C}_{0,t}[o_n, v; \Pi^n],\,\gamma(s')\neq \gamma'(s'),\,\forall s'\in [s,r)\\ \{\gamma(s), \gamma'(s)\}=\{u_1,u_2\}, \{\gamma(r-), \gamma'(r-)\}=\{v_1,v_2\} }} w(\gamma; \Pi^n) w(\gamma'; \Pi^n)  }
       {\mathsf M(\{u_1,u_2\},s,\{v_1,v_2\},r)} \Bigg],
\]
which (using \eqref{eq:Qsbd} from the proof of \Cref{lem:offspine-first-moment}) is further bounded by $(\varrho(\log(n)))^2$.
By Markov's inequality (noting that $\mathsf M_{i,j}=0$ implies $\mathcal L_{i,j}=0$),
\[
  \P\left[
    \mathcal L_{i,j}>
    \delta^{-1}(\varrho\log(n))^2\mathsf M_{i,j}
  \right]<C\delta.
\]
Combining this with the bound on $\mathsf M_{i,j}$ and taking a union
bound over  $O(\varrho^{-2})$ many admissible pairs proves
\eqref{eq:coarse-long-palm-bound}.

Finally, 
outside an event
of probability $C\delta$,
\[
  X_t^{\rm long}
  \le\sum_{\text{admissible }i,j}\mathcal L_{i,j}
  \le C\delta^{-2}(\log(n))^4/n,
\]
thus the second conclusion follows.
\end{proof}

All four classes (initial, bad short, medium and long) have now been estimated. Combining their analyses we get the following comparison of 
$X_t^*(\Pi^n)$ and $X_t^*(\widetilde\Pi)$.
\begin{prop}\label{prop:finite-cover}
For any $\delta$-good graph $G_n$ with root $o_n$, and $t$ satisfying
\eqref{eq:finite-comparison-window}, we have
\begin{equation}\label{eq:2x}
      \P\big[
    X_t^*(\widetilde\Pi) \le X_t^*(\Pi^n)
    \le
    2X_t^*(\widetilde\Pi)+C\delta^{-2}(\log(n))^4/n
  \big]
  \ge1-C\delta-o(1).
\end{equation}
Consequently, there exists a coupling of
$X_t^{*,1}(\Pi^n), X_t^{*,2}(\Pi^n)$ and $X_t^{*,1}(\Pi^{\T}), X_t^{*,2}(\Pi^{\T})$ such that
\begin{equation}\label{eq:2walker_Td_finite_comp}
  \P\left( \frac{X_t^{*,i}(\Pi^{\T}) }{2} \leq X_t^{*,i}(\Pi^n) \leq 4 X_t^{*,i}(\Pi^{\T}) +C\delta^{-2}(\log(n))^4/n,\, i=1,2\right)
    \geq 1-C\delta -o(1).
\end{equation}
\end{prop}

\begin{proof}
Combining 
\Cref{lem:gap-class-decomposition,lem:medium-gap,lem:long-gap,lem:short-gaps}, and the estimate \eqref{eq:intrinsic-geometry-probability},
we see that, with probability $>1-C\delta-o(1)$,
\[
  X_t^*(\Pi^n)-X_t^*(\widetilde\Pi)
  \le
  (0.2+o(1))X_t^*(\Pi^n)
  +
  C\delta^{-2}(\log(n))^4/n.
\]
 Rearranging the terms then proves \eqref{eq:2x}.
Then \eqref{eq:2walker_Td_finite_comp} follows by combining \eqref{eq:2x} and \Cref{prop:tree-pullback}.
\end{proof}

\section{Gaussian cutoff profile of random $d$-regular graphs}
\label{sec:profile}
In this section we finish proving the random $d$-regular graph cutoff result stated in Theorem \ref{thm:main-rdreg}, by combining results of the two previous sections.
Then \Cref{thm:main-rdreg} and \Cref{thm:tree-clt} together imply \Cref{thm:intro-tree-clt}.

This is accomplished in two steps. We first prove the Gaussian cutoff profile  under the $\delta$-good condition, and then upgrade this to all starting points, by letting
a random walk run for a short amount of time. Then we show that the good graph hypotheses of \Cref{def:good-graph}, other than (2) involving the root, are satisfied with high probability for a random regular graph. 

\subsection{Gaussian profile from general starting points}\label{ssec:GP_general_reg}
We start by deriving the central limit theorem for two walks on a good finite graph.
Recall the definition of $\tau_n(a)$ from \eqref{eq:profile-time}.
As in the previous section, we let $G_n=(V_n,E_n)$ be a connected simple $d$-regular graph with
$n=|V_n|$, and fix a root $o_n$.
We also continue using the notation $\Pi^n$ and $\Pi^\T$ for rate $1/d$ Poisson point processes on $G_n$ and the infinite $d$-regular tree $\T=(V,E)$, and $(X_t^{*,1}(\Pi^\T), X_t^{*,2}(\Pi^\T))$ and $(X_t^{*,1}(\Pi^n), X_t^{*,2}(\Pi^n))$ for two biased samples from the repeated averaging.

\begin{lemma}
\label{lem:finite-threshold}
Take any $\delta>0$ and $a\in\R$, and $n$ large enough depending on $\delta, a$.
For $G_n$ that is $\delta$-good, and $(\log(n))^5 \le m \le e^{\delta\sqrt{\log(n)}}$, we have
\[
\big|     \P[X_{\tau_n(a)}^{*,1}(\Pi^n)>m/n] -\Phi(-a) \big| \leq C \delta,
\]
\[
\big|     \P[X_{\tau_n(a)}^{*,1}(\Pi^n)\wedge 
      X_{\tau_n(a)}^{*,2}(\Pi^n)>m/n] -\Phi(-a)^2 \big| \leq  C\delta.
\]
\end{lemma}

\begin{proof}
We state the proof for two walks, and the one-walk version is analogous.
By Theorem \ref{thm:two-walker-clt}, 
 the  weights $X_{\tau_n(a)}^{*,i}(\Pi^{\T})$ ($i=1,2$) satisfy, uniformly over $1\leq m\leq e^{2\delta \sqrt{\log n}}$,
\[
  \P[X_{\tau_n(a)}^{*,1}(\Pi^{\T})
\wedge        X_{\tau_n(a)}^{*,2}(\Pi^{\T})>m/n]
  =
  \Phi(-a)^2+O(\delta).
\]
By \Cref{prop:finite-cover}, we have the two-sided bounds
\begin{align*}
  \P[X_{\tau_n(a)}^{*,1}(\Pi^n)\wedge  X_{\tau_n(a)}^{*,2}(\Pi^n)>m/(2n)]
  &\ge\Phi(-a)^2-O(\delta),\\
  \P[X_{\tau_n(a)}^{*,1}(\Pi^n)\wedge  X_{\tau_n(a)}^{*,2}(\Pi^n)>4m/n+C\delta^{-2}(\log(n))^4/n]
  &\le\Phi(-a)^2+O(\delta).
\end{align*}  
Thus the conclusion follows.
\end{proof}

We are now ready to convert the central limit theorem for two walks into an $L^1$ distance bound, using Proposition  
\ref{prop:threshold-L1}.
Below we let $X_t(v)=w_{0,t}(o_n, v;\Pi^n)$ denote the repeated averaging on $G_n$ generated by $\Pi^n$, starting with all mass at $o_n$.
\begin{prop}\label{prop:root-profile}
For any $\delta>0$ and $a\in\R$, and $n$ large enough depending on $\delta$ and $a$, if $G_n$ is $\delta$-good, we have
\[
  \P\big[
    \big|
      \norm{X_{\tau_n(a)}-\one/n}_1-2\Phi(-a)
    \big|
    >C\delta^{1/3}
  \big]
  < C\delta^{1/3}.
\]
\end{prop}

\begin{proof}
We first prove the lower bound
\begin{equation}\label{eq:tvlb}
      \P\big[ 
      \norm{X_{\tau_n(a)}-\one/n}_1-2\Phi(-a)
    <-C\delta^{1/3}
  \big]
  < C\delta^{1/3}.
\end{equation} 
We apply Proposition \ref{prop:threshold-L1} with starting common law given by $X_{\tau_n(a)}$ and parameters  $\theta=(\log(n))^5, \epsilon=\delta^{1/3}$.
Denote by $\pi_1,\pi_2$  the
resulting one- and two-walk probabilities, 
$$
\pi_1=\P\big[ X^{*,1}_{\tau_n(a)}>(\log(n))^5/n\big],\quad \pi_2=\P\big[ X^{*,1}_{\tau_n(a)}>(\log(n))^5/n, X^{*,2}_{\tau_n(a)}>(\log(n))^5/n\big].
$$ By \Cref{lem:finite-threshold}, we have $\pi_1=\Phi(-a)+O(\delta)+o(1)$ and $\pi_2=\Phi(-a)^2+O(\delta)+o(1)$.
The bound \eqref{eq:tvlb} now follows from \eqref{eq:X011bd}.
Next we prove the upper tail bound,
\begin{equation}\label{eq:tvub}
      \P\big[ 
      \norm{X_{\tau_n(a)}-\one/n}_1-2\Phi(-a)
    >C\delta^{1/3}
  \big]
  < C\delta^{1/3}.
\end{equation}
Replace $\tau_n(a)$ by $\tau'_n(a):=\tau_n(a)-(\log \log(n))^2$, and let
\begin{equation*}
    \begin{split}
        \pi'_1&:=\P\big[ X^{*,1}_{\tau'_n(a)}>(\log(n))^5/n\big],\\ 
        \pi'_2&:=\P\big[ X^{*,1}_{\tau'_n(a)}>(\log(n))^5/n, X^{*,2}_{\tau'_n(a)}>(\log(n))^5/n\big].
    \end{split}
\end{equation*}
Under (3) in \Cref{def:good-graph}, the additional running time $s$ in \eqref{eq:cleanup_time_def} is bounded by $O(\log \log(n))$.  Thus \eqref{eq:tvub} follows from \eqref{eq:Xp1ubd} and the fact that $\norm{X_t-\one/n}_1$ is non-increasing in $t$.
\end{proof}

We next relax the condition  of a tree neighborhood for the initial starting point, i.e., below we instead assume that $G_n=(V_n, E_n)$ is a  connected simple $d$-regular graph with
$n=|V_n|$, satisfying (1), (3), (4) of \Cref{def:good-graph}.
For $W^{v}$, the random walk on $G_n$ generated by $\Pi^n$ on $[0, \infty)$ and starting from $v$, we observe that
\begin{equation}\label{eq:item2bd}
     \sup_{v\in V_n} \P\big[\tx(B_{K_n}(W_{8K_n}^v))\ge 1\big]
  < Ce^{-cK_n},
\end{equation} 
which can be proved in the same way as \cite[Lemma 3.2]{LS}; see \eqref{eq:root-after-burnin} above.
Based on this observation we next prove the Gaussian cutoff profile for any starting point.
Denote $X^{(v)}_t(u)=w_{0,t}(v, u;\Pi^n)$.

\begin{prop}\label{prop:any-root-profile}
For any  $\delta>0$ and $a\in\R$, and $n$ large enough depending on $\delta$ and $a$, we have
\begin{equation}\label{eq:xtau}
     \max_{v\in V_n}\P\left[
    \abs{
      \norm{X^{(v)}_{\tau_n(a)}-\one/n}_1-2\Phi(-a)
    }
    >C\delta^{1/12}
  \right]
  < C\delta^{1/12}.
\end{equation}
\end{prop}

\begin{proof}
Put $b_n:=8K_n=8 \lfloor 8 (\log(n))^{0.49}\rfloor$. Note that with $a_n:= a-\kappa_*^{3/2}b_n/({\sigma\sqrt{\log(n)}})$, we have
$
\tau_n(a_n)=\tau_n(a)-b_n$, and $a_n-a=o(1)$.
Thus the conclusion of Proposition \ref{prop:root-profile} holds for the time $\tau_n(a)$ replaced by $\tau_n(a_n)$ as well.

Below we fix $v\in V_n$. Then by \eqref{eq:chapman-kolmogorov}, we have
\[
  X^{(v)}_{\tau_n(a)}(\cdot)
  =
  \sum_{u\in V_n} X^{(v)}_{b_n}(u)
  w_{b_n,\tau_n(a)}(u,\cdot; \Pi^n).
\]
Note that
$ w_{b_n,\tau_n(a)}(u,\cdot; \Pi^n)$ is independent of $X^{(v)}_{b_n}$ and has the same law as $X^{(u)}_{\tau_n(a)-b_n}$.

\smallskip

\noindent\textbf{Upper bound.} We first prove 
\begin{equation}\label{eq:xtau2}
     \P\left[ 
      \norm{X^{(v)}_{\tau_n(a)}-\one/n}_1-2\Phi(-a)
    >\delta^{1/6}
  \right]
  < C\delta^{1/6}.
\end{equation}
By the convexity of the $L^1$ norm, we have
\begin{equation}\label{eq:xtau3}
      \norm{X^{(v)}_{\tau_n(a)}-\one/n}_1
  \le
  \sum_{u\in V_n} X^{(v)}_{b_n}(u)\norm{ w_{b_n,\tau_n(a)}(u,\cdot; \Pi^n)-\one/n}_1.
\end{equation}
By Proposition \ref{prop:root-profile}, for any $u$ such that $B_{K_n}(u)$ is a tree,
\[
    \P\big[ 
 \big|\norm{ w_{b_n,\tau_n(a)}(u,\cdot; \Pi^n)-\one/n}_1-2\Phi(-a)\big|
    >C\delta^{1/3}\big]
  < C\delta^{1/3}.
\]
Then we have 
\begin{equation*}
\begin{split}
     & \E\Big[
    \sum_{u\in V_n}X^{(v)}_{b_n}(u)
\Big(\norm{ w_{b_n,\tau_n(a)}(u,\cdot; \Pi^n)-\one/n}_1-2\Phi(-a)
      -C\delta^{1/3}\Big)\vee 0
  \Big] \\
  = &  \E \Big[ \Big(\norm{ w_{b_n,\tau_n(a)}(W^v_{b_n},\cdot; \Pi^n)-\one/n}_1-2\Phi(-a)
      -C\delta^{1/3}\Big)\vee 0 \Big]  \\
  \leq  &\P\big[\tx(B_{K_n}(W^v_{b_n}))\ge 1\big]+ C \delta^{1/3}.
\end{split}
\end{equation*}
 The upper bound  \eqref{eq:xtau2} can then be obtained from \eqref{eq:item2bd} and \eqref{eq:xtau3}, and Markov's inequality.
\smallskip

\noindent\textbf{Lower bound.}  
We next prove
\begin{equation}\label{eq:xtau6}
     \P\left[ 
      \norm{X^{(v)}_{\tau_n(a)}-\one/n}_1-2\Phi(-a)
    <-\delta^{1/12}
  \right]
  < C\delta^{1/12}.
\end{equation}
Consider the random variable $$
Z= \sum_{u\in V_n}X^{(v)}_{\tau_n(a)}(u)
  \mathds{1}\big[X_{\tau_n(a)}^{(v)}(u)>(\log(n))^8 /n\big].
$$
As in, e.g., \eqref{eq:mtheta_comp}, we have
\[
  \norm{X^{(v)}_{\tau_n(a)}-\one/n}_1 \geq 2(Z- (\log(n))^{-8}),
\]  
which then implies
\begin{equation}\label{eq:xtau5}
    \E \left[ \norm{X^{(v)}_{\tau_n(a)}-\one/n}_1 \right]
    \geq 2(\E Z-(\log(n))^{-8})=2\big(\P\big[X^{*,(v)}_{\tau_n(a)}>(\log(n))^8/n \big] -(\log(n))^{-8}\big),
\end{equation}
where $X^{*,(v)}_{\tau_n(a)}=X^{(v)}_{\tau_n(a)}(W_{\tau_n(a)}^v)$. Note that by \eqref{eq:chapman-kolmogorov}, we have
\begin{equation}\label{eq:xtaun4}
    X^{*,(v)}_{\tau_n(a)} \geq  X^{*,(v)}_{b_n} w_{b_n,\tau_n(a)}(W_{b_n}^v, W_{\tau_n(a)}^v ; \Pi^n). 
\end{equation}
 A standard Poisson deviation bound gives that $\P[X^{*,(v)}_{b_n} < 2^{-9K_n}]<Ce^{-cK_n}$.
Thus with Lemma \ref{lem:finite-threshold}, \eqref{eq:item2bd}, and \eqref{eq:xtaun4}, we get  
$$
\P\Big[X^{*,(v)}_{\tau_n(a)}>\frac{(\log(n))^8}{n} \Big]\geq \P\Big[w_{b_n,\tau_n(a)}(W_{b_n}^v, W_{\tau_n(a)}^v)>2^{9K_n}\frac{(\log(n))^8}{n}  \Big] -o(1) > \Phi(-a)-C\delta,
$$
which, together with \eqref{eq:xtau5}, leads to 
\[
  \E\left[\norm{X^{(v )}_{\tau_n(a)}-\one/n}_1\right]
  > 2\Phi(-a)-C\delta.
\]
Combining this with the upper bound \eqref{eq:xtau2} we see that
\[
  \E\Big[
    \Big(
      2\Phi(-a)-\norm{X^{(v)}_{\tau_n(a)}-\one/n}_1
    \Big)\vee 0
  \Big]
  < C\delta^{1/6}.
\]
By Markov's inequality, this proves the lower bound \eqref{eq:xtau6}, and thus the conclusion follows.\end{proof}

\subsection{Good graph hypotheses for random regular graphs}\label{ssec:satisfy_RRG}
We now show that with high probability, the $\delta$-good conditions (1), (3), and (4) of \Cref{def:good-graph} hold for a random $d$-regular graph. 

\begin{prop}\label{prop:entropic-heat-kernel}
Let $n\to\infty$ through integers for which $nd$ is even, and $G_n=(V_n, E_n)$ be a uniformly random simple connected $d$-regular graph on $n$ vertices. 
Take any fixed $\delta>0$. As $n\to \infty$,  with probability $>1-o(1)$, $G_n$ satisfies (1), (3), and (4) in Definition \ref{def:good-graph}.
\end{prop}

\begin{proof}
The uniform local tree-excess estimate in \cite[Lemma~2.1]{LS} gives  condition (1).  
Friedman's theorem \cite{Friedman} gives condition (3) for a suitable fixed
$\lambda_{\rm sp}>0$. It remains to verify condition (4).

Using \cite[Theorem 1]{LS}, adapted to the continuous time setting via a Poissonisation, we have that for every fixed $\epsilon>0$, as $n\to\infty$,
\begin{equation}    \label{eq:TV-mixingP}
  \max_{v\in V_n} \| p_{t_\epsilon}^n(v,\cdot) - \one/n \|_1
  \xrightarrow{\Pp}0,
\end{equation}
where  $p_s^n$ denotes the  transition probability of random walk on $G_n$ from time 0 to $s$, with edge jump rate $1/(2d)$, and $t_{\epsilon}:= \frac{(1+\epsilon) \log(n)}{\kappa_{\rm RW}}$.
In particular, this implies that \eqref{eq:TV-mixing} holds with probability $1-o(1)$.

It remains to get \eqref{eq:entropic-mixing}. Take $\epsilon>0$, small enough depending on all other parameters. From \eqref{eq:TV-mixingP}, it suffices to consider $T_1<s<t_\epsilon$.
Assume that 
\eqref{eq:entropic-mixing} fails for some $v$ and $s$. Then the set 
$A_v:=\{u\in V_n:p_s^n(v,u)>e^{-\kappa_2s}\}$ 
satisfies that
$
\P[W_s^v\in A_v ]\geq \delta$ and $|A_v|\leq e^{\kappa_2 s}$.
By comparison with random walks on the infinite $d$-regular tree, we see that
$$
\P[\dist(W^v_{t_{2\epsilon}}, A_v)<((d-2)/(2d)+\epsilon)(t_{2\epsilon} -s)   ]\ge \delta/2 ,
$$
for $n$ large enough (depending on $\delta,\epsilon$). Note that
$$
|\{u\in V_n: \dist(u, A_v)<((d-2)/(2d)+\epsilon)(t_{2\epsilon} -s)\}|
\leq e^{\kappa_2 s} d(d-1)^{((d-2)/(2d) +\epsilon)(t_{2\epsilon} -s) }
=o(n),
$$
which implies $$\|p_{t_{2\epsilon}}^n(v,\cdot)-\one/n\|_1> \delta - o(1).$$ 
As the probability of the above event tends to 0 as $n\to \infty$ by \eqref{eq:TV-mixingP}, we conclude that
\eqref{eq:entropic-mixing} also holds with probability $1-o(1)$. This concludes the verification of condition (4) with probability $1-o(1)$.
\end{proof}

\begin{proof}[Proof of \Cref{thm:main-rdreg,thm:intro-tree-clt}]
\Cref{thm:main-rdreg} follows from applying \Cref{prop:any-root-profile,prop:entropic-heat-kernel} with fixed $\delta>0$ and $n\to \infty$, and then sending $\delta\to 0$.
\Cref{thm:intro-tree-clt} then follows from \Cref{thm:main-rdreg,thm:tree-clt}.
\end{proof}

\section{Bounds on $\kappa_*$ and $t_{n,\epsilon}$ for regular graphs}
\label{sec:general-lower}
In this section we establish Theorem \ref{thm:intro-general-lower}, which gives lower bounds on the $L^1$ mixing time on general regular graphs, and finish proving the $\kappa_*<\log(2)$ and $\lim_{d\to\infty}\kappa_*=\log(2)$ parts of Theorem \ref{thm:kappasbounds}. 

\subsection{Proof of Theorem \ref{thm:intro-general-lower}}
Recall the definition of $\epsilon$-mixing time $t_{n,\epsilon}$ in \eqref{eq:tn_epsilon}. 
We now prove the bound \eqref{eq:general_mixing_lower} for general degrees $d_n$ and the bound \eqref{eq:bounded_mixing_lower} for fixed degree $d$, respectively.

For the graph $G_n=(V_n, E_n)$, we let $\Pi^n$ be the rate $1/d_n$ Poisson point process on $E_n\times \R$. 
For any vertex $v\in V_n$, consider the random walk $W^{(n,v)}$ on $G_n$ generated by $\Pi^n$ in the time interval $[0,\infty)$, started at $v$.

\smallskip

\noindent\textbf{Regular graphs of general degrees $d_n$.} 
For any $v\in V_n$, $t>0$, and $\Lambda_t^{\rm inc}\sim \Poi(t)$, we have
\[
    -\log_2 X^{(n,v)}_t(W_t^{(n,v)}) \le\Lambda_t^{\rm inc},
\]
in the sense of stochastic domination.
For any $0<\delta<1/2$, we then have 
\begin{equation*} 
    \P[X^{(n,v)}_{(1-\delta)\log_2(n)}(W_{(1-\delta)\log_2(n)}^{(n,v)}) \geq 
 n^{-(1-\delta/2)}
    ]\geq 1-\P[\Lambda_{(1-\delta)\log_2(n)}^{\rm inc} \geq (1-\delta/2)\log_2(n)] > 1- n^{-c\delta^2}, 
\end{equation*}
which (as in, e.g., \eqref{eq:mtheta_comp}) implies
\begin{equation}\label{eq:general_mixing_lb2}
     \E \left[ \norm{X^{(n,v)}_{(1-\delta)\log_2 n  }-\one/n}_1 \right] \geq 2(1-n^{-c\delta^2})-2n^{-\delta/2}.
\end{equation}
Hence for any $0<\epsilon<1$ and any $0<\delta<1/2$, 
\[
    \liminf_{n\to \infty} \frac{t_{n,\epsilon}}{(1-\delta)\log_2 n }\geq 1.
\]
Taking $\delta\to 0$ leads to \eqref{eq:general_mixing_lower}.

\smallskip

\noindent\textbf{Regular graphs of fixed degree $d_n\equiv d$.}
For any $v\in V_n$ and $i\in \mathbb{N}$,  consider the event $\mathcal E_i$:
\begin{enumerate}
\item $|\{(e, s)\in \Pi^n: i\le s \le i+1/4, W_s^{(n,v)}\in e\}|=1$;
\item for the unique $(e_*, s_*)\in \Pi^n$ such that $W^{(n,v)}_{s_*}\in e_*$ and $i\le s_*\le i+1/4$, let $s_+=\min\{s>s_*: \exists (e,s)\in \Pi^n, e\cap e_*\neq\emptyset\}$; then we have $s_+\le i+1$ and $\{e\in E_n: (e, s_+)\in \Pi^n\}=\{e_*\}$.
\end{enumerate}
In words, $W^{(n,v)}$ is incident to exactly one Poisson point $(e_*, s_*)$ during $[i, i+1/4]$, and the first Poisson point incident to $e_*$ after $s_*$ is before $i+1$ and still at $e_*$.

Under $\mathcal E_i$, one can see that the Poisson clock ring at $s_+$ is redundant in terms of repeated averaging: the mass distribution is unchanged. 
Let $\Lambda_t^{\rm red}$ denote the number of $i\in\N$ with $i+1\le t$ such that $\mathcal E_i$ occurs. 
Then for any deterministic starting point $v$,  we have that $ -\log_2 X^{(n,v)}_t(W_t^{(n,v)})  + \Lambda_t^{\rm red}$ is stochastically dominated by $\Poi(t)$.
On the other hand, since $d_n=d$, there exists a constant $p_d>0$ such that
$$
\P[\mathcal E_i\,|\, \Pi^n\cap (E_n\times [0,i]), W^{(n,v)}|_{[0,i]}]>p_d, \quad \forall i\in \mathbb{N}.
$$
Thus the Chernoff bound gives 
\[
    \P\big[  \Lambda_t^{\rm red}\ge p_dt/2 \big] > 1-e^{-ct},
\]
which implies
\[
    \P\big[-\log_2 X^{(n,v)}_t(W_t^{(n,v)}) \leq (1-p_d/3) t\big]> 1-e^{-ct}. 
\]
Set $\underline c_d=(1-p_d/4)^{-1}>1$.
Analogously to \eqref{eq:general_mixing_lb2},  we get
\[
       \E \left[ \norm{X^{(n,v)}_{\underline c_d \log_2(n)  }-\one/n}_1 \right] \geq 2-n^{-c},
\]
which further implies \eqref{eq:bounded_mixing_lower}. \qed

\subsection{Completion of the proof of Theorem \ref{thm:kappasbounds}}
The inequality $0<\kappa_*<\kappa_{\rm RW}$ has been established in 
Proposition \ref{prop:kappa-star}. The bound $\kappa_*<\log(2)$ follows from \eqref{eq:bounded_mixing_lower} of Theorem \ref{thm:intro-general-lower}. 
It remains to prove that
\[
    \liminf_{d\to \infty} \kappa_*(d)\geq \log(2).
\]
Here and below, we write $\kappa_*(d)$ for $\kappa_*$ to emphasize its dependence on $d$.
All the implicit constants and $C,c>0$ below are independent of $d$.

We work on the infinite $d$-regular tree $\T=(V,E)$ with root $o$, equipped with a rate $1/d$ Poisson point process $\Pi^\T$.
Write $X_t(v)=w_{0,t}(o,v;\Pi^\T)$ for every $t\ge 0$ and $v\in V$, and let $W$ be the random walk generated by $\Pi^\T$ on the time interval $[0,\infty)$, starting from $o$.

The task of lower bounding $\kappa_*(d)$ is essentially upper bounding $X_t(W_t)$, according to \eqref{eq:kappa-star}.
For this, we need to control the weights of vertices around the walk $W$. 
For any $i\in\Z_{\ge 0}$, we call $i$ a \emph{good starting time}
if $W_i$ has at least $d-d^{3/4}$ children with zero weight at time $i$,
and denote this event by $\mathcal E_{\rm good}(i;d)$.
We call $[i,i+1)$ a \emph{good interval} if $i$ is a good
starting time, and for every $(\{W_{t-}, v\}, t)\in \Pi^\T$, $t\in [i,i+1)$, we have $X_{t-}(v)=0$.
We also denote this event by $\mathcal E_{\rm good}(i,i+1;d)$.

We proceed in two steps:  we first upper bound  $\P[\mathcal E^c_{\rm good}(i;d)]$; we then turn this into an upper bound for $\P(\mathcal E^c_{\rm good}(i,i+1;d))$, and use  \eqref{eq:forward-ratio-one} to handle the decay of $X_t(W_t)$ in intervals that are not good.

\smallskip

\noindent\textbf{Good starting time.}
Define a sequence of times $\tau_0=0<\tau_1<\tau_2<\cdots$
as follows: 
\[
\tau_k = \min\{t>\tau_{k-1}: X_{t-}(W_t)=0; \, W_s\neq W_{t-}, \,\forall s\ge t\},
\]
i.e., $\tau_k$ is the first time after $\tau_{k-1}$ at which
$W$ jumps to a vertex with zero weight immediately before the jump,
and never traverses the edge again.
We note that there is a renewal structure: whenever $W$ jumps to a vertex with zero weight immediately before the jump, which is a stopping time, with probability $\frac{d-2}{d-1}$ and independent of the history, $W$ will never traverse the edge again.
In particular, $\tau_k-\tau_{k-1}$ for all $k\ge 2$ are i.i.d.~random variables.

We next prove an upper tail bound for $\tau_2-\tau_1$.
Given $\tau_1$ and $W_{\tau_1}$, the law of $W|_{[\tau_1, \infty)}$ is that of a continuous-time random walk on $\T$ with edge jump rate $1/(2d)$, conditioned never to visit the parent of $W_{\tau_1}$.
For $t>0$, let
\[
  \widetilde{\mathcal E}_t
  :=\{\dist(W_{\tau_1+t}, W_{\tau_1})\ge t/4,\ U_t< t/16,\
          \dist(W_{\tau_1+s}, W_{\tau_1})>t/8,\;\forall s>t\},
\]
where $U_t=|\{s\in [\tau_1, \tau_1+t]: \dist(W_s, W_{\tau_1}) = \dist(W_{s-}, W_{\tau_1})-1\}|$ is the number of backward steps in $[\tau_1, \tau_1+t]$.
For $d$ large enough, we have
\begin{equation}   \label{eq:Etcupbd}
    \P[\widetilde{\mathcal E}_t^c]<Ce^{-ct}.
\end{equation} 

Now, conditional on $\tau_1$ and $W$ such that $\widetilde{\mathcal E}_t$ holds, there exist at least $\lfloor t/16 \rfloor$ vertices on the geodesic connecting $W_{\tau_1}$ and $W_{\tau_1+t}$, such that the edge joining the vertex and its parent is traversed exactly once by $W$ after $\tau_1$.
Then we have $\tau_2\le \tau_1 + t$ unless, for each one of these $\lfloor t/16 \rfloor$ vertices $u$, we have $X_{\hat{s}-}(u)>0$, where $\hat{s}=\min\{s: W_s=u\}$.
Using \Cref{lem:palm-law}, we see that the probability of this event, conditional on $\tau_1$ and $W$, is bounded by the probability that the sum of $\lfloor t/16\rfloor$ independent exponential random variables with mean $d$ is at most $t$.
For $d$ large enough this is $<Ce^{-ct}$.
Thus with \eqref{eq:Etcupbd}, we conclude that, for $d$ large enough, 
\begin{equation}   \label{eq:tau12bd}
\P[\tau_2-\tau_1>t]<Ce^{-ct}.
\end{equation} 
On the other hand, $\tau_2$ is not smaller than the time of the next Poisson point incident to $W_{\tau_1}$. Hence $\tau_2-\tau_1$ stochastically dominates an exponential random variable with mean 1.

Denote by $\T(u)$ the subtree rooted at $u\in V$.
With the renewal theorem, we now have
\begin{equation}\label{eq:renew}
  \begin{split}
  &\lim_{i\to\infty}\P[\mathcal E_{\rm good}(i;d)^c]\\
  &\quad=
  \frac{1}{\E[\tau_2-\tau_1]}
  \E\!\left[
    \int_{\tau_1}^{\tau_2}
      \mathds{1}[ |\{v\in \T(W_s): X_s(v)>0, \dist(W_s, v)=1\}| >d^{3/4}-1]
       \,\dd s\right].
  \end{split}
\end{equation}
Then consider the following event $\mathcal E_*$: 
for $s_*=\min\{s\ge \tau_1: W_s\neq W_{\tau_1}\}$, we have 
\begin{itemize}
    \item $X_{s_*-}(W_{s_*})=0$, 
    \item $s_*\le \tau_1 +\sqrt{d}$, 
    \item $|\{ v \in \T(W_{\tau_1}) : X_{s_*}(v)>0, \dist(W_{\tau_1}, v)=1\}|\le d^{3/4}-1$,
    \item $W_s\neq W_{\tau_1}$ for any $s\ge s_*$.
\end{itemize}
Under $\mathcal E_*$, we must have $\tau_2=s_*$, and the integral in \eqref{eq:renew} would be zero.
On the other hand, Poisson estimates imply that $\P[\mathcal E_*^c]<C/d$.
Together with \eqref{eq:tau12bd} and \eqref{eq:renew}, the Cauchy-Schwarz inequality yields
\begin{equation}   \label{eq:good_event_lb}
  \lim_{i\to\infty}\P[\mathcal E_{\rm good}(i;d)^c]
  \le
  \frac{\sqrt{\E[(\tau_2-\tau_1)^2]\,
                       \P[\mathcal E_*^c]}}
       {\E[\tau_2-\tau_1]}
  < \frac{C}{\sqrt d}.
\end{equation}

\noindent\textbf{Lower bound on entropy increase.}
Let \[\tau=\min\{s>i: \exists v\in V, \text{ such that }X_{s-}(v)>0, (\{W_{s-},v\}, s)\in \Pi^\T\}.\]
Conditional on $\mathcal E_{\rm good}(i;d)$, unless
\[
|\{(e, u)\in \Pi^\T: u\in [i, i+1), W_u \in e\}| > d^{1/8},
\]
 we have $|\{v\in V: \dist(v,W_s)=1, X_s(v)>0\}|\le d^{3/4}+d^{1/8}$ for each $s\in [i, (i+1)\wedge\tau)$.
Thus by Poisson estimates, we have 
\[
\P[\mathcal E_{\rm good}(i,i+1;d)|\mathcal E_{\rm good}(i;d) ] > 1-Ce^{-cd^{1/8}} - d^{1/8}\cdot\frac{d^{3/4}+d^{1/8}}{d} > 1-Cd^{-1/8}.
\] 
Combined with  \eqref{eq:good_event_lb}, this leads to
\begin{equation}\label{eq:limif_good_int}
 \liminf_{i\to \infty}   \P[\mathcal E_{\rm good} (i,i+1;d)]> 1-Cd^{-1/8}.
\end{equation} 
On $\mathcal E_{\rm good}(i,i+1;d)$, we have the exact identity
\begin{equation}    \label{eq:Xiiplog2}
\log X_i(W_i)-\log X_{i+1}(W_{i+1})=  \log(2)|\{(e, s)\in\Pi^\T: s\in [i, i+1), W_s \in e\}|.    
\end{equation}
We next consider an interval $[i, i+1)$ that is not good.
The second bound in \eqref{eq:forward-ratio-one} and the inequality $(\log(x))^2 \leq x$ for $x\geq 1$ imply
\[
    \E\left[ 
    \left(\log \frac{X_{i+1}(W_{i+1}) }{X_i(W_i)} \right)^2 \mathds{1}[ X_{i+1}(W_{i+1})\geq X_i(W_i) ]\right]\leq \E \frac{X_{i+1}(W_{i+1}) }{X_i(W_i)}\le 1.
\]
We can then lower bound $ \E\left[  \log X_i(W_i)-\log X_{i+1}(W_{i+1})  \right] $ by
\begin{equation}\label{eq:ent_inc_larged}
\begin{split}
     &  \E\left[ \log \frac{X_i(W_i)}{X_{i+1}(W_{i+1}) }
     \mathds{1}[\mathcal E_{\rm good}(i,i+1;d)] \right]
     + \E\left[ \log \frac{X_i(W_i)}{X_{i+1}(W_{i+1}) }\mathds{1}[
     \mathcal E_{\rm good}(i,i+1;d)^c ] \right]
\\
\geq  & 
\log (2) \E[|\{(e, s)\in\Pi^\T: s\in [i, i+1), W_s \in e\}| \mathds{1}[\mathcal E_{\rm good}(i,i+1;d)] ]\\
&-    \E\left[ \log \frac{X_{i+1}(W_{i+1}) }{X_i(W_i)}  \mathds{1}[ X_{i+1}(W_{i+1})\geq X_i(W_i),\mathcal E_{\rm good}(i,i+1;d)^c ]\right]\\
\geq & \log(2) -C \sqrt{\P[\mathcal E_{\rm good}(i,i+1;d)^c]},
\end{split}
\end{equation}
where we used \eqref{eq:Xiiplog2} for the first inequality, and Cauchy-Schwarz for the second inequality.

Taking $\liminf$ in \eqref{eq:ent_inc_larged} and applying \eqref{eq:limif_good_int}, we get
\begin{equation}
    \liminf_{i\to \infty}   \E[ \log X_i(W_i)-\log X_{i+1}(W_{i+1})] \geq \log(2) - Cd^{-1/16}.
\end{equation}
Summing these increments for $i$ from 0 to $n-1$ and using \eqref{eq:kappa-star} gives
$\kappa_*(d)\geq\log(2)-Cd^{-1/16}$.
This proves $\kappa_*(d)\to\log(2)$ as $d\to\infty$. \qed

\bibliographystyle{alpha}
\bibliography{bibliography}

\end{document}